\documentclass[11pt]{article}
\usepackage[margin=1.25in]{geometry}  
\usepackage{geometry}                
\usepackage{amsmath}
\usepackage{mathrsfs}
\usepackage[bookmarksopen, bookmarksnumbered]{hyperref}
\usepackage[nottoc,numbib]{tocbibind}
\hypersetup{
	colorlinks=true,
	linkcolor=blue,
	citecolor=blue,
	filecolor=magenta,      
	urlcolor=cyan,
	linktoc=page
}
\usepackage{mathtools}
\usepackage{amsthm}
\usepackage{amssymb}
\usepackage{enumitem}
\usepackage{placeins}
\usepackage{graphicx}
\usepackage{epstopdf}
\usepackage{float}
\usepackage{caption}
\usepackage{subcaption}
\usepackage{xcolor}
\usepackage{ stmaryrd }

\theoremstyle{plain}
\newtheorem{theorem}{Theorem}[section]
\newtheorem{corollary}[theorem]{Corollary}
\newtheorem{conj}[theorem]{Conjecture} 
\newtheorem{prop}[theorem]{Proposition}
\newtheorem{lemma}[theorem]{Lemma}

\theoremstyle{definition}
\newtheorem{definition}[theorem]{Definition}
\newtheorem{remark}[theorem]{Remark}
\newenvironment{claimproof}[1]{\par\noindent\underline{Proof:}\space#1}{\hfill $\blacksquare$}

\newcommand{\Z}{\mathbb{Z}}
\newcommand{\Q}{\mathbb{Q}}
\newcommand{\C}{\mathbb{C}}
\newcommand{\N}{\mathbb{N}}
\newcommand{\R}{\mathbb{R}}

\newcommand{\E}{\mathbb{E}}

\renewcommand*{\P}{\mathbb{P}}

\newcommand{\indic}{\mathbf{1}}

\newcommand{\sset}{\subset}
\newcommand{\lf}{\left}
\newcommand{\rg}{\right}

\newcommand{\mathand}{\;\text{and}\;}

\newcommand{\ga}{\gamma}
\newcommand{\Ga}{\Gamma}
\newcommand{\ep}{\epsilon}

\newcommand{\de}{\delta}
\newcommand{\be}{\beta}

\newcommand{\sig}{\sigma}

\newcommand{\la}{\lambda}
\newcommand{\al}{\alpha}
\newcommand{\Om}{\Omega}

\newcommand{\del}{\partial}
\newcommand{\Rd}{\mathbb{R}^4_\uparrow}

\newcommand{\cA}{\mathcal{A}}

\newcommand{\cC}{\mathcal{C}}

\newcommand{\cF}{\mathcal{F}}

\newcommand{\cI}{\mathcal{I}}

\newcommand{\cL}{\mathcal{L}}

\newcommand{\cS}{\mathcal{S}}

\newcommand{\cX}{\mathcal{X}}

\newcommand{\NI}{\operatorname{NI}}

\newcommand{\fA}{\mathfrak{A}}

\newcommand{\fD}{\mathfrak{D}}
\newcommand{\fE}{\mathfrak{E}}

\newcommand{\fG}{\mathfrak{G}}
\newcommand{\fg}{\mathfrak{g}}

\newcommand{\fL}{\mathfrak{L}}

\newcommand{\fX}{\mathfrak{X}}

\newcommand{\sC}{\mathscr{C}}

\newcommand{\sG}{\mathscr{G}}

\newcommand{\close}[1]{\mkern 1.5mu\overline{\mkern-1.5mu#1\mkern-1.5mu}\mkern 1.5mu}

\usepackage{MnSymbol}

\DeclareMathOperator*{\argmax}{arg\,max}

\newcommand{\eqd}{\stackrel{d}{=}}

\newcommand{\X}{\times}

\newcommand{\cvgdown}{\downarrow}

\newenvironment{claim}[1]{\par\noindent\underline{Claim:}\space#1}{}
\newenvironment{claim1}[1]{\par\noindent\underline{Claim 1:}\space#1}{}
\newenvironment{claim2}[1]{\par\noindent\underline{Claim 2:}\space#1}{}

\newcommand{\smin}{\setminus}

\newcommand{\bv}{\mathbf{v}}
\newcommand{\bw}{\mathbf{w}}
\newcommand{\bx}{\mathbf{x}}
\newcommand{\by}{\mathbf{y}}
\newcommand{\bz}{\mathbf{z}}

\newcommand{\br}{\mathbf{r}}
\newcommand{\bu}{\mathbf{u}}
\newcommand{\ba}{\mathbf{a}}
\newcommand{\bb}{\mathbf{b}}
\newcommand{\bc}{\mathbf{c}}

\newcommand{\II}[1]{\llbracket #1 \rrbracket}

\title{The $27$ geodesic networks in the directed landscape}
\author{Duncan Dauvergne}

\begin{document}
	\maketitle
	
\begin{abstract}
The directed landscape is a random directed metric on the plane that arises as the scaling limit of classical metric models in the KPZ universality class. Typical pairs of points in the directed landscape are connected by a unique geodesic. However, there are exceptional pairs of points connected by more complicated geodesic networks. We show that up to isomorphism there are exactly $27$ geodesic networks that appear in the directed landscape, and find Hausdorff dimensions in a scaling-adapted metric on $\mathbb R^4_\uparrow$ for the sets of endpoints of each of these networks. 
\end{abstract}
\begin{figure}[htb]
	\centering
	\begin{subfigure}[t]{1.6cm}
		\centering
		\includegraphics[height=1.9cm]{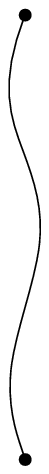}
	\end{subfigure}
	\begin{subfigure}[t]{1.6cm}
		\centering
		\includegraphics[height=1.9cm]{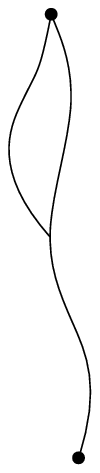}
	\end{subfigure}
	\begin{subfigure}[t]{1.6cm}
		\centering
		\includegraphics[height=1.9cm]{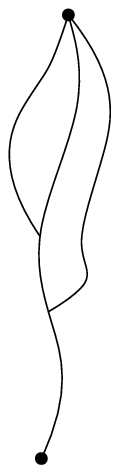}	
	\end{subfigure}
	\begin{subfigure}[t]{1.6cm}
		\centering
		\includegraphics[height=1.9cm]{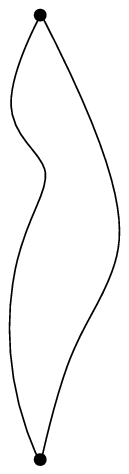}
	\end{subfigure}
	\begin{subfigure}[t]{1.6cm}
		\centering
		\includegraphics[height=1.9cm]{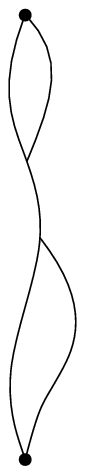}	
	\end{subfigure}
	\begin{subfigure}[t]{1.6cm}
		\centering
		\includegraphics[height=1.9cm]{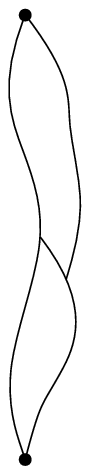}	
	\end{subfigure}	
	\begin{subfigure}[t]{1.6cm}
		\centering
		\includegraphics[height=1.9cm]{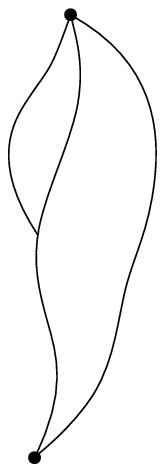}		
	\end{subfigure}
\\
\vspace{1em}
	\begin{subfigure}[t]{1.6cm}
		\centering
		\includegraphics[height=1.9cm]{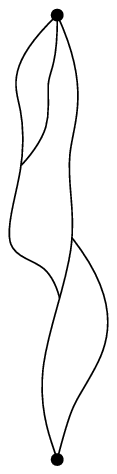}		
	\end{subfigure}
	\begin{subfigure}[t]{1.6cm}
		\centering
		\includegraphics[height=1.9cm]{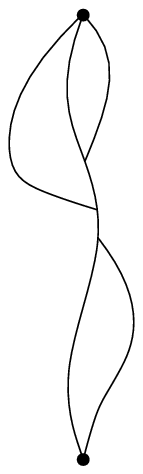}		
	\end{subfigure}
	\begin{subfigure}[t]{1.6cm}
		\centering
		\includegraphics[height=1.9cm]{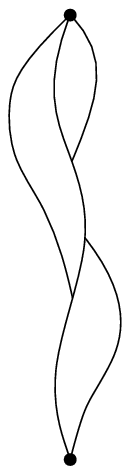}	
	\end{subfigure}
	\begin{subfigure}[t]{1.6cm}
		\centering
		\includegraphics[height=1.9cm]{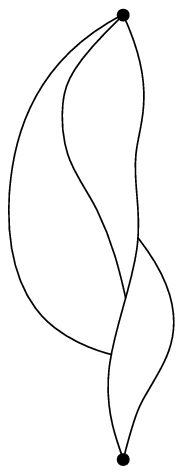}	
	\end{subfigure}
	\begin{subfigure}[t]{1.6cm}
		\centering
		\includegraphics[height=1.9cm]{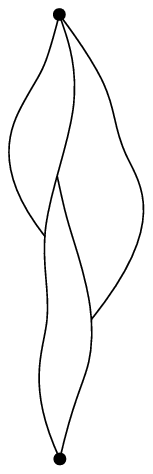}		
	\end{subfigure}
	\begin{subfigure}[t]{1.6cm}
		\centering
		\includegraphics[height=1.9cm]{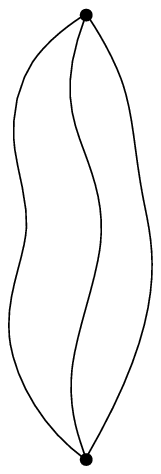}
	\end{subfigure}
	\begin{subfigure}[t]{1.6cm}
		\centering
		\includegraphics[height=1.9cm]{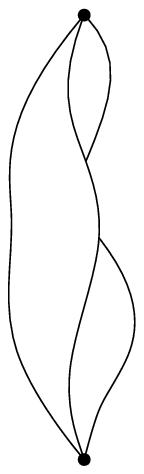}	
	\end{subfigure}
\\
\vspace{1em}
	\begin{subfigure}[t]{1.6cm}
		\centering
		\includegraphics[height=1.9cm]{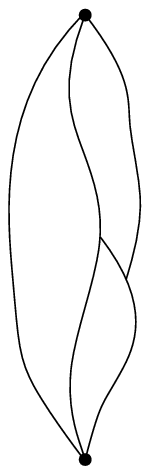}		
	\end{subfigure}
	\begin{subfigure}[t]{1.6cm}
		\centering
		\includegraphics[height=1.9cm]{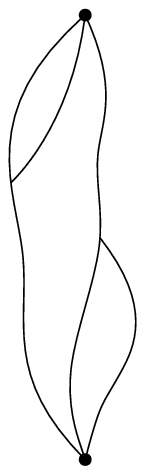}	
	\end{subfigure}
	\begin{subfigure}[t]{1.6cm}
		\centering
		\includegraphics[height=1.9cm]{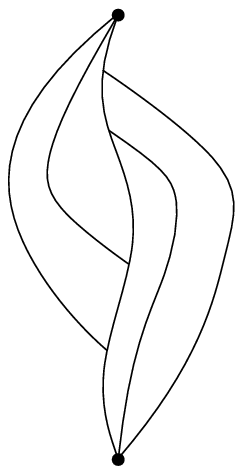}	
	\end{subfigure}
	\begin{subfigure}[t]{1.6cm}
		\centering
		\includegraphics[height=1.9cm]{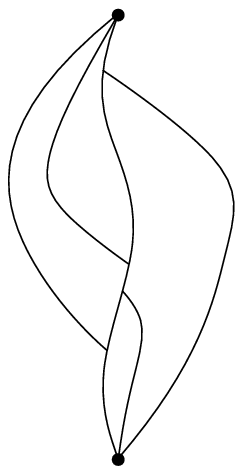}	
	\end{subfigure}
	\begin{subfigure}[t]{1.6cm}
		\centering
		\includegraphics[height=1.9cm]{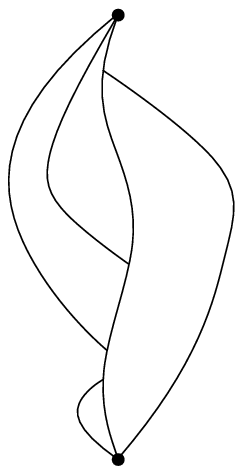}		
	\end{subfigure}
	\begin{subfigure}[t]{1.6cm}
		\centering
		\includegraphics[height=1.9cm]{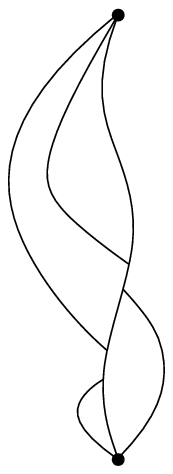}	
	\end{subfigure}
	\begin{subfigure}[t]{1.6cm}
		\centering
		\includegraphics[height=1.9cm]{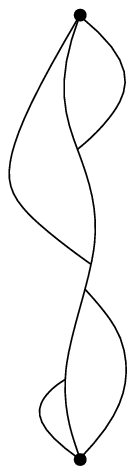}	
	\end{subfigure}	
\\
\vspace{1em}
	\begin{subfigure}[t]{1.6cm}
		\centering
		\includegraphics[height=1.9cm]{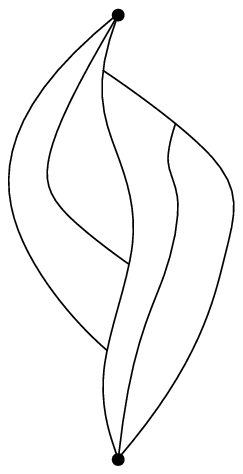}	
	\end{subfigure}
	\begin{subfigure}[t]{1.6cm}
		\centering
		\includegraphics[height=1.9cm]{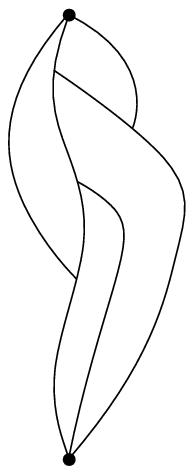}	
	\end{subfigure}
	\begin{subfigure}[t]{1.6cm}
		\centering
		\includegraphics[height=1.9cm]{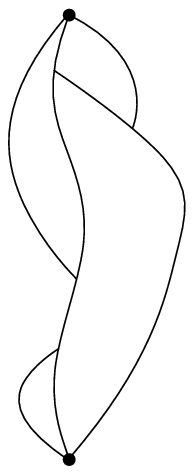}
	\end{subfigure}
	\centering
	\begin{subfigure}[t]{1.6cm}
		\centering
		\includegraphics[height=1.9cm]{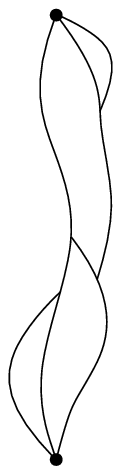}
	\end{subfigure}
	\begin{subfigure}[t]{1.6cm}
		\centering
		\includegraphics[height=1.9cm]{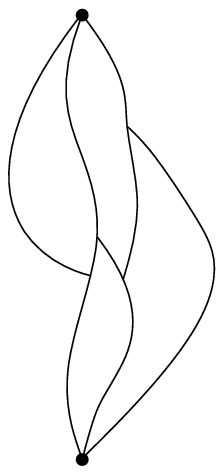}		
	\end{subfigure}
	\begin{subfigure}[t]{1.6cm}
		\centering
		\includegraphics[height=1.9cm]{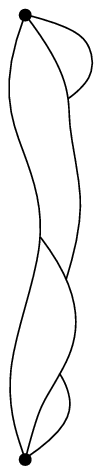}		
	\end{subfigure}
\end{figure}
\FloatBarrier
\newpage
\setcounter{tocdepth}{1}
\tableofcontents

\section{Introduction}

The Kardar-Parisi-Zhang (KPZ) universality class is a collection of one-dimensional random growth models and two-dimensional random metrics that are all expected to exhibit the same limiting behaviour under rescaling. The past thirty years has seen a period of intense research on this class thanks to the discovery of a handful of exactly solvable models, see Section \ref{S:background} for background.

The directed landscape $\cL$ was constructed in \cite{DOV} as the full scaling limit of one particular two-dimensional random metric (Brownian last passage percolation) in the KPZ class. More recently, $\cL$ has been shown to be the scaling limit of a handful of other exactly solvable random metrics \cite{dauvergne2021scaling} and of the KPZ equation \cite{wu2023kpz}, see also \cite{virag2020heat, quastel2020convergence}. It is conjectured to be the full scaling limit of all KPZ models. 

The directed landscape is a random continuous function from $$
\R^4_\uparrow = \{u = (p; q) = (x, s; y, t) \in \R^4 : s < t\}
$$
to $\R$. The value $\cL(p; q)$ is best thought of as a distance between two points $p$ and $q$ in the space-time plane. Here $x, y$ are spatial coordinates and $s, t$ are time coordinates. Unlike with an ordinary metric, $\cL$ is not symmetric, does not assign distances to every pair of points in the plane, and may take negative values. It also satisfies the triangle inequality backwards:
\begin{equation}
\label{E:triangle}
\cL(p; r) \ge \cL(p; q) + \cL(q; r) \qquad \mbox{ for all }(p; r), (p; q), (q; r) \in \Rd.
\end{equation}
Just as in true metric spaces, we can define path lengths in $\cL$, see \cite[Section 12]{DOV}. First, for a continuous function $\pi:[s, t]\mapsto \R$,  referred to as a path, let $\bar \pi(r) = (\pi(r), r)$. We define the \textbf{length} of $\pi$ by
\begin{equation}
\label{E:length}
\|\pi\|_\cL=\inf_{k\in \N}\inf_{s=t_0<t_1<\ldots<t_k=t}\sum_{i=1}^k\cL(\bar \pi(t_{i-1});\bar \pi(t_i))\,.
\end{equation}
By the triangle inequality \eqref{E:triangle}
, $\|\pi\|_\cL \le \cL(\bar \pi (s); \bar \pi (t))$ for any path $\pi$.
We say that $\pi$ is a {\bf geodesic} from $\bar \pi(s)$ to $\bar \pi(t)$ if
$\|\pi\|_\cL=\cL(\bar \pi (s); \bar \pi (t))\,.$ We write 
$$
\fg \pi = \{\bar \pi(r) : r \in [s, t]\} \sset \R^2
$$ for the image of $\pi$ in the space-time plane. By \cite{DOV}, Theorem $12.1$, for any $(p; q)\in \R^4_\uparrow$, almost surely there exists a unique geodesic $\pi$ from $p$ to $q$. However, there are exceptional points $(p; q) \in \R^4_\uparrow$ where multiple geodesics go from $p$ to $q$. The goal of this paper is to fully understand what happens at these points.

For every $(p; q) = (x, s; y, t)\in \Rd$ define the \textbf{geodesic network} from $p$ to $q$ by
$$
\Ga(u) = \bigcup \{\fg \pi : \pi \text{ is a geodesic from } p \text{ to } q\}.
$$
It turns out that any pair $(p; q) \in \Rd$ is connected by only finitely many geodesics. Moreover, from the definition of path length it is not difficult to check that any path $\tau:[s, t] \to \R$ with $\fg \tau \sset \Ga(p; q)$ is a geodesic from $p$ to $q$. Together these two points imply that $\Ga(p; q)$ has the topological structure of a finite directed graph. Indeed, for $(p; q) \in \Rd$ we can define the \textbf{network graph} $G(p; q)$ of its geodesic network with the following two rules:
\begin{itemize}
	\item The vertices $V$ of $G(p; q)$ consist of the points $p, q$ and all points $v = (z, r) \in \Ga(p; q)$ such that there exist two geodesics $\pi, \tau$ from $p$ to $q$ with $\pi(r) = \tau(r) = z$ but $\pi|_{[r-\ep, r + \ep]} \ne \tau|_{[r-\ep, r + \ep]}$ for all $\ep > 0$.
	\item For two vertices $v_1 = (z_1, r_1), v_2 = (z_2, r_2) \in V$ with $r_1 < r_2$, we add one copy of the directed edge $(v_1, v_2)$ to the edge set $E$ for every distinct path $\pi:[r_1, r_2] \to \R$ with $\fg \pi \sset \Ga(p; q)$ and $\fg \pi \cap V = \{v_1, v_2\}$.
\end{itemize}
See Figure \ref{fig:network} for an example of $\Ga(p; q), G(p; q)$. 
Now, for a finite directed graph $G$, let $G^T$ denote the transpose of $G$, given by reversing the direction of all the edges of $G$. Since the directed landscape has a time-reversal symmetry and is ergodic (see Section \ref{S:prelim}), a graph $G$ will appear as a geodesic network in $\cL$ if and only if $G^T$ appears. Because of this, we will only distinguish directed graphs $G$ \textbf{up to transpose isomorphism} $\sim$ (e.g. $G \sim H$ if either $G$ or $G^T$ is isomorphic to $H$ as a directed graph), see Figure \ref{fig:network} for an example. 
Our first main theorem classifies all the network graphs that appear in $\cL$.
\begin{figure}
	\centering
	\includegraphics[scale=0.9]{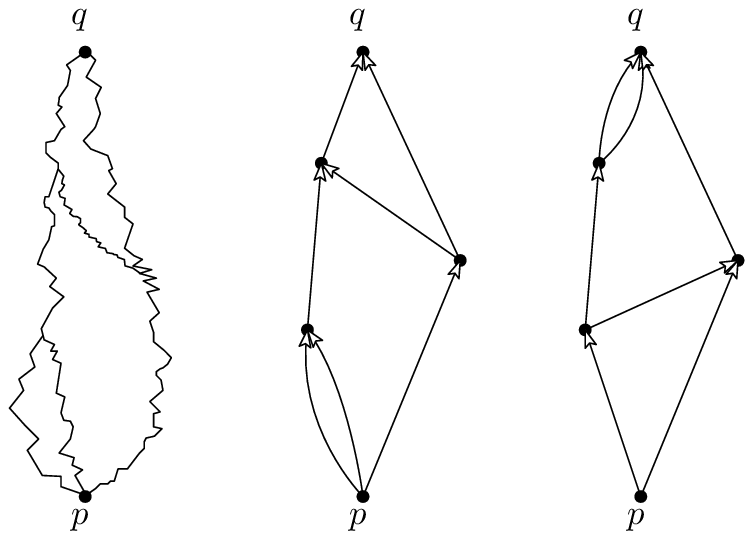}
	\caption{An example of the procedure for taking the geodesic network $\Ga(p; q)$ to the directed graph $G(p; q)$, and an example of two transpose-isomorphic networks.}
	\label{fig:network}
\end{figure}

\begin{theorem}
	\label{T:geodesic-networks}
	Almost surely, the following holds in the directed landscape $\cL$. A directed graph $G = (V, E)$ is the network graph of some point $(p;q) \in \Rd$ if and only if the following conditions hold:
	\begin{enumerate}[label=\arabic*.]
		\item $G$ is finite, planar, and loop-free.
		\item There is exactly one \textbf{source} vertex $p \in V$ (i.e. a vertex with no incoming edges) and exactly one \textbf{sink} vertex $q \in V$ (i.e. a vertex with no outgoing edges).
		\item The induced graph on $V \setminus \{p, q\}$ has no undirected cycles (i.e. the underlying undirected graph is a forest).
		\item Every vertex in $V \smin \{p, q\}$ has degree $3$.
		\item The vertices $p, q$ have degree $1, 2,$ or $3$.
	\end{enumerate}
	There are exactly $27$ graphs that satisfy these rules, enumerated in Section \ref{S:botany}. We call the set of these networks $\mathscr{G}$.
\end{theorem}  
Now, for every $G \in \sG$, define
$$
N_\cL(G) = \{(p; q) \in \Rd: G(p; q) = G\}.
$$
As mentioned previously, for a typical point $(p; q) \in \Rd$ there is unique geodesic from $p$ to $q$. All other networks are rare. Our next theorem aims to classify exactly how rare these networks are. We focus on three properties:
\begin{itemize}[nosep]
	\item The size of $N_\cL(G)$.
	\item Whether $N_\cL(G)$ is dense.
	\item The Hausdorff dimension $\dim_{1:2:3}(N_\cL(G))$ in the 1:2:3-metric on $\Rd$
	given by
	$$
	d_{1:2:3}((x, s; y, t), (x', s'; y', t')) = |t-t'|^{1/3} + |s-s'|^{1/3} + |x- x'|^{1/2} + |y-y'|^{1/2}.
	$$ 
\end{itemize}
The 1:2:3-metric is set up to take into account the natural KPZ scaling associated with the directed landscape. Indeed, since $\cL$ is H\"older-$1/2^-$ in $x, y$ and H\"older-$1/3^-$ in $s, t$, defining Hausdorff dimensions under the $d_{1:2:3}$-metric is a good proxy for defining Hausdorff dimensions using the `metric' $\cL$. It does not make sense to define Hausdorff dimensions using $\cL$ directly since it is not a true metric (i.e. it may take on positive or negative values).
 We do not address the issue of the Euclidean dimensions of each of the sets $N_\cL(G)$ in this paper.

\begin{theorem}
	\label{T:Hausdorff-dimension}
	Almost surely, the following assertions hold. Let $G = (V, E) \in \sG$, let $p, q$ denote the unique source and sink vertices in $G$, and define
	$$
	d(G) := 12 - \frac{|V| + \deg^2(p) + \deg^2(q)}{2}.
	$$
	Then:
	\begin{enumerate}[label=\arabic*.]
		\item $
		\dim_{1:2:3}(N_\cL(G)) = d(G).
		$
		\item If $d(G) = 0$, then $N_\cL(G)$ is countably infinite.
		\item $N_\cL(G)$ is dense if there is an edge $e \in E$ such that $(V, E \smin \{e\})$ is disconnected as an undirected graph. Otherwise, $N_\cL(G)$ is nowhere dense.
	\end{enumerate}
\end{theorem}

It is not difficult to check that $0 \le d(G) \le 10 = \dim_{1:2:3}(\Rd)$ for all $G \in \sG$, see Lemma \ref{L:graph-theory}. 
Note that we have chosen not the retain the planar structure when classifying geodesic networks in Theorems \ref{T:geodesic-networks} and \ref{T:Hausdorff-dimension}. A version of these theorems that takes into account the planar structure of networks also follows from the methods in this paper.

\begin{remark}
	\label{R:brownian-map-i}
	The question of which geodesic networks are dense was studied in \cite{angel2017stability} for the Brownian map and in \cite{gwynne2021geodesic} for general Liouville quantum gravity for $\ga \in (0, 2)$. It turns out that the types of dense geodesic networks are exactly the same in the directed landscape and for these models. There are six such networks, enumerated in Figure \ref{fig:normal-networks}. These networks are uniquely determined by $\deg(p), \deg(q)$, and the fact that they can be disconnected by a single edge. In \cite{angel2017stability} and \cite{gwynne2021geodesic}, these networks are referred to as $(\deg(p), \deg(q))$-\textbf{normal networks}.
	The geometric similarities between the directed landscape and the Liouville quantum gravity also go beyond the classification of dense networks; we discuss this more in Section \ref{S:speculation}.
\end{remark}

We can pass from results about geodesic networks to the simpler question of counting points $(p; q) \in \Rd$ connected by exactly $k$ geodesics.

\begin{corollary}
	\label{C:Tpq}
	For every $(p; q) \in \Rd$, let $T(p; q)$ be the number of distinct geodesics from $p$ to $q$. Then almost surely, the following statements hold.
	\begin{enumerate}[label=\arabic*.]
		\item $T(p; q) \le 9$ for all $(p; q) \in \Rd$.
		\item We have 
		\begin{align*}
		&\dim_{1:2:3}T^{-1}(1) = 10, &&\dim_{1:2:3}T^{-1}(2) = 8, &&&\dim_{1:2:3}T^{-1}(3) = 6, \\
		&\dim_{1:2:3}T^{-1}(4) = 6, &&\dim_{1:2:3}T^{-1}(5) = 3, &&&\dim_{1:2:3}T^{-1}(6) = 3
		\end{align*}
		 and $T^{-1}(7), T^{-1}(8), T^{-1}(9)$ are all countable.
	\end{enumerate}
\end{corollary}
\begin{figure}
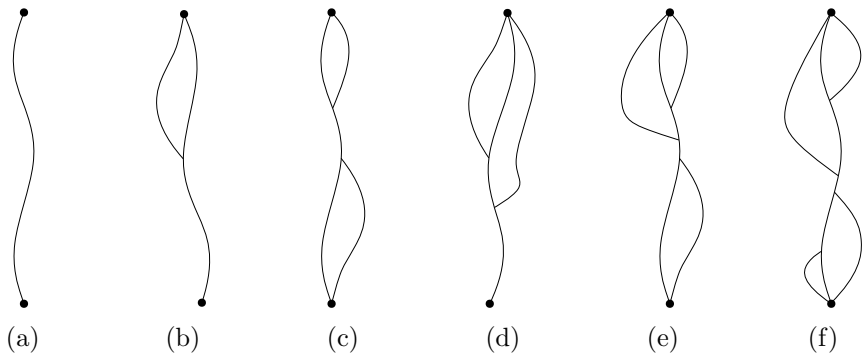

	\centering
	\begin{subfigure}[t]{2cm}
		\centering
		\includegraphics[height=4cm]{G11}
		\caption{}
	\end{subfigure}
	\begin{subfigure}[t]{2cm}
		\centering
		\includegraphics[height=4cm]{G21}
		\caption{}
	\end{subfigure}
	\begin{subfigure}[t]{2cm}
		\centering
		\includegraphics[height=4cm]{G22.4.1}
		\caption{}
	\end{subfigure}
	\begin{subfigure}[t]{2cm}
		\centering
		\includegraphics[height=4cm]{G31}
		\caption{}
	\end{subfigure}
	\begin{subfigure}[t]{2cm}
		\centering
		\includegraphics[height=4cm]{G32.5.b1}
		\caption{}
	\end{subfigure}
	\begin{subfigure}[t]{2cm}
		\centering
		\includegraphics[height=4cm]{G33.6.p1.5}
		\caption{}
	\end{subfigure}
	\caption{The six dense geodesic networks.}
	\label{fig:normal-networks}
\end{figure}

\subsection{Geometric ideas in the proof}
\label{S:geometric-ideas}

Here we outline the main geometric ideas that underlie the proofs of Theorems \ref{T:geodesic-networks} and \ref{T:Hausdorff-dimension}. We do not discuss the technical probabilistic and analytic methods that go into the proofs in order to avoid complicating the geometric picture.

\textbf{Coalescent geometry.} \qquad A key feature of the directed landscape and other models of continuum random geometry that distinguishes these objects from their Euclidean counterparts is geodesic coalescence: if two geodesics are close together, then they will typically coalesce in order to take advantage of certain favourable geodesic highways. Because of this, geodesics in the directed landscape are best thought of as driving routes in a city; geodesics between two points $p$ and $q$ will follow major geodesic routes away the endpoints $p, q$. The study of coalescence in the context of the KPZ universality class goes back at least to the work of Newman and coauthors on first passage percolation (e.g. see \cite{newman1995surface} and references therein) and is by this point very well understood in the directed landscape, e.g. see \cite{bates2019hausdorff, dauvergne2020three, busani2022stationary, bhatia2023duality}.

One useful perspective on coalescent geometry is the \textit{overlap topology} on geodesics. We say that a sequence of geodesics $\pi_n$ converges in overlap to $\pi:[s, t] \to \R$ if the endpoints of $\pi_n$ converge to the endpoints of $\pi$, and if the set 
$
O^o(\pi_n, \pi) := \{r \in (s, t) : \pi_n(r) = \pi(r)\}
$
is an interval whose endpoints converge to $s, t$.  
Many useful properties of coalescent geometry can be expressed in terms of overlap. While many results about coalescence are used throughout the paper, for the purposes of this proof sketch, we will only appeal to the following two consequences of coalescent geometry recorded in Lemma \ref{L:rightmost-geods}:
\begin{itemize}
	\item \textbf{Rational overlap approximation}. \quad $\Ga(\Rd \cap \Q^4)$ is dense in $\Ga(\Rd)$. In other words, any geodesic can be approximated in overlap by a sequence of geodesics with rational endpoints.
	\item \textbf{No geodesic bubbles}. \qquad 
	When combined with the fact that typical (i.e. rational) points in $\Rd$ are only connected by a single geodesic, rational overlap approximation implies that there are no geodesic bubbles, i.e. if $p, q$ are any points on the interior of a common geodesic then $p$ and $q$ are connected by a unique geodesic.
\end{itemize}
Just before this paper was posted, Bhatia \cite{bhatia2023duality} developed independent proofs of these two bullet points. His proof methods are quite different from ours.

\textbf{Geodesic stars.} \qquad Consider a collection of geodesics $\pi = (\pi_1, \dots \pi_k)$ defined on a common interval $[s, t]$. We call $\pi$ a \textbf{forwards $k$-geodesic star} from a point $p$ to $(\bx, t)$ where 
$$
\bx \in \R^k_< := \{\bx \in \R^k : x_1 < \dots < x_k\}
$$
if $\bar \pi_i(s) = p, \bar \pi_i(t) = (x_i, t)$ for all $i$, and if $\pi_i(r) \ne \pi_j(r)$ for all $i \ne j, r \in (s, t]$. Similarly, $\pi$ is a \textbf{reverse $k$-geodesic star} from $(\by, s)$ to $q$ if $\bar \pi_i(t) = q, \bar \pi_i(s) = (y_i, s)$ for all $i$ and if $\pi_i(r) \ne \pi_j(r)$ for all $i \ne j, r \in [s, t)$. We call $p$ a \textbf{$k$-star point} if it is the starting point of a forwards $k$-geodesic star or the ending point of a reverse geodesic $k$-star. 

The study of coalescent geometry in $\cL$ implies that for any $(p; q) \in \Rd$, if we look at the set of all geodesics from $p$ to $q$ and remove a small ball around $p$ and $q$, then we must be left with a typical geodesic forest. Therefore if $G(p; q)$ is an exceptional graph (i.e. not the trivial two-point graph) it is because either $p$ is a $k$-star point for some $k \ge 2$ or $q$ is an $\ell$-star point for some $\ell \ge 2$ (or both). Because of this, in order to classify geodesic networks it is logical to first classify star points. To this end, in Section \ref{S:weighted-stars} we prove the following theorem. 

\begin{theorem}
	\label{T:star-computation}
Let $\operatorname{Star}_k \sset \R^2$ be the set of all $k$-star points. Then almost surely we have
$$
\dim_{1:2:3}(\operatorname{Star}_1) = 5, \quad \dim_{1:2:3}(\operatorname{Star}_2) = 4, \quad \dim_{1:2:3}(\operatorname{Star}_3) = 2, \quad \operatorname{Star}_4 = \emptyset.
$$
Here $\dim_{1:2:3}$ is Hausdorff dimension with respect to the metric
$$
d_{1:2:3}((x, s), (y, t)) = |t-s|^{1/3} + |y-x|^{1/2}.
$$	 
\end{theorem}

\textbf{Cut and paste: from geodesic stars to geodesic networks.} \qquad To move from understanding geodesic stars to understanding geodesic networks, the basic idea is as follows. 
Fix a candidate network graph $G$ whose source and sink $p$ and $q$ have degrees $k$ and $\ell$, respectively, and suppose $(p; q) = (x, s; y, t) \in N_\cL(G)$. Let $\ep > 0$ be small enough so that no branching in $\Ga(p; q)$ occurs in the intervals $[s, s+\ep]$ and $[t-\ep, t]$. Cut the geodesic network $\Ga(p; q)$ into three pieces defined on the strips
$$
\R \X [s, s+\ep], \qquad \R \X [s+\ep, t - \ep], \qquad \R \X [t-\ep, t]
$$
and let $\bx \in \R^k_<, \by \in \R^\ell_<$ be such that $\{(x_1, s+ \ep), \dots, (x_k, s + \ep)\} = \Ga(p; q) \cap (\R \X \{s+\ep\})$ and $\{(y_1, t-\ep), \dots, (y_\ell, t - \ep)\} = \Ga(p; q) \cap (\R \X \{t-\ep\})$. This splitting of $\Ga(p; q)$ leaves us with:
\begin{itemize}[nosep]
	\item A forwards geodesic $k$-star $\pi = (\pi_1, \dots, \pi_k)$ from $p$ to $(\bx, s+\ep)$.
	\item An interior network $I$ consisting of all geodesics from $(x_i, s+ \ep)$ to $(y_j, t - \ep)$ whenever those two points lie on a common geodesic from $p$ to $q$. This network has a graph structure given by taking $G$, splitting its source vertex into $k$ degree-one sources and splitting its sink vertex into $\ell$ degree-one sinks. 
	\item A reverse geodesic $\ell$-star $\tau = (\tau_1, \dots, \tau_\ell)$ from $(\by, t - \ep)$ to $q$.
\end{itemize}

Therefore the problem of determining the dimension and size of $N_\cL(G)$ is essentially the problem of determining the dimension and size of the set of all compatible triples $\pi, I, \tau$. 

To think about the size of this set, we first consider the possibilities for $I$. By the rational overlap approximation discussed above there are only countably many possibilities for $I$ so for all intents and purposes we can focus on determining geodesics whose networks use a single fixed $I$. Moreover, $I$ is homeomorphic to a network obtained by shifting $x_i, y_j$ to nearby rational points. Because of this, we expect that $I$ is generic. Here this means that $I$ can only be a forest whose interior branch points all have degree $3$; heuristically, higher degree branch points should be `lower dimensional' and so will not appear in a generic $I$. The network $I$ is a forest since there are no geodesic bubbles.

On the other hand, the choice of $I$ does put a series of constraints on the lengths of the paths $\pi_i, \tau_j$. Indeed, for all paths through $G(p; q)$ to have the same length, we require that all of the lengths
$$
\cL(p; x_i, s + \ep) + \cL(x_i, s + \ep; t-\ep, y_j) + \cL(t-\ep, y_j; q)  = \|\pi_i\|_\cL + \cL(x_i, s + \ep; t-\ep, y_j) + \|\tau_j\|_\cL
$$
are equal whenever $(x_i, s + \ep)$ and $(y_j, t - \ep)$ are connected in $\Ga(p; q) \cap (\R \X [s+\ep, t - \ep])$. Letting $|F^*|$ denote the number of \textit{bounded faces} in $\Ga(p; q)$, this puts $|F^*|$ linearly independent linear constraints on the weight vector
\begin{equation}
\label{E:pi-tau-length}
(\cL(p; x_1, s + \ep), \dots, \cL(p; x_k, s + \ep), \cL(t-\ep, y_1; q), \dots, \cL(t-\ep, y_\ell; q)),
\end{equation}
see Figure \ref{fig:cut-and-paste} for an example.
 Heuristically, if the distribution of the length vector \eqref{E:pi-tau-length} for pairs of uniformly chosen geodesic stars is uniformly spread out in $\R^{k + \ell}$, then the set of pairs of geodesic stars that satisfy these linear constraints should have dimension
\begin{equation}
\label{E:123-Star-k}
\dim_{1:2:3}(\operatorname{Star}_k) + \dim_{1:2:3}(\operatorname{Star}_\ell) - |F^*|.
\end{equation}
In other words, the dimension should reduce by one for every linearly independent linear constraint. The fact that the dimension should decrease by exactly one is coming from the fact that $d_{1:2:3}$ is a proxy for the metric dimension, so the function mapping $(p; q)$ to the length vector in \eqref{E:pi-tau-length} is H\"older-$(1-\ep)$ for all $\ep > 0$.

By using Euler's formula for planar graphs, and the fact that all interior vertices of $G$ have degree $3$, we can rewrite \eqref{E:123-Star-k} as 
\begin{equation}
\label{E:dimension-fomula}
\dim_{1:2:3}(\operatorname{Star}_k) + \dim_{1:2:3}(\operatorname{Star}_\ell) + 2 - \frac{|V| + k + \ell}{2},
\end{equation}
which by Theorem \ref{T:star-computation} equals $d(G)$ in Theorem \ref{T:Hausdorff-dimension}.

\begin{figure}
	\centering
	\includegraphics[scale=1.5]{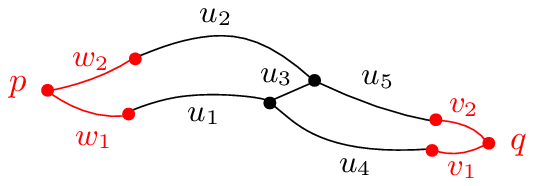}
	\caption{An example of a potential configuration of geodesics, with geodesic stars around the endpoints indicated in red, and lengths labelled by $w_i, u_i, v_i$. For all paths here to be geodesics, we require that $w_2 + u_2 = w_1 + u_1 + u_3$, and $u_3 + u_5 + v_2 = u_4 + v_1$. In other words, we have two linear restrictions on the weights of the stars at $p$ and $q$: one for each of the bounded faces of the candidate network.}
	\label{fig:cut-and-paste}
\end{figure}

\FloatBarrier

Making the above discussion rigorous is delicate. Indeed, we first need a more technical version of Theorem \ref{T:star-computation} that also takes into account the weight vector \eqref{E:pi-tau-length}. We prove this more technical statement (see Propositions \ref{P:star-lower-bd} and \ref{P:weighted-stars-upper-bound}) while showing Theorem \ref{T:star-computation}. Next, we need to understand the sizes of candidate triples $(\pi, I, \tau)$. There are many nuances here, but the mains ones involve proving sufficient independence of the weights of $\pi, I, \tau$, and giving methods to `cut' and `paste' that allow us to go back and forth between networks and triples $\pi, I, \tau$.

\subsection{The Brownian map and Liouville quantum gravity}
\label{S:speculation}

As mentioned in Remark \ref{R:brownian-map-i}, there are random surface models that exhibit geodesic geometry that is remarkably similar to that of the directed landscape, and many theorems for $\cL$ have parallels in these models. Geometric proof techniques for studying geodesics on random surfaces are also of use in studying the directed landscape. For example, we have taken inspiration from the topological ideas in \cite{gwynne2021geodesic} in proving bounds in Sections \ref{S:lower-bd}, \ref{S:upper-size} and the central role coalescence plays in our proofs is similar to the central role it plays in \cite{gall2010geodesics, angel2017stability, miller2020geodesics, gwynne2021geodesic}. In this section we discuss the similarities between the directed landscape and random surfaces with the goal of developing (conjectural) versions of Theorems \ref{T:geodesic-networks} and \ref{T:Hausdorff-dimension} for these models.

The most well-understood random surface model is the Brownian map. The Brownian map is a random metric space that arises as the scaling limit of uniformly chosen random planar maps \cite{le2013uniqueness, miermont2013brownian}, see \cite{le2019brownian} and references therein for more background.
The Brownian map is part of a family of random surfaces known as Liouville quantum gravity (LQG) indexed by a parameter $\ga$ (the Brownian map corresponds to $\ga = \sqrt{8/3}$).
Just as the Brownian map arises as the scaling limit of uniformly chosen chosen random planar maps, general LQG models are expected to arise as the scaling limit of random planar maps weighted by a partition function for a statistical mechanics model.

Roughly speaking, Liouville quantum gravity on a domain $U \sset \C$ indexed by $\ga \in (0, 2)$ is the random Riemannian manifold with Riemannian metric tensor $e^{\ga h(x + i y)} (d x^2 + dy^2)$, where $h$ is a Gaussian free field.  This is not well-defined since $h$ is not a function. However, it is possible to make sense of LQG rigorously using a regularization procedure, see \cite{berestycki2015introduction, gwynne2019random, ding2021introduction} and references therein for more background. 

 Geodesics in the Brownian map and in LQG for $\ga \in (0, 2)$ have been studied in great detail. Notable result include the classification of all geodesic networks starting at a fixed point in the Brownian map \cite{gall2010geodesics}; the classification of dense geodesic networks in the Brownian map \cite{angel2017stability} and in LQG for $\ga \in (0, 2)$ \cite{gwynne2021geodesic}; and the determination of the Hausdorff dimension of pairs of points in the Brownian map connected by $k$ geodesics for all $1 \le k \le 9$ (as in the directed landscape, no pairs are connected by more than $9$ geodesics) \cite{miller2020geodesics}. The paper \cite{miller2020geodesics} also goes much further, giving Hausdorff dimension upper bounds that are expected to be sharp and finding restrictions on the types of geodesic networks that appear in the Brownian map.

Our belief is that there are only a few qualitative mechanisms driving the geodesic geometry in these models: coalescence (see \cite{gall2010geodesics, gwynne2020confluence,  miller2020geodesics, gwynne2021geodesic} for results in the Brownian map and LQG), planarity, and rough independence of the metric in different spatial regions. Because of this, it is reasonable to expect that much of the geometric picture driving Theorems \ref{T:geodesic-networks} and \ref{T:Hausdorff-dimension} goes through for these models.

 Indeed, if a graph $G$ is a network graph in $\ga$-LQG then we expect that $G$ should satisfy points $1, 2, 3,$ and $4$ of Theorem \ref{T:geodesic-networks}. 
 Point $5$ should not necessarily hold since this point relates to the possible dimensions of star points, a quantitative feature which is not the same in the directed landscape and in these other models.
 The analogue of $\dim_{1:2:3}$ in these models is simply the Hausdorff dimension in the LQG metric, which we will denote by $\dim$. We expect that the formulas \eqref{E:123-Star-k}/\eqref{E:dimension-fomula} should also hold in these models with $\dim$ in place of $\dim_{1:2:3}$.
 
 For the special case of the Brownian map, the star sets are known to have dimensions
 \begin{equation}
 \label{E:LGMQ}
\dim(\operatorname{Star_k}) = 5 -k
 \end{equation}
 for $k = 1, 2, 3, 4, 5$. No geodesic $6$-star points exist and the question of whether geodesic $5$-star points exist is open. The upper bound here was proven in \cite{miller2020geodesics} and the lower bound in \cite{le2022geodesic}. This suggests the following formula for the Hausdorff dimension of points whose network graph is $G = (V, E)$:
\begin{equation}
\label{E:Brownian-map-dimension}
10 - k - \ell - |F^*| = 12 - \frac{3(k + \ell) + |V|}{2}.
\end{equation}
Formula \eqref{E:Brownian-map-dimension} agrees with the upper bound on the Hausdorff dimension obtained in \cite[Theorem 1.5]{miller2020geodesics}. Note that their formula is expressed in terms of a different graph invariant (splitting points). 

Formula \eqref{E:Brownian-map-dimension} is non-negative whenever $k, \ell \le 3$ and so we expect that each of the $27$ network graphs in the directed landscape also occur the Brownian map. We consider the cases when $k \vee \ell \ge 4$. Without loss of generality, assume $k \ge 4$. 
Formula \eqref{E:Brownian-map-dimension} is negative whenever $k + \ell \ge 8$, so no geodesic networks should exist with these parameters in the Brownian map. Next, if we think about the argument leading to \eqref{E:123-Star-k} carefully, we can observe if a bounded face of $G$ is incident to the source vertex $p$ but not the sink $q$, then the linear constraint from that face only involves the weight of the $k$-geodesic star at $p$. Therefore if $|F^*_p|$ is the number of bounded faces of $G$ incident to $p$ but not $q$, we should require $ \dim(\operatorname{Star_k}) - |F^*_p| \ge 0$ for the network to exist. Since $|F^*_p| \ge k - \ell$, by \eqref{E:LGMQ} a geodesic network $G$ will not exist if $k \ge 4, k - \ell \ge 2$. 

Therefore the only new geodesic networks in the Brownian map will have $k=4, \ell = 3$, and $|V| \le 3$. Up to transpose-isomorphism there is exactly one such network, exhibited in Figure \ref{fig:speculative-networks}(a). In summary, we expect that there are exactly $27 + 1 = 28$ network graphs that appear in the Brownian map: the $27$ network graphs that appear in the directed landscape, along with the single network graph exhibited in Figure \ref{fig:speculative-networks}(a).

\begin{figure}
	\centering
	\begin{subfigure}[t]{2cm}
		\centering
		\includegraphics[height=4cm]{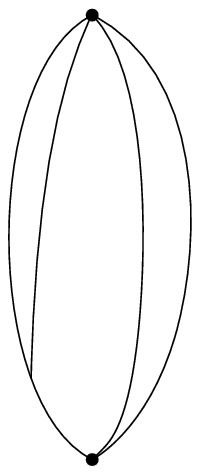}
		\caption{}
	\end{subfigure}
	\begin{subfigure}[t]{4cm}
		\centering
		\includegraphics[height=4cm]{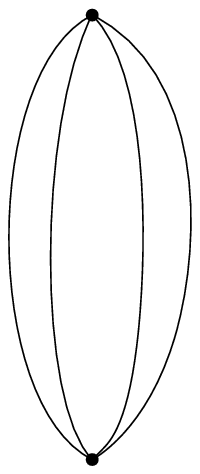}
		\caption{}
	\end{subfigure}
	\caption{The two geodesic networks that we expect to appear in certain models of Liouville quantum gravity, but do not appear in the directed landscape.}
	\label{fig:speculative-networks}
\end{figure}

For general $\ga$-LQG with $\ga \in (0, 2)$, the picture becomes hazier because star dimensions are not known. However, the set of points joined to the origin by at least three geodesics is known to be countable \cite{gwynne2021geodesic}. This fact and a coalescence heuristic strongly suggests that the set of points whose network is that of Figure \ref{fig:normal-networks}(f) is countable as well. By formula \eqref{E:dimension-fomula}, we can then conclude that $\dim (\operatorname{Star_3})$ should equal $2$, as is the case for the Brownian map and the directed landscape. Now, it is reasonable to expect that for $k \ge 4$ we have
 \begin{equation}
 \label{E:decreasing-assn}
\dim (\operatorname{Star_k}) < \dim (\operatorname{Star_3}) = 2 < \dim (\operatorname{Star_2}) < \dim (\operatorname{Star_1}).
 \end{equation}
 Under assumption \eqref{E:decreasing-assn}, formula \eqref{E:dimension-fomula} will be non-negative for all of the $27$ networks in $\cL$, suggesting that each of these networks appear in $\ga$-LQG. Next, as in the Brownian map case we do not expect to see networks with
 $
 \dim(\operatorname{Star_k}) - |F_p^*| \le \dim(\operatorname{Star_k}) - (k - \ell) < 0.
 $
 By \eqref{E:decreasing-assn}, this rules out network graphs $G$ with with $k \ge 4, \ell \le 2$. For a graph $G = (V, E)$ with $k \ge 4, \ell \ge 3$, under \eqref{E:decreasing-assn} and \eqref{E:dimension-fomula}, $G$ cannot be a geodesic network unless $|V| + k + \ell < 12$, leaving us with only possibilities when $k = 4, \ell = 3, |V| \le 4$ or $k = \ell = 4, |V| \le 3$. The two possible networks satisfying these constraints are listed in Figure \ref{fig:speculative-networks}. In summary, we expect the following.
\begin{conj}
Fix $\ga \in (0, 2)$. There are either $27, 28,$ or $29$ network graphs that appear in $\ga$-LQG: the $27$ networks that appear in the directed landscape, along with the network in Figure \ref{fig:speculative-networks}(a) if $\dim(\operatorname{Star}_4) \ge 1$ and the network in Figure \ref{fig:speculative-networks}(b) if $\dim(\operatorname{Star}_4) \ge 3/2$. Geodesic networks occur with dimensions given by formula \eqref{E:123-Star-k}/\eqref{E:dimension-fomula}.
\end{conj}

\subsection{Previous work}
\label{S:background}

Here we give a brief review of the literature on the KPZ universality class and the directed landscape, focusing on papers most closely related to the present work. For a gentle introduction to KPZ suitable for a newcomer to the area, see \cite{romik2015surprising} or the introductory articles \cite{corwin2016kardar, ganguly2021random}. Review articles and books that go into more depth include \cite{ferrari2010random, quastel2011introduction,weiss2017reflected, zygouras2018some}.

Most of the work on the KPZ universality class over the past twenty-five years has been devoted to studying exact formulas for integrable models. This work gives exact formulas for many marginals of the directed landscape.
For example, the Baik-Deift-Johansson theorem \cite{baik1999distribution} on the length of the longest increasing subsequence identified the one-point distribution $\cL(0,0; 0, 1)$ as a GUE Tracy-Widom random variable, see also \cite{johansson2000shape}. Pr\"ahofer and Spohn \cite{prahofer2002scale} proved convergence of the multi-layer PNG droplet to the parabolic Airy line ensemble $\fA = \{\fA_i : \R \to \R, i \in \N\}$. The function $\fA_1$ is the marginal $y \mapsto \mathcal L(0, 0; y, 1)$.
More recently, Johansson and Rahman \cite{johansson2019multi} and Liu \cite{liu2019multi} independently found formulas for the joint distribution of $\{\cL(p; q_i) : i \in \{1, \dots, k\}\}$, building on work from \cite{johansson2017two, johansson2018two,baik2017multi}. Matetski, Quastel, and Remenik \cite{matetski2016kpz} derived a formula for the distribution of $h_t(y) = \max_{x \in \R} g(x) + \cL(x,0; y, t)$ for a fixed function $g$. The function $h_t$ is the KPZ fixed point: it is the scaling limit for random growth models in the KPZ universality class.

To answer many questions about the KPZ universality class, it is useful to have probabilistic descriptions of the limit objects. In this direction, Corwin and Hammond \cite{CH} rigorously showed that the parabolic Airy line ensemble $\fA$ is a system of nonintersecting curves by showing that $\fA$ satisfies a certain Brownian Gibbs resampling property, see Section \ref{S:ALE-AS}. This property turns out to be extremely useful for understanding probabilistic questions about $\cL$. Indeed, \cite{CH} used it to quickly show that $\fA_1$ is locally Brownian. Moreover, it can be used to give strong regularity properties and two-point estimates for $\cL$ e.g. see \cite{hammond2016brownian, hammond2017modulus, calvert2019brownian, dauvergne2021bulk, dauvergne2023wiener}. The Airy sheet $\cS(x, y) = \cL(x, 0; y, 1)$ was built from $\fA$ in \cite{DOV} by taking the limit of a certain combinatorial isometry, with a probabilistic analysis of $\fA$ used as a key input. We require probabilistic analysis of $\fA$ based on results from \cite{dauvergne2021disjoint, dauvergne2023wiener} at various points in this paper, most crucially in the proof of the upper bound in Theorem \ref{T:star-computation}.

Since \cite{DOV} first appeared online, there has been a period of intense research on the structure of geodesics in $\cL$. The paper \cite{DOV} proved basic properties of geodesics: almost sure uniqueness and H\"older continuity. Bates, Ganguly, and Hammond \cite{bates2019hausdorff} showed that the Euclidean Hausdorff dimension of the set of pairs $x, y \in \R$ for which there are two disjoint geodesics from $(x, 0)$ to $(y, 1)$ equals $1/2$. Ganguly and Zhang \cite{ganguly2022fractal} showed that the Euclidean Hausdorff dimension of the set of points $p \in \R^2$ from which there is a geodesic $2$-star to a fixed pair of points $(x_1, 0), (x_2, 0)$ has Hausdorff dimension $5/3$. Bhatia \cite{bhatia2023duality} showed that the set of points $p \in \R^2$ that have $3$ distinct infinite geodesics in a fixed direction has Euclidean Hausdorff dimension $4/3$, and the set of points $p \in \R^2$ that have $2$ distinct infinite geodesics in a fixed direction is countable. See also \cite{rahman2021infinite, dauvergne2020three, seppalainen2021global, busani2022stationary, liu2022one, bhatia2022atypical} for more work on geodesics and \cite{gangulyfractal} for a general review of fractal geometry in the directed landscape.

\textbf{Organization of the paper.} \qquad Section \ref{S:prelim} is a preliminary section. Section \ref{S:coalescent} contains basics about coalescent geometry in $\cL$, and contains the weak estimate that every pair of points is connected by at most $121$ geodesics. Section \ref{S:weighted-stars} is devoted to understanding geodesic stars. Section \ref{S:at-most-27} allows us to conclude that there are at most $27$ geodesic networks in $\cL$; proof-wise, it is fairly similar to the upper bounds in Section \ref{S:weighted-stars}. Section \ref{S:networks-dimensions} makes rigorous the cut and paste heuristics leading to formula \eqref{E:123-Star-k}. This section completes the proof of Theorem \ref{T:geodesic-networks} and parts $1$ and $2$ of Theorem \ref{T:Hausdorff-dimension}. Section \ref{S:dense} proves Theorem \ref{T:Hausdorff-dimension}.3. The final Section \ref{S:botany} is a bookkeeping section where we enumerate the $27$ networks (with justification). It contains figures of the $27$ networks that may interest the reader. 

\textbf{Notational conventions.} \qquad All Brownian motions and Brownian bridges throughout the paper will have diffusion parameter $2$, as is the convention for KPZ limits. Large constants are typically denoted $c, c', c_1$ etc. and may change from line to line within proofs. We use a subscript $a, k$ etc. (i.e. $c_{a, k}$) if a constant depends on parameters $a, k$ etc. We write $\II{m, n}$ for the integer interval $\{m, \dots, n\}$ and write $\cC^k(X)$ for the space of continuous functions from a topological space $X$ to $\R^k$ with the topology of uniform-on-compact convergence. When $k=1$ we omit it from the notation. We write $\R^k_< = \{\bx \in \R^k: x_1 < \dots < x_k\}$ and similarly use the notation $\R^k_\le, \R^k_\ge$ etc. We often write $a^k := (a, \dots, a) \in \R^k$.

\textbf{Acknowledgements.} \qquad I would like to thank B\'alint Vir\'ag for many insightful discussions. Indeed, this paper arose out of an afternoon of conjectural enumeration from mid-2019 resulting in a set of $27$ scribbled diagrams we labelled `Botany of the directed landscape'. I would also like to thank Manan Bhatia, Ofer Busani, and Ewain Gwynne for helpful comments regarding a previous version of the paper.
\section{Preliminaries}
\label{S:prelim}
In this section we recall the relevant results and background material that we will need
in this paper. This section is rather lengthy, but for understanding all but the most technical inputs in the paper (Lemma \ref{L:sparse-space-1}, Section \ref{S:upper-bounds}, Section \ref{S:at-most-27}, and Section \ref{S:sheet-tools}), only Section \ref{S:landscape} is necessary.

\subsection{The directed landscape}
\label{S:landscape}

 We start with the characterization of the directed landscape $\cL$ from \cite{DOV}. The characterization is in terms of its two-dimensional marginal $\cS(x, y) := \cL(x, 0; y, 1)$ known as the Airy sheet. The Airy sheet itself has a more technical definition that we will give in Section \ref{S:ALE-AS}. In order to define the directed landscape and for the remainder of the paper we say that $\cS_s$ is an \textbf{Airy sheet of scale $s$} if 
$$
\cS_s(x, y) \eqd s\cS(x s^{-2}, y s^{-2})
$$
jointly over all $x, y \in \R$.

\begin{definition}
	\label{D:L-unique}
	The directed landscape $\cL:\Rd \to \R$ is the unique, in law, random function in $\cC(\Rd)$ satisfying
	\begin{enumerate}[label=\arabic*.]
		\item (Airy sheet marginals) For any $t\in \R$ and $s>0$ we have 
		$$
		\mathcal{\cL}(x, t; y,t+s^3) \eqd \cS_s(x, y) 
		$$
		jointly in all $x, y$.
		\item (Independent increments) For any disjoint time intervals $\{(t_i, s_i) : i \in \{1, \dots k\}\}$, the random functions
		$
		\cL(\cdot, t_i ; \cdot, s_i), i \in \{1, \dots, k \}
		$
		are independent.
		\item (Metric composition law) Almost surely, for any $r<s<t$ and $x, y \in \R$ we have that
		$$
		\cL(x,r;y,t)=\max_{z \in \mathbb R} [\cL(x,r;z,s)+\cL(z,s;y,t)].
		$$
	\end{enumerate}
\end{definition} 

The metric composition for $\cL$ implies the reverse triangle inequality \eqref{E:triangle}. The directed landscape also satisfies the well-known \textbf{quadrangle inequality}, which states that for all $x_1 < x_2$, $y_1 < y_2$, and $s < t$ we have
\begin{equation}
\label{E:quadrangle}
\cL(x_1, s; y_2, t) + \cL(x_2, s; y_1, t) \le \cL(x_1, s; y_1, t) + \cL(x_2, s; y_2, t)
\end{equation}
The quadrangle inequality follows from a straightforward path-crossing argument, e.g. see \cite[Lemma 9.1]{DOV}.
 Note that the reverse triangle inequality \eqref{E:triangle} is an equality at all points along a geodesic. If $\pi:[s, t] \to \R$ is a geodesic, then for any partition $s = r_0 < r_1 < \dots < r_k = t$ we have
\begin{equation}
\label{E:tri-ineq-is-equality}
\cL(\bar \pi(s); \bar \pi(t)) = \sum_{i=1}^k \cL(\bar \pi(r_{i-1}); \bar \pi(r_i)).
\end{equation}
The directed landscape has invariance properties which we use throughout the paper. 

\begin{lemma} [Lemma 10.2, \cite{DOV} and Proposition 1.23, \cite{dauvergne2021scaling}]
	\label{L:invariance} We have the following equalities in distribution as random functions in $\cC(\R^4_\uparrow)$. Here $r, c \in \R$, and $q > 0$.
	\begin{itemize}[nosep]
		\item[1.] Time stationarity.
		$$
		\displaystyle
		\cL(x, t ; y, t + s) \eqd \cL(x, t + r ; y, t + s + r).
		$$
		\item[2.] Spatial stationarity.
		$$
		\cL(x, t ; y, t + s) \eqd \cL(x + c, t; y + c, t + s).
		$$
		\item[3.] Flip symmetries.
		$$
		\cL(x, s; y, t) \eqd \cL(-x, s; -y, t) \eqd \cL(y, -t; x, -s). 
		$$
		\item[4.] Skew stationarity.
		$$
		\cL(x, t ; y, t + s) \eqd \cL(x + ct, t; y + ct + sc, t + s) + s^{-1}[(x - y)^2 - (x - y - sc)^2].
		$$
		\item[5.] $1:2:3$ rescaling.
		$$
		\cL(x, t ; y, t + s) \eqd  q \cL(q^{-2} x, q^{-3}t; q^{-2} y, q^{-3}(t + s)).
		$$
	\end{itemize}
\end{lemma}

\begin{remark}
\label{R:ergodicity}
The directed landscape $\cL$ is strong mixing with respect to the time shift in Lemma \ref{L:invariance}.1 by the independence of time increments, it is strong mixing with respect to the spatial shift in Lemma \ref{L:invariance}.2 by \cite[Proposition 2.6]{dauvergne2022non}, and it is ergodic with respect to the $1:2:3$ rescaling in Lemma \ref{L:invariance}.5 by \cite[Proposition 2.5]{dauvergne2022non}. 
\end{remark}

We will also need a series of regularity results about $\cL$. The following proposition states that the directed landscape is H\"older $1/2^-$ in space and H\"older $1/3^-$ in time. This proposition is motivation for the definition of the metric $d_{1:2:3}$.

\begin{prop}[Proposition 10.5, \cite{DOV}]
	\label{P:mod-land-i}
	Let $\mathcal R(x,t;y,t+s)=\cL(x,t;y,t+s) + (x-y)^2/s$ and let $K\subset \R^4_\uparrow$ be a compact set. Then for all $u = (x, s; y, t), v = (x', s'; y' ,t') \in K$ we have
	$$
	|\mathcal R(u) - \mathcal R(v)| \le C \lf(\tau^{1/3}\log^{2/3}(\tau^{-1}+1) + \xi^{1/2}\log^{1/2}(\xi^{-1}+1) \rg),
	$$
	where
	$$
	\xi = ||(x, y) - (x', y')||_\infty \qquad \text{ and } \qquad \tau = ||(s,t) - (s', t')||_\infty.
	$$
	Here
	$C$ is a random constant depending on $K$ with $\E a^{C^{3/2}} < \infty$  for some $a > 1$.
\end{prop}

The next theorem should be thought of as a \emph{landscape shape theorem}: at first order, $\cL(x, t; y, t + s)$ is the parabola $-\tfrac{(x - y)^2}{s}$.
\begin{theorem}[Corollary 10.7, \cite{DOV}]
	\label{T:landscape-shape}
	There exists a random constant $C$ satisfying
	$$
	\P(C > m) \le ce^{-dm^{3/2}}
	$$
	for universal constants $c,d > 0$ and all $m > 0$, such that for all $u = (x, t; y, t + s) \in \R^4_\uparrow$, we have
	\begin{equation}
	\label{E:tail-1}
	\left|\cL(x, t; y, t + s) + \frac{(x - y)^2}{s} \right|\le C s^{1/3} \log^{4/3}\lf(\frac{2(\|u\|_2 + 2)^{3/2}}{s}\rg)\log^{2/3}(\|u\|_2 + 2).
	\end{equation}
\end{theorem}

We will frequently use the following corollary of Theorem \ref{T:landscape-shape} that follows the fact the triangle inequality is an equality along geodesics, \eqref{E:tri-ineq-is-equality}.

\begin{corollary}
	\label{C:geo-containment}
For every $a \in \R$, there exists a random compact set $K \sset \R^2$ such that if $\pi$ is a geodesic from $p \in [-a, a]^2$ to $q \in [-a, a]^2$ then $\fg \pi \sset K$. Moreover, all geodesics from $p \in [-a, a]^2$ to $q \in [-a, a]^2$ are H\"older-$1/2$ continuous with a common (random) H\"older constant $C_K$.
\end{corollary}

Note that geodesics are typically much smoother than H\"older-$1/2$. Indeed, the geodesic between two fixed points is H\"older-$2/3^-$ almost surely (see Proposition 12.3, \cite{DOV}). The reason for Corollary \ref{C:geo-containment} is simply to have a crude equicontinuity bound that applies to all geodesics at once.

To prove coalescence and tightness statements about geodesics we will use almost sure statements about the geodesic structure in $\cL$.

\begin{lemma}
	\label{L:regularity}
	There exists a set $\Om$ of probability one where the following hold for $\cL$:
	\begin{enumerate}[nosep, label=\arabic*.]
		\item (Theorem 12.1,\cite{DOV}). For every rational $(p; q) \in \Rd \cap \Q^4$, there is a unique $\cL$-geodesic from $p$ to $q$.
		\item The bound \eqref{E:tail-1} holds for a finite constant $C$.
		\item (Lemma 13.2, \cite{DOV}). For every $u = (p;q ) = (x, s; y, t)\in \Rd$, there exist rightmost and leftmost $\cL$-geodesics $\ga^R_u$ and $\ga^L_u$ from $p$ to $q$. That is, there exist geodesics $\ga^R(u), \ga^L(u)$ from $p$ to $q$ such that for any other geodesic $\ga'$ from $p$ to $q$, we have $\ga^L(u) \le \ga' \le \ga^R(u)$. 
For fixed $s < t$, the functions $(x, y) \mapsto \ga^L(x, s;y,t)$ and $(x, y) \mapsto \ga^R(x, s;y,t)$ are monotone increasing in $x$ and $y$.
		\item (Lemma 2.7, \cite{bates2019hausdorff}) For any real numbers $s < t, x_1 < x_2, y_1 < y_2$, and any geodesics $\ga_i$ from $(x_i, s)$ to $(y_i, t)$, $i = 1, 2$, we have $\ga_1(r) \le \ga_2(r)$ for all $r \in [s, t]$.
	\end{enumerate}
\end{lemma}

When working with paths and geodesics throughout the paper, we will frequently appeal to the following basic facts, which fall out immediately from the continuity of $\cL$ and the definition of path length.

\begin{enumerate}[label=\textbf{F.\arabic*},ref=F.\arabic*]
	\item \label{F:usc-paths} \textit{Path length is upper semi-continuous.} \qquad More precisely, if $\pi_n:[s_n, t_n] \to \R$ is a sequence of paths with $s_n \to s, t_n \to t$ for some $s < t$, and $\pi_n(r) \to \pi(r)$ for all $r \in (s, t)$, $\pi_n(s_n) \to \pi(s), \pi(t_n) \to \pi(t)$ then
	$$
	\limsup_{n \to \infty} \|\pi_n\|_\cL \le \|\pi\|_\cL.
	$$
	Note that the right-hand side can be defined by \eqref{E:length} even if $\pi$ is not a path (i.e. a \textit{continuous} function). However, if $\|\pi\|_\cL \ne - \infty$, then $\pi$ is forced to be continuous by Theorem \ref{T:landscape-shape}.  
	\item \label{F:limits-geo} \textit{Limits of geodesics are geodesics.} \qquad In the context of point $1$ above, if all the $\pi_n$ are geodesics, then 
	$$
	\lim_{n \to \infty} \|\pi_n\|_\cL = \lim_{n \to \infty} \cL(\bar \pi_n(s_n); \bar \pi_n(t_n)) = \cL(\bar \pi(s); \bar \pi(t)).
	$$
	Therefore by upper semi-continuity of path length, $\pi$ is also a geodesic.
	\item \label{F:concats} \textit{Concatenations add path length.} \qquad For paths $\pi:[s, r] \to \R, \tau:[r, t] \to \R$ with $\pi(r) = \tau(r)$, let $\pi \oplus \tau:[s, t] \to \R$ be equal to $\pi$ on $[s,r]$ and $\tau$ on $[s, t]$. Then
	$$
	\|\pi \oplus \tau\|_\cL = \|\pi\|_\cL + \|\tau\|_\cL.
	$$
	\item \label{F:switching}  \textit{Switching preserves geodesics.} \qquad If $\pi:[s, t] \to \R, \tau:[s', t'] \to \R$ are geodesics with $\pi(r) = \tau(r), \pi(r') = \tau(r')$ for some $r < r' \in [s, t] \cap [s', t']$, then $\pi|_{[s, r]} \oplus \tau|_{[r, r']} \oplus \pi|_{[r', t]}$ is a geodesic.
\end{enumerate}

\subsection{The Airy line ensemble}
\label{S:ALE-AS}
The Airy sheet is built from another important object, the parabolic Airy line ensemble. 
The \textbf{parabolic Airy line ensemble} is random continuous function $\fA:\N \X \R \to \R$, where $(i, x) \mapsto \fA_i(x)$. It is ordered so that almost surely
\begin{equation}
\label{E:A-ordering}
\text{for all $i, x$ we have $\fA_i(x) > \fA_{i+1}(x)$.}
\end{equation}
The parabolic Airy line ensemble is defined uniquely via a determinantal formula that we will not require in this paper. The term parabolic comes from the fact that the Airy line ensemble $\cA:\N \X \R\to \R$ given by $\cA_i(x) = \fA_i(x) + x^2$ is stationary in $x$. We also have the flip symmetry $\fA \eqd \tilde \fA$, where $\tilde\fA_i(x): =\fA_i(-x)$. We will explain shortly how the Airy sheet arises from the Airy line ensemble, but for now, the uninitiated reader can motivate their interest in $\fA$ by noting that $\fA_1(x) \eqd \cS(0, x) = \cL(0,0; x, 1)$, where the equality in distribution is as functions from $\R \to \R$.

The ensemble $\fA$ arises as the scaling limit at the edge of a sequence of $n$ non-intersecting Brownian motions as $n \to \infty$.
Because of this, it is best thought of as a system of countably many non-intersecting Brownian motions. This notion is made rigorous by a probabilistic resampling property called the \textbf{Brownian Gibbs property}, which we now introduce.

\begin{lemma}[The Brownian Gibbs property, \cite{CH}]
	\label{L:BG-property}
Let $\cF_{k, a, b}$ denote the $\sig$-algebra generated by $\fA$ outside of the set $\II{1, k} \X [a, b]$. Then conditional on $\cF_{k, a, b}$, the law of $\fA|_{\II{1, k} \X [a, b]}$ is given by a $k$-tuple of Brownian bridges $B = (B_1, \dots, B_k)$  with $B_i(a) = \fA_i(a), B_i(b) = \fA_i(b)$ for all $i$, drawn independently of each other and of $\fA|_{(\II{1, k} \X [a, b])^c}$ and conditioned on the event
$$
B_1(x) > B_2(x) > \dots > B_k(x) > \fA_{k+1}(x)
$$
for all $x \in [a, b]$. This event has positive probability by \eqref{E:A-ordering}.
\end{lemma}
The Brownian Gibbs property implies that the Airy lines $\fA_i$ are locally absolutely continuous with respect to Brownian motions. However, it does not say much about exactly how this absolute continuity occurs because of the difficulty in resampling non-intersecting Brownian bridges that avoid a lower boundary. Quantitative versions of absolute continuity were first investigated in \cite{hammond2016brownian} and further studied in \cite{calvert2019brownian, dauvergne2023wiener}. 
 For our purposes, we require a result from \cite{dauvergne2023wiener} that relates the Airy line ensemble to a collection of linearly shifted Brownian bridges.
 
 \begin{theorem}[part of Theorem 1.8, \cite{dauvergne2023wiener}]
 	\label{T:resampling-candidate}
 	Fix $t \ge 1, k \in \mathbb N$, and let $S = \II{1, k} \X [0, t]$. There is a random continuous function $\fL = \fL^{t, k}: S \to \R$ such that:
 	\begin{enumerate}[label=\arabic*.]
 		\item For any Borel set $A \sset \cC(S)$ we have
 		$$
 		\P(\fA|_S \in A) \le e^{c_k t^3} \P(\fL \in A).
 		$$
 		\item We have $\fL = L + B$, where $L = (L_1, \dots, L_k)$ is a $k$-tuple of affine functions, and $B$ is a $k$-tuple of independent Brownian bridges with $B_i(0) = B_i(t) = 0$. Crucially, $B$ and $L$ are independent.
 		\item (Lemma 4.2 and Corollary 4.4, \cite{dauvergne2023wiener}) There exists $d_{k, t} > 0$ such that for all $m > 0, r \in \{0, t\}, i \in \II{1, k}$: 
 		$$
 		\P(|L_i(r)| > m) \le 2e^{-d_{k, t} m^{3/2}}.
 		$$
 	\end{enumerate}
\end{theorem}
Note that Theorem \ref{T:resampling-candidate} is stated on the interval $[0, t]$, but by the shift invariance of $\cA(t) = \fA(t)+t^2$, also holds on any interval $[a, a + t]$. The difference $\fL_i(0) - \fL_i(t)$ also satisfies strong bounds ensuring that locally, each line $\fA_i$ is not just absolutely continuous with respect to Brownian motion, but has a uniformly bounded Radon-Nikodym derivative. We will just need this for the line $\fA_1$.

\begin{prop}[Corollary 1.3, \cite{dauvergne2023wiener}]
	\label{P:bounded-above}
	Let $B:\R \to \R$ be a two sided Brownian motion. Then there exists a constant $c > 0$ such that for any $t \ge 1$ and any Borel set $A \sset \cC([-t, t])$ we have
	$$
	\P(\fA_1|_{[-t, t]} - \fA_1(0) \in A) \le e^{c t^3} \P(B \in A).
	$$
\end{prop}

\subsection{The Airy sheet and an isometry}
\label{S:Airy-sheet}
We can define the Airy sheet using last passage across the Airy line ensemble, as follows. 
First, consider a continuous function $f \in \cC(I \X \R)$ where $I$ is an integer interval, and a non-increasing cadlag function $\pi:[x, y] \to \II{m, n} \sset I$, henceforth a \textbf{path} from $(x, n)$ to $(y, m)$. Define the length of $\pi$ with respect to the environment $f$ by
\begin{equation}
\label{E:pi-length}
\|\pi\|_f = \sum_{i = m}^n f_i(\pi_i) - f_i(\pi_{i+1}).
\end{equation}
Here $\pi_i = \inf \{t \in [x, y] : \pi(t) < i\}$ is the time when $\pi_i$ jumps off of line $i$, and if this set is empty we set $\pi_i = y$. We write
$$
f[(x, n) \to (y, m)] = \sup_\pi \|\pi\|_f,
$$
where the supremum is over all paths from $(x, n)$ to $(y, m)$. Last passage percolation is best thought of as a planar directed metric and certain last passage models have a beautiful solvable structure. Indeed, last passage percolation across Brownian motions was the first model shown to converge to the directed landscape, see \cite{DOV}. We can now define the Airy sheet.

\begin{definition}
	\label{D:Airy-sheet}
	Let $\fA$ be a parabolic Airy line ensemble, and let $\tilde \fA(x) = \fA(-x)$. For $(x, y, z) \in (0, \infty) \X \R^2$ define
	\begin{equation}
	\label{E:S'-form}
	\begin{split}
	\cS'(x, y, z) &= \lim_{k \to \infty} \fA[(-\sqrt{k/(2x)}, k) \to (y, 1)]- \fA[(-\sqrt{k/(2x)}, k) \to (z, 1)],\\
	\cS'(-x, y, z) &= \lim_{k \to \infty} \tilde \fA[(-\sqrt{k/(2x)}, k) \to (-y, 1)]- \tilde \fA[(-\sqrt{k/(2x)}, k) \to (-z, 1)]
	\end{split}
	\end{equation}
	and let $\cS(0, y) = \fA_1(y)$.
	For $(x, y) \in \Q \setminus \{0\} \X \Q$, define
	\begin{equation}
	\label{E:Sxy-lim}
	\cS(x, y) = \E\cA_1(0)+\lim_{n \to \infty} \frac{1}{2n} \sum_{z=-n}^{0}\cS'(x, y, z) - (z - x)^2,
	\end{equation}
	The {\bf Airy sheet} defined from $\fA$ is the unique continuous extension of $\cS$ to $\mathbb R^2$. 
\end{definition}

\begin{figure}
	\centering
	\includegraphics[width=6cm]{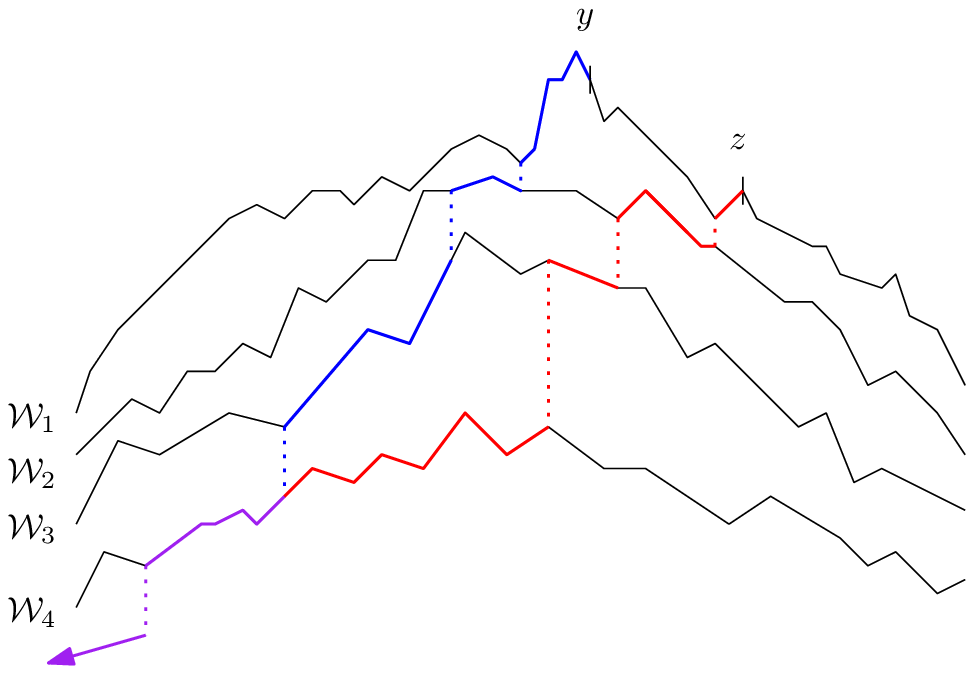}
	\caption{An illustration of $\cS(x, y, z)$ in the definition of the Airy sheet. This quantity is a limit of differences of last passage values from $(-\sqrt{k/(2x)}, k)$ in $\fA$. As we take $k$ to infinity, the portions of the last passage paths illustrated in red and blue above stabilize, and only the tail (illustrated in purple) changes for large $k$. This implies that the limit in \eqref{E:S'-form} exists and allows us to define $\cS(x, y) - \cS(x, z) :=\cS(x, y, z)$. Equation \eqref{E:Sxy-lim} extracts single Airy sheet values from these differences by ergodicity. Note that if we increase $x$, the tail will move from left to right.}
	\label{fig:Airy-sheet}
\end{figure}

See Figure \ref{fig:Airy-sheet} for an illustration of Definition \ref{D:Airy-sheet}.
It is not at all clear that the limits in Definition \ref{D:Airy-sheet} exist, or that the resulting function has a continuous extension to $\R^2$. However, a.s. these properties do hold and so the construction above is well-defined.
 This is shown for $x > \Q^+$ in \cite{DOV}, see also \cite[Section 1.11]{dauvergne2021scaling} for the extension to all $x \in \Q \smin \{0\}$. In this coupling, a.s.\ we have $\cS(0, x) = \fA_1(x)$ for all $x \in \R$ and for any $x > 0$, $y, z \in \R$ we have
\begin{equation}
\label{E:statement-of-x}
\cS'(x, y, z) = \cS(x, y) - \cS(x, z), \qquad \cS'(-x, y, z) = \cS(-x, y) - \cS(-x, z),
\end{equation}
see Definition 8.1 and Remark 8.1 in \cite{DOV}. 
While Definition \ref{D:Airy-sheet} is somewhat involved, one should think of it as capturing the idea that 
\begin{equation}
\label{E:sSWxinfty}
`` \cL(x, 0; y, 1) = \cS(x,y) = \fA[(x, \infty) \to (y, 1)]".
\end{equation}
In other words, there is a \textit{partial isometry} between $\cL$ and $\fA$. The more complex Definition \ref{D:Airy-sheet} is required to rigorously make sense of the right-hand side of \eqref{E:sSWxinfty}. The exact expression $(-\sqrt{k/(2x)}, k)$ for the asymptotic direction representing $(x, \infty)$ is not particularly intuitive and follows from calculations. See \cite{DOV}, \cite{dauvergne2021scaling} for more background. 

We can also use Definition \ref{D:Airy-sheet} to define the notion of paths and geodesics from $(x, \infty)$ to $(y, 1)$ across $\fA$, and lengths in $\fA$ of these paths.
For $x \in [0, \infty), y\in \R,$ and $n \in \N$, a nonincreasing cadlag function $\pi:(-\infty, y] \to \{n, n+1, \dots\}$ is a \textbf{parabolic path} from $(x, \infty)$ to $(y, n)$ if
\begin{equation}
\label{E:parabolic-path}
\lim_{z \to -\infty} \frac{\pi(z)}{2z^2} = - x.
\end{equation}
This definition guarantees that for $x > 0$, if $z_k$ is the largest time when $\pi(z) \le k$, then $z_k = - \sqrt{k/(2x + o(1))}$, matching the asymptotic direction from \eqref{E:S'-form}. 

For a parabolic Airy line ensemble $\fA$ with corresponding Airy sheet $\cS$ as in Definition \ref{D:Airy-sheet}, and any parabolic path starting at $x \ge 0$, define its \textbf{path length} by
\begin{equation}
\label{E:piSW}
\|\pi\|_\fA = \cS(x,y) + \lim_{z \to -\infty} \lf( \|\pi|_{[z,y]}\|_\fA - \fA[(z, \pi(z)) \to (y, 1)]\rg). 
\end{equation}
Note that we use $(y, 1)$ on the right hand side even if the parabolic path ends at a line $n \ne 1$. 
The limit above always exists and is nonpositive, see \cite[Lemma 4.1]{dauvergne2021disjoint}. Moreover, by \eqref{E:statement-of-x}, the formula \eqref{E:piSW} also holds if $\cS(x,y)$ is replaced by $\cS(x, w)$ and $\fA[(z, \pi(z)) \to (y, 1)]$ is replaced by $\fA[(z, \pi(z)) \to (w, 1)]$ for any $w \in \R$.

 The idea of definition \eqref{E:piSW} is that the best possible path length for $\pi$ should be the Airy sheet value $\cS(x, y)$. This accounts for the first term in \eqref{E:piSW}. For fixed $z$, the second term in \eqref{E:piSW} measures the discrepancy between the length of $\pi$ and the optimal length on the compact interval $[z, y]$; taking $z \to -\infty$ then gives the discrepancy between $\|\pi\|_\fA$ and $\cS(x,y)$. Now, for $(y, n) \in \R \X \N$ we define the \textbf{parabolic last passage value}
\begin{equation}
\label{E:fA-x-infty}
\fA[(x, \infty) \to (y, n)] = \sup_\pi \|\pi\|_\fA,
\end{equation}
where the supremum is over all parabolic paths from $(x, \infty)$ to $(y, n)$. We call a path a geodesic if it achieves this supremum. The next lemma gives a few basics about paths and geodesics in $\fA$.
\begin{lemma}
	\label{L:basics-para-path}
	Let $\fA, \cS$ be coupled as in Definition \ref{D:Airy-sheet}. Then we have the following.
\begin{enumerate}[nosep, label=(\roman*)]
	\item (Lemma 4.1(ii), \cite{dauvergne2021disjoint}) $\cS(x, y) = \fA[(x, \infty) \to (y, 1)]$ for all $x \ge 0, y \in \R$. This makes sense of the isometry \eqref{E:sSWxinfty}.
	\item (Lemma 4.1(i, ii, v), \cite{dauvergne2021disjoint}) Almost surely, for every $(x, y) \in [0, \infty) \X \R$, there exists a geodesic $\pi$ in $\fA$ from $(x, \infty)$ to $(y, 1)$. Moreover, almost surely for every $(x, y) \in [0, \infty) \X \R$, there is a geodesic $\pi_R\{x,y\}$ from $x$ to $y$ satisfying $\pi(t) \le \pi_R\{x,y\}(t)$ for any geodesic $\pi$ from $x$ to $y$ and all $t \in (-\infty, y]$. We call $\pi_R\{x,y\}$ the \textbf{rightmost geodesic} from $(x, \infty)$ to $(y, 1)$. We have the monotonicity $\pi_R\{x,y\}(t) \le \pi_R\{x',y'\}(t)$ whenever $x \le x', y \le y'$ and $t \le y$.
	\item (Lemma 3.9, \cite{dauvergne2021last}) For any $x \in [0, \infty), y < z \in \R$ and any geodesic $\pi$ from $(x, \infty)$ to $(z,1)$, we have
	\begin{align*}
\fA[(x, \infty) \to (z, 1)] &= \max_{i \in \N} \fA[(x, \infty) \to (y, i)] + \fA[(y, i) \to (z, 1)] \\
&=\fA[(x, \infty) \to (y, \pi(y))] + \fA[(y, \pi(y)) \to (z, 1)].
	\end{align*}
	Here the last passage values starting at $(x, \infty)$ are parabolic last passage values, and the others are standard last passage values. 
\end{enumerate}
\end{lemma}

\subsection{The extended landscape}
\label{S:extended-L}

We extend the definition of $\cL$ to measure optimal lengths of disjoint paths. Here and throughout the paper, we say that two paths $\pi, \tau:[s, t] \to \R$ are \textbf{disjoint} if $\pi(r) \ne \tau(r)$ for all $r \in (s, t)$. 

\begin{definition}
	\label{D:ext-land}
	Let $\fX_\uparrow$ be the space of all points $(\bx, s; \by, t)$, where $s < t \in \R$ and $\bx, \by$ lie in the same space $\R^k_\le = \{ \bx \in \R^k: x_1 \le \dots \le x_k \}$ for some $k \in \N$. For $\bu = (\bx, s; \by, t) \in \fX_\uparrow$, define
	\begin{equation}
	\label{E:Lextend}
	\cL(\bu) = \sup_{\pi_1, \dots, \pi_k} \sum_{i=1}^k \|\pi_i\|_\cL.
	\end{equation}
	Here and throughout we use the convention that $k$ is such that $\bx, \by \in \R^k_\le$. The supremum is over all $k$-tuples of paths $\pi = (\pi_1, \dots, \pi_k)$ where each $\pi_i$ is a path from $(x_i, s)$ to $(y_i, t)$, and the paths $\pi_i$ are disjoint. We call such a collection $\pi$ a \textbf{disjoint $k$-tuple} from $(\bx, s)$ to $(\by, t)$. We call the extension of $\cL$ from $\fX_\uparrow \to \R \cup \{-\infty\}$ the \textbf{extended landscape}.
	
	Geometrically, the space $\fX_\uparrow = \bigcup_{k=1}^\infty \fX^k_\uparrow$ where
	$
\fX^k_\uparrow = \{(\bx, s; \by, t) : s < t, \bx, \by \in \R^k_\le\}
	$
	inherits the Euclidean metric and topology.
	With this setup, the extended landscape is continuous on $\fX_\uparrow$ by \cite[Theorem 1.4/1.6]{dauvergne2021disjoint}. In fact, as with the standard landscape we have a modulus of continuity bound, see the comment immediately after \cite[Lemma 6.5]{dauvergne2021disjoint}.
	\begin{lemma}
		\label{L:extended-modulus}
Fix $k \in \N, \ep > 0$ and a compact subset $K \sset \fX^k_\uparrow$. Then there exists a random $C > 0$ such that for all $\bu = (\bx, s; \by, t), \bu' = (\bx', s'; \by', t') \in K$ we have
$$
|\cL(\bu) - \cL(\bu')|\le C\lf(\|\bx - \bx'\|^{1/2-\ep}_2 + \|\by - \by'\|^{1/2-\ep}_2 + |s -s'|^{1/3-\ep} + |t-t'|^{1/3-\ep}\rg).
$$
	\end{lemma}
	 Moreover, by \cite[Theorem 1.7]{dauvergne2021disjoint}, almost surely the supremum in \eqref{E:Lextend} is attained by a disjoint $k$-tuple for all $\bu \in \fX_\uparrow$.  We call a disjoint $k$-tuple that attains this supremum a \textbf{disjoint optimizer}. 
	Defining $\fX = \bigcup_{k=1}^\infty \R^k_\le \X \R^k_\le$, we call the function $\cS:\fX \to \R$ given by
	$$
	\cS(\bx, \by) := \cL(\bx, 0; \by, 1)
	$$
	the \textbf{extended Airy sheet}.
\end{definition}

We can give a similar definition for last passage across the Airy line ensemble. We say parabolic paths $\pi_i:(-\infty, y_i] \to \N, i = 1, 2$ with $y_1 < y_2$ are disjoint if $\pi_1(z) < \pi_2(z)$ for $z < y_1$.
For $(\bx, \by) \in \fX$ with $x_1 \ge 0$, let
$$
\fA[(\bx, \infty) \to (\by, 1)] = \sup_\pi \sum_{i=1}^k \|\pi\|_\fA,
$$
where the supremum is over all $k$-tuple of disjoint parabolic paths $\pi_i$ from $(x_i, \infty)$. Again, this supremum is always attained, see \cite[Proposition 5.8]{dauvergne2021disjoint}. 

Remarkably, the isometry in Lemma \ref{L:basics-para-path}(i) extends to $\fX$, and the parabolic Airy line ensemble can actually be defined from the extended landscape. For this definition and throughout the paper, we write $x^k = (x, \dots, x) \in \R^k$.

\begin{theorem}
	\label{T:iso-multiples}
	Let $\cS, \cL$ be as in Definition \ref{D:ext-land}, and define $\fA:\N \X \R \to \R$ so that $\sum_{i=1}^k \fA_i(x) = \cS(0^k, x^k)$. Then:
	\begin{enumerate}[label=\arabic*.]
		\item (Equation (9), \cite{dauvergne2021disjoint}) $\fA$ is a parabolic Airy line ensemble.
		\item (Theorem 1.3, \cite{dauvergne2021disjoint}) $\fA, \cS$ are coupled as in Definition \ref{D:Airy-sheet}. Moreover, almost surely, $\fA[(\bx, \infty) \to (\by, 1)] = \cS(\bx, \by)$ for all $(\bx, \by) \in \fX$ with $x_1 \ge 0$.
	\end{enumerate}
\end{theorem}

Hopefully the relevance of the extended landscape to the study of geodesic networks is clear. Indeed, interesting geodesic networks arise when certain points are connected by disjoint geodesics in $\cL$, which is equivalent to an extended landscape value equalling a sum of usual landscape values.

We end this section by recording a simpler form for last passage across $\fA$ starting at $(0^k, \infty)$. For this definition, for $\by \in \R^k_\le$ we define the disjoint last passage value 
$$
\fA[(z, \II{1, k}) \to (\by, 1)]
$$
to be the maximum weight $\sum_{i=1}^k \|\pi_i\|_\fA$ of a $k$-tuple of paths $\pi_i$ from $(z, i)$ to $(y_i, 1)$ that satisfy $\pi_i > \pi_{i+1}$ on $(z, y_i)$.

\begin{prop}[Proposition 5.9, \cite{dauvergne2021disjoint}]
	\label{P:high-paths-B}
	Almost surely the following holds. For any $k \in \N$ and $\by \in \R^k_\le$ there is a disjoint optimizer $\pi$ in $\fA$ from $(0^k, \infty)$ to $(\by, 1)$ that only uses the top $k$ lines. In particular, for any $z < y_1$ we have
	\begin{equation}
	\label{E:0ky}
	\fA[(0^k, \infty) \to (\by, 1)] = \sum_{i=1}^k \fA_i(z) + \fA[(z, \II{1, k}) \to (\by, 1)].
	\end{equation}
\end{prop}

\section{Coalescent planar geometry}
\label{S:coalescent}

As discussed in Section \ref{S:geometric-ideas}, geodesics in the directed landscape typically coalesce. The main goal of this section is to establish results related to this coalescence. Along the way we will show that the network graph $G(p; q)$ of any pair $(p; q) \in \Rd$ satisfies point 1-3 of Theorem \ref{T:geodesic-networks}; this follows from Proposition \ref{P:sparse-space-2}(ii) together with Lemma \ref{L:rightmost-geods}.3.
 We start the following lemma, which bounds the number locations geodesics can have at a fixed time. 

\begin{lemma}
	\label{L:sparse-space-1}
The following claims hold almost surely for $\cL$.
	\begin{enumerate}[label=(\roman*)]
		\item Fix rational times $s < r < t$ and any interval $[a, b]$. There exists a random finite set of points $X = \{x_1, \dots, x_N\} \sset \R$ such that if $\pi:[s, t] \to \R$ is any geodesic with $\pi(s), \pi(t) \in [a, b]$, then
$\pi(r) \in X$. 
		\item For every point $u = (p; q) = (x, s; y, t) \in \Rd$ and $r \in (s, t)$ there is a random set of points $X = \{x_1, \dots, x_{11}\} \sset \R$ such that if $\pi$ is a geodesic from $p$ to $q$ then
		$\pi(r) \in X$.
	\end{enumerate}
\end{lemma}

The proof of Lemma \ref{L:sparse-space-1} requires analysis of the Airy line ensemble. 

\begin{proof}[Proof of Lemma \ref{L:sparse-space-1}(i)]
By rescaling (Lemma \ref{L:invariance}), it is enough to prove the lemma when $s = 0, r = 1,$ and $t \ge 2$. Define $\cS_1(x, y) = \cL(x, 0; y, 1)$, $\cS_2(z, y) = \cL(y, 1, z, 2)$. These are both Airy sheets by the symmetries in Lemma \ref{L:invariance}. Let 
$$
F_{x, z}(y) = \cS_1(x, y) + \cS_2(z, y).
$$
By Corollary \ref{C:geo-containment}, any geodesic $\pi:[0, t] \to \R$ with $\pi(s), \pi(t) \in [a, b]$ satisfies $\pi(0), \pi(1), \pi(2) \in [-C, C]$ for some random constant $C < \infty$. Therefore since $\pi(1)$ is a maximizer of the function $F_{\pi(0), \pi(2)}$ on the interval $[-C, C]$, it suffices to show that for every $n \in \N$, almost surely there are at most finitely many $y \in [-n, n]$ that maximize $F_{x, z}$ for some $x, z \in [-n, n]$. By shift invariance of $\cS$ (Lemma \ref{L:invariance}.2) rather than working on the interval $[-n, n]$ we can work on the interval $[0, 2n]$.

Let $\fA, \fA'$ be Airy line ensembles coupled to $\cS_1, \cS_2$ as in Definition \ref{D:Airy-sheet}.
By Lemma \ref{L:basics-para-path}(i, iii), for $x, z \in [0, 2n], y \in \R$ we can write
\begin{align}
\nonumber
F_{x, z}(y) &= \fA[(x, \infty) \to (y, 1)] + \fA'[(z, \infty) \to (y, 1)] \\
\label{E:max-ijN}
&= \max_{i, j \in \N} \fA[(x, \infty) \to (-1,i)] + \fA[(-1, i) \to (y, 1)] \\
&\qquad + \fA'[(z, \infty) \to (-1,j)] + \fA'[(-1, j) \to (y, 1)],
\end{align}
Moreover, if we let $\pi_R, \pi_R'$ be the rightmost geodesics in $\fA, \fA'$ from $(2n, \infty)$ to $(2n, 1)$, then by the monotonicity in Lemma \ref{L:basics-para-path}(ii) and the second equality in Lemma \ref{L:basics-para-path}(iii), 
letting $M = \max(\pi_R(-1), \pi'_R(-1)),$ we can take the maximum in \eqref{E:max-ijN} over $i, j \in \II{1, M}$. Therefore if $y \in [0, 2n]$ maximizes $F_{x, z}$ for some $x, z \in [0, 2n]$, then $y$ must maximize the function
$$
H_{i, j}[\fA, \fA'](y) := g_i[\fA](y) + g_j[\fA'](y), \qquad \text{where} \quad g_\ell[A](y) = A[(-1, \ell) \to (y, 1)].
$$
for some $i, j \in \II{1, M}$ where $H_{i, j}[\fA, \fA']:[-1, 2n+1] \to \R$. To finish the proof, it is enough to show that almost surely $H_{i, j}[\fA, \fA']$ has at most one maximizer in $[0, 2n]$ for all fixed $i, j$. Now,
$$
\fA|_{\II{1, i} \X [-1, 2n+1]} - \fA(0), \qquad \fA'|_{\II{1, j} \X [-1, 2n+1]} - \fA'(0)
$$
are jointly absolutely continuous with respect to the law of $i + j$ independent Brownian motions $B = (B_1, \dots, B_i), B' = (B_1', \dots, B_j')$ by Lemma \ref{L:BG-property}. Therefore it suffices to show that $H_{i, j}[B, B']$ has at most one maximizer in $[0, 2n]$. For this, we note that the law of $g_i[B]$ restricted to $[0, 2n]$ is absolutely continuous with respect to the law of a Brownian motion $W$ with $W(-1) = 0$ for every $i \in \N$. This is well-known: for example, it follows from the fact that the map $g_i[B]$ is the top line $A_1$ in an ensemble of $i$ nonintersecting Brownian motions $A_1 > \dots > A_i$ started from the initial condition $A(-1) = 0$, see \cite{o2003path}. Therefore $g_i[B] + g_j[B']$ has a unique maximizer on $[0, 2n]$ almost surely, as desired. 
\end{proof}

To prove Lemma \ref{L:sparse-space-1}.(ii), we first need a lemma about Brownian motion.

\begin{lemma}
	\label{L:Brownian-max-lemma}
	Let $I$ be a real interval, let $f:I \to \R$, and for $\ep, m > 0$ let $E_{m, \ep}(f)$ be the maximum size of the largest set of points $\{z_1, \dots, z_k\} \sset I$ with $|z_i - z_j| > m$ for all $i \ne j$ and 
	$$
	f(z_i) \ge \max_{z \in I} f(z) - \ep
	$$
	for all $i$. Then if $B:[-1, 1] \to \R$ is a Brownian motion we have
	$$
	\P(E_{m, \ep}(B) \ge k) \le c_{m, k} \ep^{k-1}.
	$$
\end{lemma}

\begin{proof}
	For $i \in \{-2m, \dots, 2m-1\}$, let $M_i = \max_{x \in [i/(2m), (i+1)/2m]} B(x)$. Let $\cI$ be the set of $k$-tuples of indices $\mathbf i = (i_1, \dots, i_k) \in \{-2m, \dots, 2m-1\}^k$ such that $|i_j - i_{j'}|	\ge 2$ for all $j \ne j'$. If $E_{m, \ep}(B) \ge k$ then there exists $\mathbf i \in \cI$ such that $|M_{i_j} - M_{i_{j'}}| \le \ep$ for all $j \ne j$. 
	
	To complete the proof it just remains to take a union bound over $\cI$ and observe that for any $\mathbf i \in \cI$ the vector $(M_{i_1}, \dots, M_{i_k})$ has a Lebesgue density on $\R^k$ bounded above by a constant depending only on $m$.
\end{proof}

\begin{proof}[Proof of Lemma \ref{L:sparse-space-1}(ii)]
Setting some notation, for $u = (x, s; y, t) \in \Rd$ and $r \in (s, t)$ define
$$
Y_{u, r}(z) = \cL(x, s; z, r) + \cL(z, r; y, t),
$$	
and let $T(u, r) = \argmax Y_{u, r}(z)$. The set $T(u, r)$ is the set of all spatial locations along geodesics from $(x, s)$ to $(y, t)$ at time $r$. Our aim is to show that $\# T(u, r) \le 11$ for all $u, r$. First, observe that by continuity of geodesics in $\cL$, if $\# T(u, r') \ge 12$ for some point $(u, r')$, then $\# T(u, r) \ge 12$ for all $r$ in an open interval containing $r'$; in particular, this holds at some rational $r$.
	
	Next, fix $r$ rational and $\ga > 0$ and let $K_\ga$ be the set of all $u = (x, s; y,t; r) \in \Rd \X \R$ where $|x|, |s|, |y|, |t| \le \ga^{-1}$ and $s + \ga < r < t - \ga$. By Theorem \ref{T:landscape-shape}, for every $\ga > 0$ there is a random $D > 0$ such that for all $(u, r) \in K_\ga$, we have $T(u,r) = T_D(u, r)$ where $T_d(u, r) := T(u, r)\cap [-d, d]$.
	Therefore it is enough to show that for every fixed $\ga > 0$ and $d > 0$, almost surely $T_d(u, r) \le 11$ for all $(u, r) \in K_\ga$. 
	For every $\ep > 0$, let 
	$$
	K_{\ga, \ep} = K_\ga \cap [\Z/(\ep^2 \log \ep)] \times [\Z/(\ep^3 \log \ep)] \times [\Z/(\ep^2 \log \ep)] \times [\Z/(\ep^3 \log \ep)].
	$$
	By the modulus of continuity for $\cL$ (Proposition \ref{P:mod-land-i}), if $T_d(u, r) \ge 12$ for some $u \in K_\ga$ then using the notation of Lemma \ref{L:Brownian-max-lemma}, there exists an $m \in \N$ such that for all small enough $\ep > 0$ we have
	\begin{equation}
	\label{E:Kgaep}
	\max_{u \in K_{\ga, \ep}} E_{1/m, \ep} (Y_{u, r}|_{[-d, d]}) \ge 12.
	\end{equation}
	Therefore it is enough to show that the probability of \eqref{E:Kgaep} goes to $0$ with $\ep$ for every fixed $\ga, d, m$. We do this with a union bound over $K_{\ga, \ep}$, using that $\# K_{\ga, \ep} \le c_\ga \ep^{-10} \log^8(\ep^{-1})$. By Proposition \ref{P:bounded-above} and Lemma \ref{L:Brownian-max-lemma}, letting $B, B'$ denote two independent two-sided Brownian motions defined on $[-d, d]$, for all $(u, r) \in K_\ga$ we have
	$$
	\P(E_{1/m, \ep} (Y_{u, r}|_{[-d, d]}) \ge 12) \le c_{d, \ga} \P(E_{1/m, \ep} (B + B') \ge 12) \le c_{m, d,\ga} \ep^{11}.
	$$
	The fact that the first bound is uniform over all choices of $(u, r) \in K_{\ga, \ep}$ follows from the symmetries in Lemma \ref{L:invariance} and the fact that we are working on a compact set. Taking a union bound yields the result.
\end{proof}

\subsection{Coalescence via Lemma \ref{L:sparse-space-1}}

Given Lemma \ref{L:sparse-space-1}, all coalescence results in this section are based on essentially topological arguments. Throughout this section we work on the almost sure event $\Om$ defined in Lemma \ref{L:regularity}, and on the almost sure set where Lemma \ref{L:sparse-space-1} holds. All results will be deterministic while on this set.

To state coalescence results in this section, define the \textbf{overlap} $O(\pi, \tau)$ of two geodesics $\pi:[s, t] \to \R, \tau:[s',t'] \to \R$ as the closure of the set
$$
\{r \in (s, t) \cap (s', t') : \pi(r) = \tau(r)\}.
$$
Note that $O(\pi, \tau)$ can contain endpoints of the open interval $(s, t) \cap (s', t')$ if $\pi$ and $\tau$ overlap on a dense set near these endpoints.
Our first result on coalescence in the following lemma.

\begin{lemma}
	\label{L:rightmost-geods}
	Let $(p; q) = (x, s; y, t)\in \Rd$, and let $\ga$ be any geodesic from $p$ to $q$. Then:
	\begin{enumerate}[nosep, label=\arabic*.]
		\item There exists a sequence of rational points $(p_n, q_n) \in (x_n, s_n; y_n, t_n) \in \Rd$ with $p_n \to p$ and $q_n \to q$ such that letting $\ga_n$ be the unique geodesic from $p_n$ to $q_n$, then $O(\ga_n, \ga)$ is a closed interval whose endpoints converge to $s$ and $t$, respectively.
		\item For any other geodesic $\pi$, the overlap $O(\ga, \pi)$ is a (possibly empty) closed interval.
		\item For any $[s', t'] \sset (s, t)$, $\ga|_{[s', t']}$ is the unique geodesic from $\bar \ga(s')$ to $\bar \ga(t')$.
	\end{enumerate}
\end{lemma}

Shortly before this paper was first posted, Bhatia \cite[Theorem 1]{bhatia2023duality} developed a different proof of this result. When $\ga$ is a rightmost geodesic, Lemma \ref{L:rightmost-geods}.1 also follows from the proof of \cite[Corollary 3.5]{dauvergne2020three}.

\begin{proof}
	First, it is enough to show that there is a set of pairs $S := \{(s', t') \in \Q^2 \cap (s, t) : s' < t'\}$ such that $(s, t)$ is in the closure of $S$ and such that the three points of the lemma hold when $\ga$ is replaced by $\ga|_{[s', t']}$ for $(s', t') \in S$. This is clear for part $3$. The rational approximation in part $1$ follows by a diagonalization argument from the existence of rational approximations for all $\ga|_{[s', t']}$, and finally if $O(\ga, \pi)$ is not a closed interval, then $O(\ga|_{[s', t']}, \pi)$ will not be a closed interval for some $(s', t')$ close to $(s, t)$.
	
	To define the set $S$, first define  $f:[s, t] \to \N$ by
	$$
	f(r) = \# \{\pi(r) : \pi \text{ is a geodesic from $p$ to $q$}\}.
	$$
	By Lemma \ref{L:sparse-space-1}(ii), $f(r) \le 11$ for all $r \in (s, t)$. 
	The boundedness of $f$ implies that $f$ has a dense set of local maxima. Noting also that $f$ is lower semi-continuous, $f$ must be constant in a region around each of its local maxima, so $f$ has a dense set of rational local maxima $M \sset \Q$. Define
	$$
	S := \{(s', t') \in M^2 \cap (s, t) : s' < t'\},
	$$
	and for every $s \in M$, let $[a_s, b_s]$ be a closed interval containing $s$ in its interior on which $f$ is constant. 
	
	\textbf{Proof of 1 for $\ga|_{[s', t']}, (s', t') \in S$.} \qquad Since $f$ is constant on $[s', b_{s'}], [a_{t'}, t']$, all geodesics from $\bar \ga(s')$ to $\bar \ga(t')$ must coincide on $[s', b_{s'}], [a_{t'}, t']$. In particular, the geodesic $\ga$ coincides with the rightmost geodesic $\pi$ from $\bar \ga(s')$ to $\bar \ga(t')$ on $[s', b_{s'}], [a_{t'}, t']$.
	
		Let $x' = \ga(s'), y' = \ga(t')$, let $x_n, y_n$ be rational sequences with $x_n \cvgdown x', y_n \cvgdown y'$, and let $\ga_n$ be the geodesics from $(x_n, s')$ to $(y_n, t')$. By the monotone ordering on rightmost geodesics (Lemma \ref{L:regularity}.3) and the fact that pointwise limits of geodesics are geodesics (Fact \ref{F:limits-geo}), $\ga_n \cvgdown \pi'$ pointwise, where $\pi'$ is a geodesic from $\bar \ga(s')$ to $\bar \ga(t')$. Since $\ga_n(r) \ge \pi(r)$ for all $n \in \N, r \in [s', t']$ and $\pi$ is the rightmost geodesic we have $\pi' = \pi$.
		
	Therefore for any rational $\ep > 0$, either there are infinitely many values in the set $\{\ga_n(s' + \ep), \ga_n(t' - \ep)\}$ or else $\ga_n(s' + \ep) = \pi(s'+\ep), \ga_n(t'+\ep) = \pi(t' + \ep)$ for all small enough $\ep$. The former possibility cannot occur by Lemma \ref{L:sparse-space-1}(i). The latter implies that $\{s' + \ep, t' - \ep\} \sset O(\ga_n, \pi)$ for large enough $n$. Since $[s', b_{s'}] \cup [a_{t'}, t'] \sset O(\ga|_{[s', t']}, \pi)$ we have $\{s' + \ep, t' - \ep\} \sset O(\ga_n, \ga|_{[s', t']})$ for small enough $\ep > 0$ and large enough $n$. Finally, $O(\ga_n, \ga)$ is connected since otherwise we could contradict the uniqueness of the rational geodesics $\ga_n$ by appealing to geodesic switching (Fact \ref{F:switching}). Therefore $O(\ga_n, \ga|_{[s', t']})$ is an interval converging to $[s', t']$.
	
	\textbf{Proof of 2, 3 for $\ga|_{[s', t']}, (s', t') \in S$.} \qquad We use statement $1$ for $\ga|_{[s', t']}$ with approximating sequence $\ga_n$. Since $\ga_n$ is the unique geodesic from $p_n$ to $q_n$, the set $O(\ga_n, \pi)$ is connected for any geodesic $\pi$ by appealing to geodesic switching.
	Therefore statement $2$ follows from $1$, since if $O(\ga, \pi)$ is disconnected, then 
	$
	O(\ga_n, \pi) 
	$
	must be disconnected for large enough $n$. Statement $3$ also follows from $1$ since otherwise we could contradict uniqueness of rational geodesics by again using a geodesic switching argument.
\end{proof}

We can use Lemma \ref{L:rightmost-geods} to upgrade Lemma \ref{L:sparse-space-1}. To state the upgrade, let $K_1 = [x_1, y_1] \X [s_1, t_1], K_2 = [x_2, y_2] \X [s_2, t_2]$ with $t_1 < s_2$ be two compact boxes, let $I = [r, r'] \sset [t_1, s_2]$, define
$$
H(K_1, K_2, I) = \# \{\pi|_I : \pi \text{ is a geodesic from some $p \in K_1$ to $q \in K_2$}\}.
$$
In words, $H(K_1, K_2, I)$ counts the number of paths defined on the interval $I$ that are segments of geodesics from $K_1$ to $K_2$.

\begin{prop}
	\label{P:sparse-space-2}
	The following claims hold almost surely for $\cL$.
	\begin{enumerate}[label=(\roman*)]
		\item For any compact sets $K_1 = [x_1, y_1] \X [s_1, t_1], K_2 = [x_2, y_2] \X [s_2, t_2]$ with $t_1 < s_2$ and $I = [r, r'] \sset (t_1, s_2)$, we have $H(K_1, K_2, I) < \infty$. Note that this no longer holds if $I$ contains either $t_1$ or $s_2$.
		\item For every point $(p; q) = (x, s; y, t) \in \Rd$, we have $H(p, q, [s, t]) \le 121$. That is, there are at most $121$ geodesics between any pair of points in $\Rd$.
	\end{enumerate}
\end{prop}

\begin{proof}
By Corollary \ref{C:geo-containment}, there exists a random compact interval $[-A, A]$ and rational times $r_1 \in (t_1, r), r_2 \in (r', s_2)$ such that any geodesic from $K_1$ to $K_2$ lives in the interval $[-A, A]$ at times $r_1, r_2$. Therefore by Lemma \ref{L:sparse-space-1}(i), we have
\begin{equation}
\label{E:H-finite}
H(K_1, K_2, \{r_1\}) < \infty, \qquad H(K_1, K_2, \{r_2\}) < \infty.
\end{equation}
Next, by Lemma \ref{L:rightmost-geods}.3, if $p$ is an interior geodesic point at time $r_1$ and $q$ is an interior geodesic point at time $r_2$, then there is a unique geodesic from $p$ to $q$. Therefore
\begin{equation}
\label{E:HK1}
H(K_1, K_2, I) \le H(K_1, K_2, [r_1, r_2]) \le H(K_1, K_2, \{r_1\}) H(K_1, K_2, \{r_2\}),
\end{equation}
which is finite by \eqref{E:H-finite}.

For part (ii), using \eqref{E:HK1} with $K_1 = \{p\}, K_2 = \{q\}, I = [s +\ep, t - \ep]$ along with Lemma \ref{L:sparse-space-1}(ii) gives that
$
H(p, q, [s +\ep, t - \ep]) \le 121
$
for all $\ep > 0$. Now, $\lim_{\ep \to 0^+} H(p, q, [s +\ep, t - \ep]) = H(p, q, [s, t])$, yielding part (ii).
\end{proof}

\subsection{The overlap metric}

We can use overlap to define a tractable metric on the space of geodesics.
Let $\ga, \pi$ be two geodesics with domains $I, I'$, and let $\la$ denote Lebesgue measure on $\R$. Define the \textbf{overlap distance} between $\ga$ and $\pi$ as
$$
d_o(\ga, \pi) = \la(I) + \la(I') - 2 \la (O(\pi, \ga)),
$$
and write $\to_o$ for convergence in this metric.
For this next proposition we let $\Ga(S)$ denote the set of geodesics with endpoints in a set $S \sset \Rd$.
\begin{prop}
	\label{P:overlap-metric}
	The overlap distance $d_o$ is a metric on the space of all geodesics. Under this metric:
	\begin{enumerate}[nosep, label=\arabic*.]
		\item The countable set $\Ga(\Q^4_\uparrow)$ is dense in $\Ga(\Rd)$.
		\item For any compact set $K \sset \Rd$, the overlap metric makes $\Ga(K)$ a compact Polish space and $\ga_n \to_o \ga$ in $\Ga(K)$ if and only if $\fg \ga_n \to \fg \ga$ in the Hausdorff topology.
	\end{enumerate}
\end{prop}

\begin{proof}
	It is easy to check that $d_o$ defines a metric. Indeed, for geodesics $\pi, \ga, \tau$ where $\ga$ has domain $I$, the set containment $O(\pi, \ga) \cap O(\ga, \tau) \sset O(\pi, \tau)$ implies that
	$$
	\la(O(\pi, \ga)) + \la(O(\ga, \tau)) \le \la(O(\pi, \tau)) + \la(I),
	$$
	which yields the triangle inequality. The density of $\Ga(\Q^4_\uparrow)$ in $\Ga(\Rd)$ follows from Lemma \ref{L:rightmost-geods}.1. 
	
	For the second point, first observe that by Corollary \ref{C:geo-containment} and the fact that pointwise limits of geodesics are geodesics (Fact \ref{F:limits-geo}), the space $\Ga(K)$ is both complete and compact in the topology of Hausdorff convergence of the graphs $\fg \ga$. Therefore we just need to verify that the Hausdorff and overlap metrics induce the same topology on $\Ga(K)$. If $\fg \ga_n \to \fg \ga$, then we must have that $\ga_n \to_o \ga$ as otherwise we contradict Proposition \ref{P:sparse-space-2}(i). On the other hand, if $\ga_n \to_o \ga$ then the only potential limit point for $\fg \ga_n$ in the space $\{\fg \pi : \pi \in \Ga(K)\}$ is $\fg \ga$. Therefore by compactness of the graph space $\{\fg \pi : \pi \in \Ga(K)\}$, we have $\fg \ga_n \to \fg \ga$.
\end{proof}

We end the section with one more overlap approximation lemma.

\begin{lemma}
\label{L:overlap-at-endpoints}
For any $(p; q) = (x, s; y, t)\in \Rd$ and any geodesic $\ga$ from $p$ to $q$ we can find a sequence of rational points $q_n \to q$ and a sequence of geodesics $\pi_n$ from $p$ to $q_n$ such that the overlap $O(\pi_n, \ga) = [s, r_n]$ for a sequence $r_n \to t$ as $n \to \infty$.
\end{lemma}

\begin{proof}
Since there are only finitely many geodesics from $p$ to $q$ (Proposition \ref{P:sparse-space-2}(ii)) for every $\ep > 0$ we can find a point $r \in (t-\ep, t)$ and $\de > 0$ such that every geodesic $\sig$ from $p$ to $q$ with $\sig(r) = \ga(r)$ has $[r-\de, r] \sset O(\sig, \ga)$. If we can prove the claim in lemma with $\ga$ replaced by $\ga|_{[s, r]}$, then the claim for $\ga$ follows by taking $\ep \to 0$ and applying a diagonalization argument.

Consider any sequence of rational points $q_n = (x_n, t_n) \to \bar \ga(r)$ and let $\sig_n$ be a sequence of geodesics from $p$ to $q_n$. By Proposition \ref{P:overlap-metric}, after possibly passing to a subsequence we may assume that $\sig_n \to_o \sig$ for some geodesic $\sig$ to $p$ to $\bar \ga_n$. Therefore $\sig_n(r-\de) = \sig(r-\de) = \ga(r-\de)$  for all large enough $n$, so the concatenation $\pi_n := \ga|_{[s, r-\de]} \oplus \sig_n|_{[r-\de, t_n]}$ is a geodesic from $p$ to $q_n$. The sequence $\pi_n$ satisfies the claim in the lemma.
\end{proof}

\section{Geodesic stars}
\label{S:weighted-stars}

As discussed in Section \ref{S:geometric-ideas} one of the main components in the proof of Theorems \ref{T:geodesic-networks} and \ref{T:Hausdorff-dimension} is a careful study of weighted geodesic stars. This is the subject of the present section. Along the way, we will prove Theorem \ref{T:star-computation}.
First, for $\bx \in \R^k_<, s \in \R$, and $p \in \R \X (-\infty, s), q \in \R \X (s, \infty)$ define the \textbf{weight} vectors
$$
W^-_{\bx, s}(p) = (\cL(p; x_1, s), \dots, \cL(p; x_k, s)), \qquad W^+_{\bx, s}(q) = (\cL(x_1, s; q), \dots, \cL(x_k, s; q)).
$$
Next, for $W \sset \R^k$ we let
$$
\operatorname{Star}(\bx, s; W)
$$
be the set of points $p \in \R \X (-\infty, s)$ such that $W^-_{\bx, s}(p) \in W$ and such that there is a forwards $k$-geodesic star from $p$ to $(\bx, s)$. We similarly let $\operatorname{Star}_R(\bx, s; W)$ be the set of points $q \in \R \X (s, \infty)$ such that $W^+_{\bx, s}(q)\in W$
 and such that 
 there is a reverse $k$-geodesic star from $(\bx, t)$ to $q$.
For $k, \ell \in \N$, $\bx \in \R^k_<, \by \in \R^\ell_<$, $s < t$, and $W \sset \R^{k} \X \R^{\ell}$ we define
$$
\operatorname{Star}(\bx, s; \by, t; W) = \bigcup_{(w_1, w_2) \in W} \operatorname{Star}(\bx, s; \{w_1\}) \X \operatorname{Star}_R(\by, y; \{w_2\}).
$$
The goal of the section is to bound the sizes of the sets $\operatorname{Star}(\bx, s; \by, t; W)$ when $W \sset \R^k \X \R^\ell$ is an affine set satisfying $W = W + (a^k, b^\ell)$ for every $a, b \in \R$. These are the sets that will appear when we apply cut and paste operations to geodesic networks. We call a set $W$ satisfying $W = W + (a^k, b^\ell)$ for every $a, b \in \R$ a \textbf{slide-invariant set}.

\subsection{The lower bound}
\label{S:lower-bd}

The lower bound is easier so we start there.
The lower bound on the sizes of the sets $\operatorname{Star}(\bx, s; \by, t; W)$ takes a strong form, in that we will get an almost sure lower bound for all quintuples $(\bx, s; \by, t; W)$ simultaneously.

First, for $\bx \in \R^3_>, t \in \R$ define the sets
\begin{align*}
S^-(\bx, t) := \{\bw \in \R^3 : \cL(0, t-1; x_i, t) - w_i > \cL(0, t-1; x_2, t) - w_2, i = 1, 3\}, \\
S^+(\bx, t) := \{\bw \in \R^3 : \cL(x_i, t; 0, t +1) - w_i > \cL(x_2, t; 0, t + 1) - w_2, i = 1, 3\}.
\end{align*}
Also, for $\bx \in \R^k$ for $k = 1, 2$ let $S^\pm(\bx, t) = \R^k$.

\begin{prop}
	\label{P:star-lower-bd}
	Almost surely, the following assertion holds for every $k, \ell \in \{1, 2, 3\}, \bx \in \R^k_<, \by \in \R^\ell_<$, $s < t$, and $n \in \N$.
	
	There exist compact sets $K_1 = K_1(\bx, \by, n) \sset \R\X (-\infty, s), K_2 = K_2(\bx, \by, n) \sset \R\X (t, \infty)$ such that the following holds. Consider any slide-invariant affine set $W \sset \R^{k + \ell}$ with
	\begin{equation}
	\label{E:intersection-condition-star}
	W' =	W \cap [S^-(\bx, s) \X S^+(\by, t)] \cap (-n, n)^{k + \ell} \ne \emptyset,
	\end{equation}
	and let $g(W)$ be the smallest slide-invariant subset of $\R^{k+\ell}$ containing $W'$. Then the set $\operatorname{Star}(\bx, s; \by, t; g(W)) \cap (K_1 \X K_2)$ is non-empty and
	\begin{align*}
\dim_{1:2:3}[\operatorname{Star}(\bx, s; \by, t; g(W)) \cap (K_1 \X K_2)] \ge \dim W + 10 - \frac{k(k+1)}{2} - \frac{\ell(\ell+ 1)}{2}.
	\end{align*}
Here $\dim W$ is the linear dimension.
\end{prop}

The main part of the proof of Proposition \ref{P:star-lower-bd} is the following lemma.

\begin{lemma}\label{L:lower-bd-points'}
	Almost surely, the following statements hold.
	\begin{enumerate}[label=\arabic*.]
		\item Let $\bx \in \R^2_<$ and $s < t \in \R$, and define the function $F = F[\bx, s, t]:\R\to \R$ by
		$$
		F(y) = \cL(y, s; x_1, t) - \cL(y, s; x_2, t).
		$$
		Then $F$ is a continuous, non-increasing function with $\operatorname{Range}(F) = \R$. In particular, the set $F^{-1}(u)$ is a nonempty closed interval for all $u \in \R$. Letting $a < b$ denote the endpoints of this interval, there are forward geodesic stars from $(a, s)$ to $(\bx, t)$ and from $(b, s)$ to $(\bx, t)$.  
		\item 	Let $\bx \in \R^3_<, t \in \R$ and recall the definition of $S^-(\bx, t)$ above. 
	Then for any $\bu = (u_1, 0, u_3) \in S^-(\bx, t)$ there exists a point $p \in \R \X (t-1, t)$ for which there is a forwards geodesic star $\pi = (\pi_1, \pi_2, \pi_3)$ from $p$ to $(\bx, t)$ with
		$$
		(\|\pi_1\|_\cL - \|\pi_2\|_\cL, \|\pi_3\|_\cL - \|\pi_2\|_\cL) = (u_1, u_3).
		$$
	\end{enumerate}
\end{lemma}

The proof of Lemma \ref{L:lower-bd-points'} is essentially topological, using only planarity and properties of the coalescent geometry for $\cL$. We take inspiration from previous proofs.
A version of Lemma \ref{L:lower-bd-points'}.1 where the almost sure claim is only for fixed $s, t$ in $\R$ was shown as part of \cite[Theorem 1.9]{bates2019hausdorff} using related ideas. Moreover, similar topological ideas for both parts $1$ and $2$ have been used in \cite{gwynne2021geodesic} to prove the existence of certain geodesic networks in Liouville quantum gravity.

\begin{proof}
	\textbf{Proof of 1.} \qquad The fact that $F$ is a non-increasing function of $y$ follows from the quadrangle inequality \eqref{E:quadrangle}. The asymptotics in Theorem \ref{T:landscape-shape} guarantee that $F(y) \to -\infty$ as $y \to -\infty$ and $F(y) \to \infty$ as $y \to \infty$ so $\operatorname{Range}(f) = \R$. Fixing $u \in \R$ and letting $F^{-1}(u) = [a, b]$, it just remains to check that there are $2$-geodesic stars from $(a, s)$ and $(b, s)$ to $(\bx, t)$. 
	
	Suppose first that there is not a geodesic star from $(a, s)$ to $(\bx, t)$. Then the leftmost geodesics $\pi_1, \pi_2$ from $(a, s)$ to $(x_1, t), (x_2, t)$ satisfy $\pi_1(r) = \pi_2(r)$ for some $r \in (s, t)$. Since $\pi_1, \pi_2$ are leftmost we necessarily have $\pi_1 = \pi_2$ on $[s, r]$. Next, let $\pi^n$ be the sequence of leftmost geodesics from $(a-1/n, s)$ to $(x_1, t)$. The sequence $\pi^n$ is tight in the overlap metric (Proposition \ref{P:overlap-metric}) and any subsequential limit must be a geodesic from $(a, s)$ to $(x_1, t)$. The ordering on leftmost geodesics (Lemma \ref{L:regularity}.3) then implies that $\pi^n \to_o \pi$.
	Therefore for some $n$ we have $\pi^n(r) = \pi_1(r) = \pi_2(r)$ and $\pi^n|_{[r, t]} = \pi_1|_{[r, t]}$, and so 
	\begin{align*}
	\cL(a - 1/n, s; x_1, t) - \cL(a - 1/n, s; x_2, t) &\le \|\pi^n\|_\cL - \|\pi^n|_{[s, r]} \oplus \pi_2|_{[r, t]}\|_\cL \\
	&= \|\pi_1|_{[r, t]}\| - \|\pi_2|_{[r, t]}\| \\
	&= \cL(a, s; x_1, t) - \cL(a, s; x_2, t) = u.
	\end{align*} 
	Since $F$ is non-increasing, this must be an equality, contradicting that $F^{-1}(c) = [a, b]$. A similar argument shows that there is a geodesic star from $(b, s)$ to $(\bx, t)$. 
	
	\textbf{Proof of 2.} \qquad  Fix $\bu = (u_1, 0, u_3) \in S^-(\bx, t)$, and set $u_2 = 0$. Define three regions $R^\bu_1, R^\bu_2, R^\bu_3 \sset \R \X (-\infty, t)$ by 
	$$
	R^\bu_i = \{p \in \R \X (-\infty, t) : \cL(p; x_i, t) - u_i \ge \cL(p; x_j, t) - u_j, j \in \II{1, 3} \smin \{i\}\}.
	$$
	We can make a few observations about the sets $R^\bu_i$. 
	\begin{enumerate}[label=\arabic*.]
		\item By Theorem \ref{T:landscape-shape}, each $R^\bu_i$ contains the intersection of an open ball around the point $(x_i, t)$ with $\R \X (-\infty, t)$ and the set $R_2^\bu \cap (\R \X [t-1, t))$ is bounded.
		\item By continuity of $\cL$, each of the sets $R^\bu_i$ is closed in $ \R \X (-\infty, t)$.	
		\item By the quadrangle inequality \eqref{E:quadrangle} 
		and the definition of $S^-(\bx, t)$ we have that $R_2^\bu \cap (\R  \X \{t-1 \}) = \emptyset$.
	\end{enumerate}
	We claim that in this case the set of $P$ of points $p \in \R \X (t-1, t)$ where
	\begin{equation}
	\label{E:want-discrepancy-point}
	\cL(p; x_1, t) - u_1 = \cL(p; x_2, t) =  \cL(p; x_3, t) - u_3
	\end{equation}
	is nonempty. The three points above guarantee that $R_2^\bu \cap (\R \X [t-1, t))$ is a closed, bounded, nonempty subset which is disjoint from $\R  \X \{t-1 \}$. Therefore there is a minimal time $s_0 \in (t-1, t)$ for which
	$
	R_2^\bu \cap (\R  \X \{s_0\}) \ne \emptyset.
	$
	By the quadrangle inequality \eqref{E:quadrangle} and continuity of $\cL$, there are points $y_1 \le y_2 \in \R$ such that
	$$
	[y_1, y_2]  \X \{s_0\} = R_2^\bu \cap (\R \X \{s_0\}), \quad y_1 \in R_1^\bu, \quad y_2 \in R_3^\bu.
	$$
	Moreover, since $\R \X [t-1, s_0)$ does not intersect $R_2^\bu$, and $R^\bu_1 \cup R^\bu_3$ is a closed set containing the complement of $R_2^\bu$, we have
	$$
	[y_1, y_2]  \X \{s\} \sset R_1^\bu \cup R_3^\bu.
	$$
	Using this, along with the fact that $y_1 \in R_1^\bu, y_2 \in \R_3^\bu$ and $R^\bu_1$ and $R^\bu_3$ are closed yields a point $(y, s_0) \in  R_1^\bu \cap R_3^\bu$ with $y \in [y_1, y_2]$. This point is in $P$.
	
	Now, $P$ is closed since $\cL$ is continuous, and by point $1$ above it is separated from the line $\R\X \{t\}$. Therefore there is a point $p = (y, s)$ in $P$ with maximal time coordinate. If there is a geodesic star from $p$ to $(\bx, t)$, then by \eqref{E:want-discrepancy-point}, $p$ satisfies the conditions of part $2$. We will show that the maximality of $s$ guarantees the existence of such a geodesic star.
	
	Let $(\pi_1, \pi_2, \pi_3)$ be the three leftmost geodesics from $p$ to $(x_i, t), i = 1, 2, 3$. We have $\pi_1 \le \pi_2 \le \pi_3$ by Lemma \ref{L:regularity}.3 and the overlaps $O(\pi_i, \pi_j)$ are always intervals of the form $[s, r']$ for some $r' \ge s$ since otherwise we could perform a geodesic switching argument (Fact \ref{F:switching}) to contradict that one of the $\pi_i$ is leftmost.
	
	If there is an $r \in (s, t)$ with $\pi_1(r) = \pi_2(r) = \pi_3(r)$, then $(\pi_1(r), r) \in P$ contradicting the maximality of $s$. Therefore either $\pi_1$ and $\pi_2$ are disjoint or $\pi_2$ and $\pi_3$ are disjoint. Without loss of generality, assume $\pi_2$ and $\pi_3$ are disjoint, and that $O(\pi_1, \pi_2) = [s, r']$ with $s < r'$.  
	Now for every $r \in (s, r']$, define
	$$
	[a_r, b_r] = \{y' \in \R : \cL(y', r; x_1, t) - u_1 = \cL(y', r; x_2, t)\}.
	$$
	The fact that the right-hand side above is an interval follows from part $1$, and $[a_r, b_r]$ always contains $\pi_1(r)$.
	\begin{claim}
		$b_r \in [\pi_1(r), \pi_3(r))$. In particular, $b_r \to y$ as $r \cvgdown s$. 
	\end{claim}
	\begin{claimproof}
		Set$$
		d_r = \min \{x \in \R : (x, r) \in R^\bu_3\}.
		$$
		At the point $d_r$, we have that 
		$$
		(\cL(d_r, r; x_1, t) - u_1) \vee \cL(d_r, r; x_2, t) =  \cL(d_r, r; x_3, t) - u_3.
		$$
		Therefore if $d_r \in [a_r, b_r]$ then the point $(d_r, r) \in P$ which contradicts the maximality of $s$. On the other hand, if $d_r \le \pi_1(r)$, then
		\begin{equation*}
		\begin{split}
		\cL(p; \bar \pi_1(r)) + \cL(\bar \pi_1(r), x_3, t) &\ge \cL(p; \bar \pi_1(r)) + \cL(\bar \pi_1(r), x_2, t) + u_2 \\
		&=\cL(p; x_2, t) + u_2 \\
		&=\cL(p; x_3, t).
		\end{split}
		\end{equation*}
		Here the inequality uses that $d_r \le \pi_1(r)$, the first equality uses that $\pi_2$ is a geodesic from $p$ to $(x_2, s)$ going through the point $\bar \pi_1(r)$, and the second equality uses that $p \in P$. This calculation implies that there is a geodesic from $p$ to $(x_3, s)$ going through the point $\bar \pi_1(r)$. Since $\pi_3(r) > \pi_1(r)$, this contradicts that $\pi_3$ is the leftmost geodesic from $p$ to $(x_3, s)$. Hence $d_r > \pi_1(r)$. 
		
		Next, if $\pi_3(r) < d_r$, then by similar reasoning we have
		\begin{equation*}
		\begin{split}
		\cL(p; \bar \pi_3(r)) &+ (\cL(\bar \pi_3(r); x_1, t) - u_1) \vee \cL(\bar \pi_3(r); x_2, t) \\
		&> \cL(p; \bar \pi_3(r)) + \cL(\bar \pi_3(r); x_3, t) - u_3 \\
		&=\cL(p; x_3, s) - u_3 \\
		&=(\cL(p; x_1, t) - u_1) \vee \cL(p; x_2, t),
		\end{split}
		\end{equation*}
		which contradicts the triangle inequality for $\cL$. Therefore $d_r \in (\pi_1(r), \pi_3(r)] \cap [a_r, b_r]^c$. Since $\pi_1(r) \in [a_r, b_r]$ this implies $b_r < d_r \le \pi_3(r)$, completing the proof of the claim.
	\end{claimproof}
	
	Now, by part $1$ of the lemma, at the point $b_r$, there are disjoint geodesics $\pi^{1, r}, \pi^{2, r}$ from $(b_r, r)$ to $(x_1, t)$ and $(x_2, t)$. Since $(b_r, r) \to p$, these geodesic sequences are precompact in the overlap topology with limits that are geodesics from $p$ to $(x_1, t)$ and $(x_2, t)$ (Proposition \ref{P:overlap-metric}). Therefore since disjointness is preserved by overlap convergence there exist disjoint geodesics $\tilde \pi_1, \tilde \pi_2$ from $p$ to $(x_1, t), (x_2, t)$. 
	
	If $\tilde \pi_2$ is disjoint from the rightmost geodesic $\tau_3$ from $p$ to $(x_3, t)$, then $(\tilde \pi_1, \tilde \pi_2, \tau_3)$ gives a desired geodesic star. If not, then letting $\tau_2$ be the rightmost geodesic from $p$ to $(x_2, t)$, there exists an interval $[s, r''] \sset [s, r']$ along which $\tau_2 = \tau_3$. For every $r \in (s, r'']$, define
	$$
	[a_r', b_r'] = \{y' \in \R :  \cL(y', r; x_2, t) = \cL(y', r; x_3, t) - u_3 \},
	$$
	and note that $\pi_3(r) \in [a_r', b_r']$.
	Now, if $a_r' \le b_r$ for any $r \in (s, r'']$, then since $b_r < \pi_3(r)$, we have that $b_r \in [a_r', b_r']$ and so $(b_r, r)$ is necessarily in $P$. This contradicts the maximality of $s$. Therefore we may assume $b_r < a_r'$ for all $r \in (s, r'']$. 
	
	Again using part $1$ of the lemma, the leftmost geodesic from $(a_r', r)$ to $(x_2, t)$ is disjoint the rightmost geodesic from $(a_r', r)$ to $(x_3, t)$. Therefore since $b_r < a_r'$, the rightmost geodesic from $(b_r, r)$ to $(x_2, t)$ is disjoint from $\tau_3$; this uses the monotonicity in Lemma \ref{L:regularity}.4.
	Therefore $\pi^{2, r}, \pi^{1, r}$ are disjoint from $\tau_3$ for all $r$, and since overlap convergence preserves disjointness, so are $\tilde \pi_1, \tilde \pi_2$. This is a contradiction. 
\end{proof}

From Lemma \ref{L:lower-bd-points'} and the modulus of continuity for $\cL$ we can lower bound the Hausdorff dimension of the sets $\operatorname{Star}(\bx, s; \by, t; W)$ and prove Proposition \ref{P:star-lower-bd}.
We require a simple lemma relating Hausdorff dimension and H\"older continuous maps. This lemma is well-known, but we could not locate a reference for the exact generality we need here, so include a short proof. Here we let $\dim_d$ denote Hausdorff dimension with respect to a metric $d$.

\begin{lemma}
	\label{L:Hausdorff-Holder}
	Let $(X, d_1), (Y, d_2)$ be two metric spaces, each of which can be written as a countable union of compact sets. Let $f:X \to Y$ be a function and suppose that $f$ is locally H\"older-$\al$ continuous (i.e. H\"older-$\al$ continuous on every compact subset of $X$). Then for any set $U \sset X$,
	\begin{equation}
	\label{E:fU-d1}
	\dim_{d_2} (f(U)) \le \al^{-1}\dim_{d_1}(U).
	\end{equation}
\end{lemma}

\begin{proof}
	Let $X = \bigcup_{n=1}^\infty K_n$, where the $K_n$ are compact. By countable stability of Hausdorff dimension, it suffices to prove \eqref{E:fU-d1} for $U \cap K_n$ for all $n$. Now suppose that $\dim_{d_1}(U \cap K_n) < \beta < \infty$. Then there exists $m > 0$ such that for all $\ep > 0$, there exists a cover $V_i, i \in \N$ of $U \cap K_n$ with $\operatorname{diam}(V_i) < \ep$ and $\sum_i \operatorname{diam}(V_i)^\be < m$. Then $f(V_i), i \in \N$ is a cover for $f(U \cap K_n)$, and letting $c_n$ be the H\"older-$\al$ constant for $f$ on $K_n$, we have that
	$$
	\operatorname{diam}(f(V_i)) \le c_n \operatorname{diam}(V_i)^\al \le c_n \ep^{\al}, \quad \sum_i \operatorname{diam}(f(V_i))^{\be/\al} \le c_n^{\be/\al} \sum_i \operatorname{diam}(V_i)^\be < c_n^{\be/\al} m
	$$
	and so $\dim_{d_2} (f(U \cap K_n)) \le \be/\al$.
\end{proof}

\begin{proof}[Proof of Proposition \ref{P:star-lower-bd}] 
	We prove Proposition \ref{P:star-lower-bd} by finding a locally H\"older continuous map relating the star-sets and the set $W$, and then applying Lemma \ref{L:Hausdorff-Holder}. Let $(\bx, s)  \in \R^k \X \R, (\by, t) \in \R^\ell \X \R$. The definition of the map depends on the choice of $k, \ell$, and by time-reversal symmetry of $\cL$ (Lemma \ref{L:invariance}.3) we may assume $k \le \ell$. First, define $L_3:\R^3 \to \R^2, L_2:\R^2 \to \R$ by
	$$
	L_3(\bw) = (w_1 - w_2, w_3 - w_2), \qquad L_2(\bw) = w_1 - w_2. 
	$$
	Next we define the compact set $K_1$. If $k = 1$, set $K_1 = [-1, 1] \X [s-1, s-1/2]$. If $k = 2$, then by Theorem \ref{T:landscape-shape} there is a random constant $C > 0$ such that
	\begin{equation}
	\label{E:px1}
	|\cL(p; x_1, s) - \cL(p; x_2, s)| > 2n \qquad \text{ for all } p \in [-C, C]^c \X  [s-1, s).
	\end{equation}
	In this case we set $K_1 =  [-C, C] \X [s-1, s-1/2]$. Finally, when $k = 3$ we keep the definition of $C$ above. We can also find a random $\ep > 0$ such that whenever $p \in [-C, C] \X  [s-\ep, s)$ we have
	\begin{equation}
	\label{E:px2x3}
|\cL(p; x_1, s) - \cL(p; x_2, s)| \vee |\cL(p; x_3, s) - \cL(p; x_2, s)| > 2n.
	\end{equation}
	In this final case, set $K_2 = [-C, C] \X [s-1, s-\ep].$ We can define $K_2$ analogously.
	
	Finally, define maps $f_{k, \ell}:K_1 \X K_2 \to \R^4$ for the six choices of $k \le \ell$ as follows. For each of these maps, we equip the domain with the $d_{1:2:3}$ metric and the range with a metric $e_{k, \ell}$ defined below.
	\begin{itemize}[nosep]
		\item $f_{1, 1}$ is the identity and $e_{1, 1} = d_{1:2:3}$.
		\item $f_{1, 2}(z, r; z', r') = (z, r, r', L_2 \circ W^+_{\by, t}(z', r'))$ and 
		$$
		e_{1, 2}(\bu, \bv) = |u_1 - v_1|^{1/2} + |u_2-v_2|^{1/3} + |u_3 - v_3|^{1/3} + |u_4 - v_4|.
		$$
		\item $f_{1, 3}(z, r; z', r') = (z, r, L_3 \circ W^+_{\by, t}(z', r'))$ and 
		$$
		e_{1, 3}(\bu, \bv) = |u_1 - v_1|^{1/2} + |u_2-v_2|^{1/3} + |u_3 - v_3| + |u_4 - v_4|.
		$$
		\item $f_{2, 2}(z, r; z', r') = (r, r', L_2 \circ W^-_{\bx, s}(z, r), L_2 \circ W^+_{\by, t}(z', r'))$ and
		$$
		e_{2, 2}(\bu, \bv) = |u_1 - v_1|^{1/3} + |u_2 - v_2|^{1/3} + |u_3 - v_3| + |u_4 - v_4|.
		$$
		\item $f_{2, 3}(z, r; z', r') = (r, L_2 \circ W^-_{\bx, s}(z, r); L_3 \circ W^+_{\by, t}(z', r'))$ and
		$$
		e_{2, 3}(\bu, \bv) = |u_1 - v_1|^{1/3} + |u_2 - v_2| + |u_3 - v_3| + |u_4 - v_4|.
		$$
		\item $f_{3, 3}(z, r; z', r') = (L_3 \circ W^-_{\bx, s}(z, r), L_3 \circ W^+_{\by, t}(z', r'))$ and let $e_{3, 3}$ be the usual Euclidean metric.
	\end{itemize}
	By the modulus of continuity for $\cL$ (Proposition \ref{P:mod-land-i}) all of the maps $f_{k, \ell}$ are  locally H\"older-$(1-\ep)$ continuous for all $\ep > 0$ with the prescribed codomain metrics. Now, let $L_{k, \ell} := (L_k, L_\ell):\R^{k + \ell}\to \R^{k-1} \X \R^{\ell - 1}$. We claim that
	\begin{equation}
	\label{E:f-image}
	B_{k, \ell} \X L_{k, \ell} (g(W)) \sset f_{k, \ell}(\operatorname{Star}(\bx, s; \by, t; g(W)))
	\end{equation}
	where $B_{k, \ell}$ is a $6-k-\ell$ dimensional product of positive-length intervals.
	It suffices to check \eqref{E:f-image} when $W = L_{k, \ell}^{-1}(\ba, \bb)$ for some $\ba \in \R^{k-1}, \bb \in \R^{\ell-1}$, since every other slide-invariant set $W$ is a disjoint union of sets of this form. We may assume $g(W) \ne \emptyset$ since otherwise the claim is trivial. In this case $L_{k, \ell}(g(W)) = (\ba, \bb)$,
	\begin{align*}
	\operatorname{Star}(\bx, s; \by, t; W) = \operatorname{Star}(\bx, s; L_k^{-1}(\ba)) \X \operatorname{Star}_R(\by, t; L_\ell^{-1}(\bb)),
	\end{align*}
	and $L_k^{-1}(\ba) \cap (-n, n)^k \ne \emptyset, L^{-1}_\ell(\bb) \cap (-n, n)^\ell \ne \emptyset$.
	
	Now, if $k = 1$, then $L_k^{-1}(\ba) = \R$ and so $\operatorname{Star}(\bx, s; L_k^{-1}(\ba)) \cap K_1 = K_1$ since there are geodesics between any pair of points. If $k = 2$, then by Lemma \ref{L:lower-bd-points'}.1, $\operatorname{Star}(\bx, s; L_k^{-1}(\ba))$ contains a point of the form $(y, r)$ for all $r \in \R$. Moreover, since $L_2^{-1}(\ba) \cap (-n, n)^2 \ne \emptyset$, if $(y, r) \in \operatorname{Star}(\bx, s; L_k^{-1}(\ba))$ for $r \in [s-1, s- 1/2]$ then by \eqref{E:px1} we must have $y \in [-C, C]$. Therefore $\operatorname{Star}(\bx, s; L_k^{-1}(\ba))$ contains a point in $K_1 \cap (\R \X \{r\})$ for all $r \in [s-1, s- 1/2]$. If $k = 3$, then by Lemma \ref{L:lower-bd-points'}.2 and the fact that $g(W) \ne \emptyset$, the set $\operatorname{Star}(\bx, s; L_k^{-1}(\ba))$ contains a point $p \in \R \X (s-1, s)$. Moreover, since $L_3^{-1}(\ba) \cap (-n, n)^3 \ne \emptyset$, by \eqref{E:px1} and \eqref{E:px2x3} we must have $p \in K_1$. 
This analysis, along with similar reasoning applied to  $\operatorname{Star}_R(\by, t; L_\ell^{-1}(\bb)) \cap K_2$, yields \eqref{E:f-image} for $W = L_k^{-1}(\ba) \X L^{-1}_\ell(\bb)$.

	The fact that the left-hand side of \eqref{E:f-image} is non-empty implies that $\operatorname{Star}(\bx, s; \by, t; g(W))$ is non-empty. Moreover, Lemma \ref{L:Hausdorff-Holder} and the fact that $f_{k, \ell}$ is locally H\"older-$(1-\ep)$ continuous for all $\ep > 0$ implies that
	$$ 
	\dim_{1:2:3} \operatorname{Star}(\bx, s; \by, t; g(W)) \ge \dim_{e_{k, \ell}}(B_{k, \ell} \X L_{k, \ell} (g(W))).
	$$
	It is straightforward to calculate the Hausdorff dimension of the right-hand side above. Indeed, since $W$ is a slide-invariant affine set, $L_{k, \ell}(g(W))$ is a non-empty intersection of a $\dim W - 2$ dimensional affine set and an open set. The set $B_{k, \ell}$ is simply a product of positive-length intervals. Going through the cases for each $k, \ell$, we get that 
	$$
	\dim_{e_{k, \ell}}(B_{k, \ell} \X L_{k, \ell} (g(W))) \ge 12 - \frac{k(k+1)}{2} - \frac{\ell(\ell+1)}{2} + \dim W - 2,
	$$
	yielding the bound in the proposition.
\end{proof}

\subsection{Upper bounds on dimension}
\label{S:upper-bounds}

The upper bound on the dimension of $\operatorname{Star}(\bx, s; \by, t; W)$ is more difficult, and leans heavily on an analysis of a last passage problem in the Airy line ensemble arising via Proposition \ref{P:high-paths-B}. The upper bound we achieve is not as strong as the lower bound: it only holds almost surely for a fixed quintuple $(\bx, s; \by, t; W)$, rather than for all quintuples simultaneously.

\begin{prop}
	\label{P:weighted-stars-upper-bound}
	Let $s < t$, let $\ell, k \in \{1, 2, 3, 4\}$, let $\bx \in \R^k_<, \by \in \R^\ell_<$, and let $W \sset \R^{k+\ell}$ be any affine set. Then setting
	$$
	d(k, \ell, W) = \dim W + 10 - \frac{k(k+1)}{2} - \frac{\ell(\ell+ 1)}{2}
	$$
	almost surely we have 
	$$
	\dim_{1:2:3}[\operatorname{Star}(\bx, s; \by, t; W)] \le d(k, \ell, W) \vee 0.
	$$
	Moreover if $d(k, \ell, W) < 0$ then the set $\operatorname{Star}(\bx, s; \by, t; W)$ is almost surely empty.
\end{prop}

Before proceeding with the proof of Proposition \ref{P:weighted-stars-upper-bound} we show how when combined with the lower bound in Proposition \ref{P:star-lower-bd}, it implies Theorem \ref{T:star-computation}.

\begin{proof}[Proof of Theorem \ref{T:star-computation}]
The $k = 1$ case is immediate since $\operatorname{Star}_1 = \R^2$ and $\dim_{1:2:3}(\R^2) = 5$. Let $k \in \{2, 3, 4\}$. Let $\operatorname{Star}_{k, 1}$ be the set of all points $p$ that start a forwards $k$-geodesic star. By time-reversal symmetry of $\cL$, it suffices to prove the theorem with $\operatorname{Star}_{k, 1}$ in place of the larger set $\operatorname{Star}_k$.

First, by Lemma \ref{L:overlap-at-endpoints}, if $p \in \operatorname{Star}(\bx, t, \R^k)$ then $p \in \operatorname{Star}(\bx', t', \R^k)$ for some rational $\bx', t'$ close to $\bx, t$. Therefore
$$
\operatorname{Star}_{k, 1} = \bigcup_{(\bx, t) \in \Q^k_< \X \Q} \operatorname{Star}(\bx, t, \R^k).
$$
Now, almost surely for all $(\bx, t) \in \Q^k_< \X \Q$ we have
$
\operatorname{Star}(\bx, t; 0, t + 1; \R^{k+1}) = \operatorname{Star}(\bx, t, \R^k) \X \R \X (t+1, \infty).
$
If $k = 4$, then Proposition \ref{P:weighted-stars-upper-bound} implies that almost surely 
$$
\dim_{1:2:3}[\operatorname{Star}(\bx, t, \R^k) \X \R \X (t+1, \infty)] \le 4.
$$ This can only hold if $\operatorname{Star}(\bx, t, \R^k)$ is empty, and so almost surely $\operatorname{Star}_{4, 1}$ is also empty. 
If $k = 2, 3$, then Proposition \ref{P:star-lower-bd} and Proposition \ref{P:weighted-stars-upper-bound} together imply that 
$$
\dim_{1:2:3}[\operatorname{Star}(\bx, t, \R^k) \X \R \X (t+1, \infty)] = 10 - \frac{k(k-1)}{2}.
$$
Now, standard results about the Hausdorff dimension of Cartesian products imply that the Hausdorff dimension of a Cartesian product $A \X B$ equals the sum of the Hausdorff dimensions when the Hausdorff dimension of $B$ coincides with its upper packing dimension (e.g. see \cite[Corollary 8.11]{mattila1999geometry} for the result in Euclidean space; the proof goes through verbatim for dimensions defined with the metric $d_{1:2:3}$). Therefore
\begin{align*}
\dim_{1:2:3}[\operatorname{Star}(\bx, t, \R^k) \X \R \X (t+1, \infty)] &= \dim_{1:2:3}[\operatorname{Star}(\bx, t, \R^k)] + \dim_{1:2:3}[\R \X (t+1, \infty)] \\
&= \dim_{1:2:3}[\operatorname{Star}(\bx, t, \R^k)] + 5,
\end{align*}
which implies that $\dim_{1:2:3}(\operatorname{Star}_{2, 1}) = 4, \dim_{1:2:3}(\operatorname{Star}_{3, 1}) = 2$.
\end{proof}

The proof of Proposition \ref{P:weighted-stars-upper-bound} is based on an application of Theorem \ref{T:iso-multiples} and Proposition \ref{P:high-paths-B}. These results are connected to the study of geodesic stars by the following lemma. 

For this lemma and throughout this section, for $k \in \N$ let $T_k = \{(i, j) \in \N^2: i + j \le k + 1 \}$. We say that $\bw \in \R^{T_k}$ is a $k$-level \textbf{Gelfand-Tsetlin pattern} if it satisfies the interlacing inequalities 
$$
w_{i, j} \le w_{i + 1, j} \le w_{i, j+1}
$$
for all $i + j \le k$, and let $\operatorname{GT}_k$ be the space of $k$-level Gelfand-Tsetlin patterns. For $\bx \in \R^{k+1}_\le$ we let $\operatorname{GT}_k(\bx)$ be the set of all $\bw \in \operatorname{GT}_k$ such that $(\bw, \bx) \in \operatorname{GT}_{k+1}$ where $(\bw, \bx)_{i, j} = w_{i - 1, j}$ for $i \ge 2$ and $(\bw, \bx)_{1, j} = x_j$.

\begin{lemma}
	\label{L:jump-lemma}
	For every point $p = (x, t) \in \R^2$ and $s > t$ define a sequence of continuous functions $\cL^{p, s}_1, \cL^{p, s}_2\dots,$ from $\R \to \R$ by the formula
	$$
	\sum_{i=1}^k \cL^{p, s}_i(y) = \cL(p^k; y^k, s).
	$$ 
	Then almost surely, for all $k \in \{2, 3, 4\}, \by \in \R^k_<, s \in \R$, and $p = (x, t) \in \R \X (-\infty, s)$, the following two statements are equivalent:
	\begin{enumerate}[label=\Alph*.]
		\item There is a $k$-geodesic star from $p$ to $(\by, s)$ and $W^-_{\by, s}(p) = \bw \in \R^k$.
		\item We have $\cL^{p, s}_1(y_i) = w_i$ for all $i \in \II{1, k}$ and there exists $\bz \in \operatorname{GT}_{k-1}(\by)$ with $\cL^{p, s}_i(z_{i, j}) = \cL^{p, s}_{i+1}(z_{i, j})$ for all $(i, j) \in T_k$.
	\end{enumerate}
\end{lemma}

\begin{proof}
First, if there is a $k$-geodesic star $\pi$ from $p = (x, t)$ to $(\by, s)$ with $w(\pi) = \bw \in \R^k$, then 
\begin{equation}
\label{E:cLxk}
\cL(x^k, t; \by, s) = \sum_{i=1}^k \cL(x, t; y_i, s).
\end{equation}
Moreover, as discussed after Definition \ref{D:ext-land}, almost surely the extended landscape value $\cL(x^k, t; \by, s)$ is attained by a disjoint $k$-tuple for all $p = (x, t), k, \by, s$. Therefore almost surely for all choices of $x, k, t, \by, s$, statement $A$ holds if and only if both \eqref{E:cLxk} holds and 
$
\cL^{p, s}_1(y_i) = \cL(p; y_i, s) = w_i
$
for all $i \in \II{1, k}$.

On the other hand, almost surely we have $\cL^{p, s}_1 \ge \cL^{p, s}_2 \ge \dots$ for all $p, s$. Indeed, this holds for a fixed $p, s$ since in this case $\cL^{p, s}$ is a rescaled parabolic Airy line ensemble (Theorem \ref{T:iso-multiples}.1). The extension to all $p, s$ follows by continuity. Therefore almost surely for all $p, s, k, \by$, there exists $\bz \in \operatorname{GT}_{k-1}(\by)$ with $\cL^{p, s}_i(z_{i, j}) = \cL^{p, s}_{i+1}(z_{i, j})$ for all $(i, j) \in T_{k-1}$ if and only if
\begin{equation}
\label{E:GT-pattern}
\cL^{p, s}[((y_1-1)^k, k) \to (\by, 1)] + \sum_{i=1}^k \cL_i^{p, s}(y_1-1) = \sum_{i=1}^k \cL^{p, s}_1(y_i),
\end{equation}
see Figure \ref{fig:pinches}. Therefore to prove the lemma it suffices to show that almost surely, for all $k, \by, s, p$ the statements \eqref{E:GT-pattern} and \eqref{E:cLxk} are equivalent. The right-hand sides of \eqref{E:cLxk} and \eqref{E:GT-pattern} are equal by definition. Also, for a fixed choice of $k, \by, s, p$ the left-hand sides are equal almost surely by Theorem \ref{T:iso-multiples}.2 and Proposition \ref{P:high-paths-B}. This extends to all rational $k, \by, s, p$ by countable additivity, and then almost surely to all $k, \by, s, p$ since the left-hand sides of both \eqref{E:cLxk} and \eqref{E:GT-pattern} are continuous by continuity of the extended landscape. Therefore almost surely, for all $k, \by, s, p$ the statements \eqref{E:GT-pattern} and \eqref{E:cLxk} are equivalent.
\end{proof}

\begin{figure}
\centering
\includegraphics[scale=0.8]{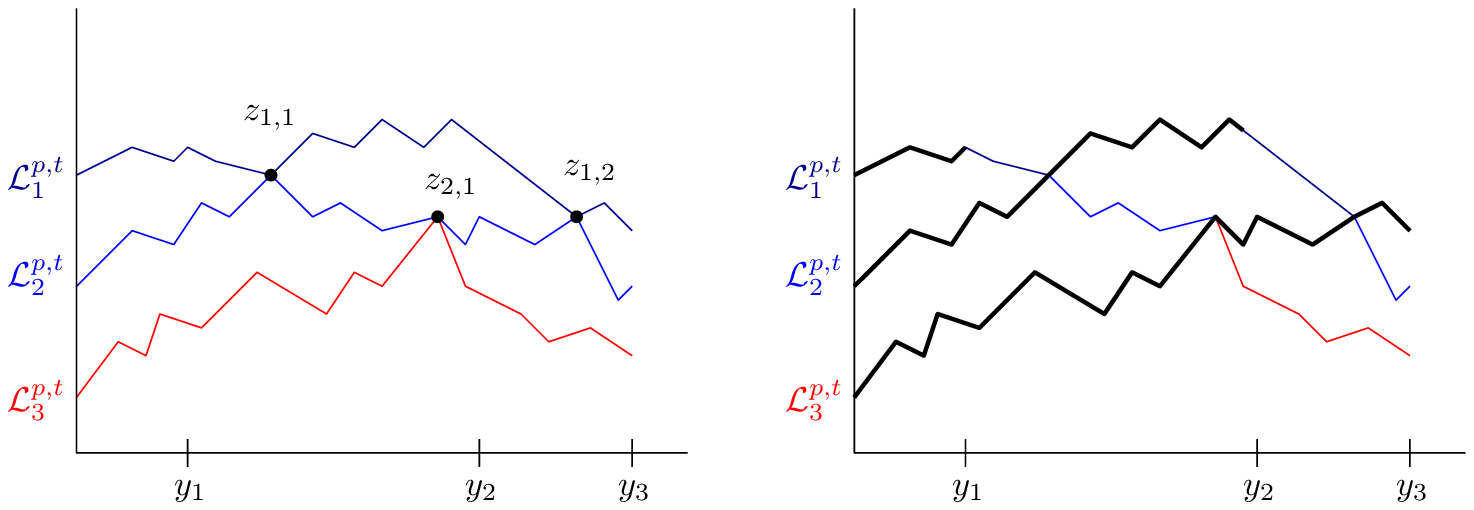}
\caption{An example of an ensemble $\cL^{p, s}$ where \eqref{E:GT-pattern} holds with $k = 3$. The left-hand side illustrates the Gelfand-Tsetlin pattern $\bz$ and the right-hand side illustrates the disjoint optimizer for $\cL^{p, s}[((y_1-1)^k, k) \to (\by, 1)]$.}
\label{fig:pinches}
\end{figure}

To upper bound the dimension of $\operatorname{Star}(\bx, s; \by, t; W)$, by Lemma \ref{L:jump-lemma} and a discretization argument it will be enough to upper bound the probability that at a fixed $p$, the claim in Lemma \ref{L:jump-lemma}.B approximately holds. For fixed $p, s$, Theorem \ref{T:iso-multiples}.1 gives that $\cL^{p, s}$ is a rescaled parabolic Airy line ensemble. Moreover, Theorem \ref{T:resampling-candidate} implies that on any compact set, $\cL^{p, s}$ has bounded Radon-Nikodym derivative with respect to a collection of independent Brownian bridges, suggesting we first try to analyze the probability that Lemma \ref{L:jump-lemma}.B holds for such bridges. This is the goal of Lemma \ref{L:ep-0100}, Lemma \ref{L:BB-nonint}, and Corollary \ref{C:Brownian-weighted-star}. Note that all Brownian bridges  in this section variance $2$.

For this next suite of lemmas, we say that $B = (B_1, \dots, B_k) \in \sC^k([a, b])$ is a $k$-tuple of independent Brownian bridges from $(a, \bx)$ to $(b, \by)$ if each $B_i:[a, b] \to \R$ is an independent Brownian bridge with $B_i(a) = x_i, B_i(b) = y_i$. We write $\P_{a, b}(\bx, \by; \cdot)$ for the law of $B$.  We define the set $\NI \sset \cC^k([a, b])$ to be the set of all functions $f$ with $f_1 > f_2 > \dots f_k$.

Throughout the remainder of this section, for a Gelfand-Tsetlin pattern $\bw \in \operatorname{GT}_k$ we let $\# \bw$ be the number of distinct values in $\bw$ and set $$
\Theta(\bw) = \min \{|w_{i, j} - w_{i', j'}| : w_{i, j} \ne w_{i', j'}, (i, j), (i', j) \in T_k\}.
$$ 

\begin{lemma}
	\label{L:ep-0100}
	Fix $k \in \{2, 3, 4\}, m \ge 1, \ep \in (0, 1/100)$. Let $B$ be a $k$-tuple of Brownian bridges from $(0, \bx)$ to $(1, \by)$ for vectors $\bx, \by \in [-m, m]^k_>$, and let $\bw$ be a $(k-1)$-level Gelfand-Tsetlin pattern satisfying $\Theta(\bw) \ge \ep$ and with $w_{i, j} \in (\ep, 1- \ep)$ for all $(i, j) \in T_{k-1}$. 
	
	For $\de > 0$, let $V_{\bx, \by}(\bw, \de)$ be the event where
	$$
	B \in \NI, \quad  |B_i(w_{i, j}) - B_{i+1}(w_{i, j})| \le \de \quad \text{ for all } (i, j) \in T_k.
	$$
	Then
	$$
	\P(V_{\bx, \by}(\bw, \de)) \le e^{c_\ep m}  \de^{k(k-1)/2 + 2 \# \bw}.
	$$ 
\end{lemma}

The main step in the proof of Lemma \ref{L:ep-0100} is the following basic bound on the probability that a collection of Brownian bridges $B$ is nonintersecting (i.e. $B \in \NI$).

\begin{lemma}
	\label{L:BB-nonint}
	Fix $t \in (0, 1), k \in \{2, 3, \dots\}, m \ge 1$, and let $\bx, \by \in [-m, m]^k_>$. Fix $\de \in (0, 1)$, and let $G_\de(\bx)$ be the graph on $\II{1,k}$ where $(i, j)$ is an edge exactly when $|x_i - x_j| < \de$. Define
	$$
	E_\de(\bx) = \# \{(i, j) \in \II{1, k}^2: \text{$i < j$ and $i, j$ are in the same component of } G_\de(\bx)\}.
	$$
	Similarly define $G_\de(\by), E_\de(\by),$. Then
	$$
	\P_{0, t}(\bx, \by; \NI) \le \de^{E_\de(\bx) + E_\de(\by)}\exp(c m k^4/t),
	$$
	where $c > 0$ is an absolute constant.
\end{lemma}

\begin{proof}
	First, we have
	\begin{align}
	\label{E:integrated-middle}
	\P_{0, t}(\bx, \by; \NI) &= \int_{\R^k_>} \P_{0, t/2}(\bx, \bz; \NI) \P_{t/2, t}(\bz, \by; \NI) \rho(\bz) d \bz,
	\end{align}
	where 
	$$
	\rho(\bz) = \frac{1}{(\pi t)^{k/2}} \exp \lf(-\frac{\|\bz - (\by + \bx)/2\|^2}{t} \rg)
	$$ 
	is the Gaussian density. Next, we estimate $\P_{0, t/2}(\bx, \bz; \NI)$ using the Karlin-McGregor formula. Consider a component $H$ of $G_\de(\bx)$ and note that $H$ is an integer interval $H = \II{a_H, b_H}$. Let $\bx^H = (x_{a_H}, \dots, x_{b_H})$. We can write
	$$
	\P_{0, t/2}(\bx, \bz; \NI) \le \prod_H \P_{0, t/2}(\bx^H, \bz^H; \NI).
	$$
	where the product is over components of $G_\de(\bx)$. Next, let $\iota^H = (|H| - 1, |H| - 2, \dots, 0) \in \R^{|H|}$ and let $\tilde \bz^H = (z_{a_H} - z_{b_H},z_{a_H + 1} - z_{b_H}, \dots, 0) \in \R^{|H|}$. Since the spacing between the coordinates of $\de \iota^H$ is greater than the spacing for $\bx^H$ by the definition of $G_\de(\bx)$, and Brownian bridge law commutes with linear shifts, we have
	$$
	\P_{0, t/2}(\bx^H, \bz^H; \NI) \le \P_{0, t/2}(\de \iota^H, \bz^H; \NI) = \P_{0, t/2}(\de \iota^H, \tilde \bz^H; \NI).
	$$
	We focus on bounding the latter probability. By the Karlin-McGregor formula (see p. 48 in \cite{hammond2016brownian} for the form used below) we have
	\begin{align*}
	\P_{0, t/2}(\de \iota_H, \tilde \bz^H; \NI) &= \exp \Big(- t^{-1}\sum_{i=1}^{|H|} \de \iota^H_i \tilde z^H_i\Big) \det\Big(e^{\de \iota^H_i \tilde z^H_j /t}\Big)_{1 \le i, j \le |H|} \\
	&\le \prod_{1 \le i < j \le |H|} (e^{\de \tilde z^H_j/t} - e^{\de \tilde z^H_i/t}) \\
	&\le \de^{|H| (|H| - 1)/2} \exp \lf(\frac{c |H| (|H| - 1)\tilde z^H_1}{t} \rg)
	\end{align*}
	for some constant $c > 0$.
	In the first inequality, we have used that the determinant is a Vandermonde by the form of $\iota_H$, and that $\tilde z_i \ge 0$ for all $i$ to bound the first term by $1$. Now, using that $\sum_H |H| (|H| - 1)/2 = E_\de(\bx)$ and that $\tilde z^H_1 \le z_1 - z_k$ for any $H$, we get that 
	$$
	\P_{0, t/2}(\bx, \bz; \NI) \le \de^{E_\de(\bx)} \exp \lf( \frac{ck^2(z_1 - z_k)}{t}  \rg).
	$$
	A similar bound on $\P_{t/2, t}(\bz, \by; \NI)$ implies that \eqref{E:integrated-middle} is bounded above by
	$$
	\de^{E_\de(\bx) + E_\de(\bx)} \int_{\R^k_> } e^{\frac{2ck^2|z_1 - z_k|}{t}} \rho(\bz) dz.
	$$
	The integral above converges, and since $\bx, \by \in [-m, m]^k_>$, it is bounded above by $\exp(c' m k^4/t)$ for some $c' > 0$, giving the desired result.
\end{proof}

\begin{proof}[Proof of Lemma \ref{L:ep-0100}]
	To simplify notation, we write $\ell = \# \bw$ throughout.
	First, let $\bv \in \R^\ell_<$ denote the unique vector whose coordinates are equal to the entries of $\bw$ and let $v_0 = 0, v_{\ell+1} = 1$. 
	Letting $\ba^0 := \bx,  \ba^{\ell+1} :=\by,$ we can write
	\begin{align}
	\label{E:P01-NI}
	\P(V_{\bx, \by}(\bw, \de)) = \int_{\Delta} \prod_{i=0}^\ell \P_{v_i, v_{i+1}}(\ba^i, \ba^{i+1}; \NI) \rho(\ba^1, \dots, \ba^\ell) d \ba^1 \cdots d \ba^\ell.  
	\end{align}
	where
	$$
	\rho(\ba^1, \dots, \ba^\ell) = \sqrt{4\pi}  \exp \lf( \frac{\|\ba^0 - \ba^{\ell+1}\|^2}{4}\rg)\prod_{i=1}^k \frac{1}{\sqrt{4 \pi (v_{i+1}- v_i)}} \exp \lf( - \frac{\|\ba^i - \ba^{i+1}\|^2}{4[v_i- v_{i+1}]}\rg)
	$$
	is the Gaussian density.
	Here the region $\Delta = \prod_{i=1}^\ell \Delta_i$, and $\Delta_i$ is the set of points $\ba^i \in \R^k_>$ where $a^i_{i'} - a^i_{i' + 1} \le \de$ whenever $v_i = w_{i', j'}$ for some $(i',j') \in T_{k-1}$. Next we use Lemma \ref{L:BB-nonint} to estimate $\P_{v_i, v_{i+1}}(\ba^i, \ba^{i+1}; \NI)$. First, 
	for every $i \in \II{1, \ell}$, let 
	$$
	n_i = \# \{i' : \text{ $v_i = w_{i', j'}$ for some $j'$ with $(i', j') \in T_{k-1}$}\}.
	$$
	Also let $n_0 = n_{\ell+1} = 0$.
	Observe that if $\ba^i \in \Delta_i$, then with notation as in Lemma \ref{L:BB-nonint} we have that $E_\de(\ba^i) \ge n_i$, and so by Lemma \ref{L:BB-nonint} we have
	\begin{align*}
	\P_{v_i, v_{i+1}}(\ba^i, \ba^{i+1}; \NI) \le \de^{n_i + n_{i+1}} \exp (c k^4 \|\ba\|_\infty/(v_{i+1} - v_i)) \le \de^{n_i + n_{i+1}} \exp (c_\ep \|\ba\|_\infty).
	\end{align*}
	Here $\|\ba\|_\infty = \|(\ba^0, \dots, \ba^{\ell+1})\|_\infty$, and the second inequality follows since $k \in \{2, 3, 4\}$ and $v_i - v_{i+1} \ge \ep$. Plugging this back into \eqref{E:P01-NI} we get that
	\begin{align}
	\label{E:P01xyx}
	\P(V_{\bx, \by}(\bw, \de)) \le \de^{2 \sum_{i=1}^\ell n_i} \int_\Delta \exp ((\ell + 1)c_\ep \|\ba\|_\infty) \rho(\ba^1, \dots, \ba^\ell) d \ba^1 \cdots d \ba^\ell
	\end{align}
	Now, 
	\begin{align*}
	\int_\Delta \exp &((\ell + 1)c_\ep \|\ba\|_\infty) \rho(\ba^1, \dots, \ba^\ell) d \ba^1 \cdots d \ba^\ell \\
	&\le e^{c_\ep m} \int_{\Delta} \prod_{i=1}^\ell \exp \lf( -\frac{\|\ba^i\|^2}{c_{\ep}}\rg) d \ba^1 \cdots d \ba^\ell  \\
	&= e^{c_\ep m} \prod_{i=1}^\ell  \int_{\Delta_i} \exp \lf( -\frac{\|\ba^i\|^2}{c_{\ep}}\rg) d \ba^i \le e^{c_\ep m} \de^{\sum_{i=1}^k n_i},
	\end{align*}
	where in the final line we have used the definition of $\Delta_i$. Together with \eqref{E:P01xyx} we get that 
	$$
	\P(V_{\bx, \by}(\bw, \de)) \le e^{c_\ep m} \de^{3 \sum_{i=1}^\ell n_i}.
	$$
	To complete the proof, we just need to prove the inequality $k(k-1)/2 + 2 \ell \le 3 \sum_{i=1}^\ell n_i$ for any choice of the Gelfand-Tsetlin pattern $\bw$. This can be checked by routine casework since we have $k \in \II{2, 3, 4}$ so $\bw$ has at most $k(k-1)/2$ points. Note that this is an equality if all the $w_i$ are distinct.
\end{proof}

 For this next corollary, we also build in a statement about trivial Gelfand-Tsetlin patterns. For this, we let $T_0 = \emptyset, \operatorname{GT}_0 = \{\emptyset\}$ be the trivial space containing only the empty set. We set $\Theta(\emptyset) = \infty, \# \emptyset = 0$.  

\begin{corollary}
	\label{C:Brownian-weighted-star}
	Fix $k \in \{1, 2, 3, 4\}, m \ge 1, \ep \in (0, 1/100)$. Consider the following data:
	\begin{itemize}
		\item A $k$-tuple $B$ of Brownian bridges from $(0, \bx)$ to $(1, \by)$ for vectors $\bx, \by \in [-m, m]^k_>$.
		\item A vector $\bz \in (\ep, 1- \ep)^k_<$ such that $z_i - z_{i-1} > \ep$ for all $i \in \II{2, k}$, and a vector $\al \in \R^k$.
		\item A $(k-1)$-level Gelfand-Tsetlin pattern $\bw \in\operatorname{GT}_{k-1}(\bz)$ with $\Theta(\bw) \ge \ep$.
	\end{itemize} 
	For $\de \in (0, 1)$, let $U_{\bx, \by}(\bw, \bz, \al, \de)$ be the event where
	$$
	|B_1(z_j) - \al_j| \le \de, \forall j \in \II{1, k}; \qquad B \in \NI; \qquad |B_i(w_{i, j}) - B_{i+1}(w_{i, j})| \le \de, \forall (i, j) \in T_{k-1}.
	$$
	Then
	$$
	\P(U_{\bx, \by}(\bw, \bz, \al, \de)) \le e^{c_\ep m}  \de^{k(k+1)/2 + 2 \# \bw}
	$$ 
	where $c_{\ep} > 0$ is an $\ep$-dependent constant.
\end{corollary}

\begin{proof}
In the case $k = 1$, the set $U_{\bx, \by}(\bw, \bz, \al, \de)$ is simply the set where $|B_1(z_1) - \al_1| \le \de$, which has probability at most $c_\ep \de$ by an elementary computation, yielding the corollary. For the remainder of the corollary, we assume $k \in \{2, 3, 4\}$.

	Let $U^*_{\bx, \by}(\bz, \al, \de)$ be the event where 
	$
	|B_1(z_j) - \al_j| \le \de$ for all $j \in \II{1, k}.
	$
	With notation as in Lemma \ref{L:ep-0100}, we can write
	$$
	U_{\bx, \by}(\bw, \bz, \al, \de) = U^*_{\bx, \by}(\bz, \al, \de) \cap V_{\bx, \by}(\bw, \de).
	$$
	Therefore by that lemma, to prove the corollary it suffices to show that
	$$
	\P(U^*_{\bx, \by}(\bz, \al, \de) \mid V_{\bx, \by}(\bw, \de)) \le c_{\ep} \de^k.
	$$
	The event $V_{\bx, \by}(\bw, \de)$ is measurable with respect to the $\sig$-algebra $\cF$ generated by $W_1 := B_1 - B_2, W_2 := B_2 - B_3, \dots, W_{k-1} := B_{k-1} - B_k$. On the other hand, the sum $W_k = B_1 + \dots + B_k$ is a Brownian bridge of variance $2k$, independent of $\cF$, and we can write $U^*_{\bx, \by}(\bz, \al, \de)$ in terms of the bridges $W_i, i \in \II{1, k}$ as 
	$$
	|W_k(z_j) + \be_j| \le k\de \quad \text{ for all } j \in \II{1, k},
	$$
	where $\be_j = (k-1) W_1(z_j) + (k-2) W_2(z_j) + \dots + W_{k-1}(z_j)  - k\al_j.$ Therefore
	\begin{align*}
	\P(U^*_{\bx, \by}(\bz, \al, \de) \mid V_{\bx, \by}(\bw, \de)) &\le \sup \P(U^*_{\bx, \by}(\bz, \al, \de) \mid \cF) \\
	&\le \sup_{\be \in \R^k}(|W_k(z_j) + \be_j| \le k\de \quad \text{ for all } j \in \II{1, k}) \le c_\ep \de^k,
	\end{align*}
	where the final bound follows since $\ep < z_1$, $1- \ep < z_k$, and $z_i - z_{i-1} > \ep$ for all $i \in \II{2, k}$. This gives the result. 
\end{proof}

\begin{proof}[Proof of Proposition \ref{P:weighted-stars-upper-bound}]
\textbf{Step 1: Reduction to pinches in the Airy line ensemble.} \qquad For $p = (z, r)$ and $t < r < s$ we define random functions $\cL^{p, s}, \hat \cL^{p, t} \in \cC(\N \X \R)$ as follows: write
	$$
	\sum_{i=1}^k \cL^{p, s}_i(x) = \cL(p^k; x^k, s) \qquad \sum_{i=1}^k \hat \cL^{p, t}_i(x) = \cL(x^k, t; p^k). 
	$$
	By Lemma \ref{L:jump-lemma} and time-reversal symmetry of $\cL$, almost surely $\operatorname{Star}(\bx, s; \by, t; W)$ equals the set of all $(p; q)$ for which there exists $(\ba, \bb) \in W$ such that the following claims hold:
	\begin{enumerate}[nosep, label=\arabic*.]
		\item $\cL^{p, s}_1(x_i) = a_i$ for all $i \in \II{1, k}$ and there exists $\bz \in \operatorname{GT}_{k-1}(\bx)$ with $\cL^{p, s}_i(z_{i, j}) = \cL^{p, s}_{i+1}(z_{i, j})$ for all $(i, j) \in T_{k-1}$.
		\item 
		$\hat \cL^{q, t}_1(y_i) = b_i$ for all $i \in \II{1, \ell}$ and there exists $\bw \in \operatorname{GT}_{k-1}(\by)$ with $\hat\cL^{q, t}_i(w_{i, j}) = \hat \cL^{q, t}_{i+1}(w_{i, j})$ for all $(i, j) \in T_{\ell-1}$.
	\end{enumerate}
	We can write 
	$$
	\operatorname{Star}(\bx, s; \by, t; W) = \bigcup_{n = 1}^\infty \bigcup_{m_1 = 0}^{k(k-1)/2} \bigcup_{m_2 = 0}^{\ell(\ell - 1)/2} \operatorname{Star}_{n, m_1, m_2}(\bx, s;\by, t; W)
	$$
	where $(p; q) \in \operatorname{Star}_{n, m_1, m_2}(\bx, s; \by, t; W)$ if
	\begin{itemize}[nosep]
		\item $p \in K_n := [-n, n] \X [s-n, s - 1/n]$ and $q \in K_n' := [-n, n] \X [t + 1/n, t + n]$.
		\item Points $1$ and $2$ above hold for some $\bz, \bw$ with $\Theta(\bz) > 1/n, \# \bz = m_1, \Theta(\bw) > 1/n, \#\bw = m_2$. 
		\item $(\cL^{p, s}_1(x_1), \dots, \cL^{p, s}_1(x_k), \hat \cL_1^{q, t}(y_1), \dots, \hat \cL_1^{q, t}(y_\ell)) \in W \cap [-n, n]^k$.
	\end{itemize}  
	By countable stability of Hausdorff dimension, it is enough to prove the proposition for
	\begin{equation}
	\label{E:Star-nmw}
	\dim_{1:2:3}[\operatorname{Star}_{n, m_1, m_2}(\bx, s; \by, t; W)]
	\end{equation}
	for all fixed $n, m_1, m_2$.
	
	\textbf{Step 2: Discretization.} \qquad
	For a point $u = (p; q) = (x', s'; y', t') \in \R^2$, define the box 
	$$
	\Delta_\ep(u) = [x' - \ep^2, x' + \ep^2] \X [s' - \ep^3, s' + \ep^3] \X [y' - \ep^2, y' + \ep^2] \X [t' - \ep^3, t' + \ep^3],
	$$
	and for a set $S$, let $\Delta_\ep(S) = \bigcup_{u \in S} \Delta_\ep(u)$.
	Consider the grid 
	$
	\fG_{n, \ep} := (\ep^3 \Z \X \ep^2 \Z)^2 \cap (K_{n+1} \X K_{n+1}')
	$
	and set
	$$
	S_{n, \ep} := \{(p; q) \in \fG_{n, \ep} : \Delta_\ep(p; q) \cap \operatorname{Star}_{n, m_1, m_2}(\bx, s; \by, t; W) \ne \emptyset\}.
	$$
	To prove the proposition it is enough to show that
	\begin{equation}
	\label{E:q-G-ep}
	|S_{n, \ep}| = o(\ep^{-d(k, \ell, W) - \eta})
	\end{equation}
	for every $\eta > 0$, almost surely as $\ep \cvgdown 0$.
	
	We require one more discretization step before analyzing the probability of \eqref{E:q-G-ep}. Let $\fD_{n, \ep}$ be the set of all $(\bz, \bw) \in \operatorname{GT}_{k-1}(\bx) \X \operatorname{GT}_{k-1}(\by)$ such that $z_{i, j}, w_{i', j'} \in \ep^2 \Z$ for all choices of $(i, j) \in T_{k-1}, (i', j') \in T_{\ell -1}$.
	Also, for all $\ep > 0$ we can cover $W \cap [-n, n]^k$ by $c_n \ep^{- \dim W}$ balls of radius $\ep$, where $c_n > 0$ is an $n$-dependent constant. Let $\fE_{n,\ep} \sset [-n - 1, n + 1]^3$ be the set of centres of these balls. 
	
	Now, for $\de > 0$ and $(\bz, \bw, \ba) \in \fD_{n, \ep} \X \fE_{n, \ep}$ let $\cX_{n, \ep, \de}(\bz, \bw, \ba, \cL^{p, s}, \hat \cL^{\hat q, - t})$ be the event where
	\begin{align}
	\label{E:Lpt-approx}
	|\cL^{p, s}_i(z_{i, j}) - \cL^{p, s}_{i+1}(z_{i, j})| \le \ep^{1 - \de}, \qquad |\hat \cL^{q, t}_i(w_{i', j'}) - \hat \cL^{q, t}_{i+1}(w_{i', j'})| \le \ep^{1 - \de}
	\end{align}
	for all $(i, j) \in T_{k-1}, (i', j') \in T_{\ell - 1}$, and
	$$
	\|(\cL^{p, s}_1(x_1), \dots, \cL^{p, s}_1(x_k), \hat \cL_1^{q, t}(y_1), \dots, \hat \cL_1^{q, t}(y_\ell)) - \ba\|_2 \le \ep^{1 - \de}.
	$$
	By the modulus of continuity for the extended landscape $\cL$ (Lemma \ref{L:extended-modulus}), letting
	$$
	S^*_{n, \ep, \de} := \{(p; q) \in \fG_{n, \ep} : \bigcup_{(\bz, \bw, \ba) \in \fD_{n, \ep} \X \fE_{n, \ep}} \cX_{n, \ep, \de}(\bz, \bw, \ba, \cL^{p, s}, \hat \cL^{q, - t}) \text{ holds}\}
	$$ we have that
	\begin{equation}
	\label{E:discretized-final}
	\lim_{\ep \to 0}\P(S_{n, \ep} \sset S^*_{n, \ep, \de}) = 1
	\end{equation}
	for every fixed $\de > 0$. Hence to prove \eqref{E:q-G-ep} it suffices to control $|S^*_{n, \ep, \de}|$. 
	
	\textbf{Step 3: Bounding $\P(\cX_{n, \ep, \de}(\bz, \bw, \ba, \cL^{p, s}, \hat \cL^{\hat q, - t}))$.} \qquad
	By Theorem \ref{T:iso-multiples}.1, Theorem \ref{T:resampling-candidate}, and the symmetries of $\cL$ (Lemma \ref{L:invariance}), for any fixed $(p; q) \in \fG_{n, \ep}$ the law $\mu_{p, q}$ of the process
	$$
	(\cL^{p, s}|_{\II{1, k} \X [-x_1 - 1, x_k + 1]}, \hat \cL^{q, t}|_{\II{1, k} \X [-y_1 - 1, y_\ell + 1]})
	$$
	is absolutely continuous with respect to the law $\nu_{p, q}$ of $(B, B')$, where
	\begin{itemize}[nosep]
		\item $B$ is a $k$-tuple of independent Brownian bridges from $(x_1 - 1, \bu)$ to $(x_k + 1, \bu')$ where $\bu =\bu(p, s), \bu' = \bu'(p, s)$ are random endpoints, independent of $B$.
		\item $B'$ is an $\ell$-tuple of independent Brownian bridges from $(x_1 - 1, \bv)$ to $(x_k + 1, \bv')$ where $\bv =\bv(q, t), \bv' = \bv'(q, t)$ are random endpoints, independent of $B'$.
		\item $B$ and $B'$ are independent.
		\item There exists $c'_{\bx, \by, n} > 0$ such that for all $a > 0$, we have 
		$$
		\P(\|(\bu, \bu', \bv, \bv')\|_\infty > a) \le c'_{\bx, \by, n}e^{ - a^{3/2}}.
		$$
	\end{itemize} 
	Moreover, by Theorem \ref{T:resampling-candidate}.1, the Radon-Nikodym derivative $\tfrac{d\mu_{p, q}}{d\nu_{p, q}}$ is bounded above by a constant $c_{\bx, \by, n} > 0$. Since 
	all points in $\fG_{n, \ep}, \ep > 0$ are contained in the compact set $K_{n+1} \X K_{n+1}'$ we can take this constant to be independent of $p, q, \ep$.
	Finally, recall that almost surely we have
	$$
	\cL^{p, s}|_{\II{1, k} \X [-x_1 - 1, x_k + 1]} \in \NI, \qquad \hat \cL^{q, t}|_{\II{1, k} \X [-y_1 - 1, y_\ell + 1]} \in \NI
	$$
	for any fixed $p, q$. Therefore for $(p; q) \in \fG_{n, \ep}$ and $(\bz, \bw, \ba) \in \fD_{n, \ep} \X \fE_{n, \ep}$ we have
	\begin{align*}
	\P(\cX_{n, \ep, \de}(\bz, \bw, \ba, \cL^{p, s}, \hat \cL^{\hat q, - t})) \le c_{\bx, \by, n} \P(\cX_{n, \ep, \de}(\bz, \bw, \ba, B, B') \cap \{B \in \NI, B' \in \NI\}).
	\end{align*}
	Since $B, B'$ are independent, we can bound the right-hand side above by Corollary \ref{C:Brownian-weighted-star}. With notation as in Corollary \ref{C:Brownian-weighted-star}, we have
	\begin{align*}
	\P(\cX_{n, \ep, \de}&(\bz, \bw, \ba, B, B') \cap \{B \in \NI, B' \in \NI\}) \\
	&= \E [\P(U_{\bu, \bu'}(\bz, \bx, (a_1, \dots, a_k), \ep^{1-\de})) \P(U_{\bv, \bv'}(\bw, \by, (a_{k+1}, \dots, a_{k+\ell}), \ep^{1-\de}))]\\
	&\le \ep^{(1- \de)[k(k+1)/2 + \ell(\ell+ 1)/2 + 2 m_1 + 2m_2]}\E e^{c_{\bx, \by, n} \|(\bu, \bu', \bv, \bv')\|_\infty} \\
	&\le c_{\bx, \by, n} \ep^{(1- \de)[k(k+1)/2 + \ell(\ell+ 1)/2 + 2 m_1 + 2m_2]}
	\end{align*}
	In the above calculation, in the expectation the only randomness involves the endpoints $\bu, \bu', \bv, \bv'$. The first inequality is by Corollary \ref{C:Brownian-weighted-star} and the second inequality is by the fourth bullet above.
	Therefore by a union bound over the $c_n \ep^{-10}$ points in $\fG_{n, \ep}$, the $c_{\bx, \by, n} \ep^{-2(m_1 + m_2)}$ points in $\fD_{n, \ep}$, and the $c_n \ep^{- \dim W}$ points in $\fE_{n, \ep}$, we have that
	\begin{align*}\E |S^*_{n, \ep, \de}| \le 
	&c''_{\bx, \by, n} \ep^{- 10 - \dim W + k(k+1)/2 + \ell(\ell+ 1)/2 - \de[k(k+1)/2 + \ell(\ell+ 1)/2 + 2 m_1 + 2m_2]} \\
	\le \; &c''_{\bx, \by, n} \ep^{- d(k, \ell, W) - 1000\de}.
	\end{align*}
	Taking $\de > 0$ small enough and applying Markov's inequality and \eqref{E:discretized-final} yields \eqref{E:q-G-ep} for any fixed $\eta > 0$, as desired.
\end{proof}

\subsection{Upper bound on set size for $3$-geodesic stars}
\label{S:upper-size}

We finish this section with a lemma that gives a converse to Lemma \ref{L:lower-bd-points'} by showing that certain sets of geodesic stars with weight in particular sets are finite.
This lemma can be used to refine Proposition \ref{P:weighted-stars-upper-bound} in the $d(k, \ell, W) = 0$ case when $k = \ell = 3$, and $W$ is a slide-invariant affine set of dimension $2$. For this lemma, define $L_3:\R^3 \to \R^2$ by $L_3(x_1, x_2, x_3) = (x_1 - x_2, x_3 - x_2)$.
\begin{lemma}
	\label{L:upper-bd-3-stars}
	Almost surely, for all $\bx \in \R^3_<,t \in \R,$ and $m \in \R^2$ we have
		$$
		\# \operatorname{Star}(\bx, t; L_3^{-1}(m)) \le 135.
		$$
\end{lemma}

The bounds here is not optimal; the key point is just that it is finite. The proof of Lemma \ref{L:upper-bd-3-stars} is topological.

\begin{proof}
	Throughout the proof, we write $X = \operatorname{Star}(\bx, t; L_3^{-1}(m))$.
	We start with two claims.
	
	\begin{claim1}
		Suppose that $p_1 = (y_1, s_1), p_2 = (y_2, s_2) \in X$ and that $\pi_1, \pi_2$ are two geodesics starting at points $p_{i(1)}, p_{i(2)}$ and ending at points $(x_{i'(1)}, t), (x_{i'(2)}, t)$ where $i(1), i(2) \in \{1, 2\}$ and $i'(1), i'(2)\in \{1, 2, 3\}$. Suppose additionally that there exists a time $r \in [s_1 \vee s_2, t)$ such that $\pi_1(r) = \pi_2(r)$.
		Then the paths
		\begin{equation}
		\label{E:pipi'}
		\pi_1|_{[s_{i(1)}, r]} \oplus \pi_2|_{[r, t]} \qquad \mathand \qquad \pi_2|_{[s_{i(2)}, r]} \oplus \pi_1|_{[r, t]}
		\end{equation}
		are geodesics.
	\end{claim1}
	\begin{claimproof}
		Observe that
		\begin{equation}
		\label{E:pi-i-1}
		\begin{split}
		\cL(p_{i(1)}; x_{i(1)}, t) &+ \cL(p_{i(2)}; x_{i(2)}, t) = \|\pi_1\|_\cL + \|\pi_2\|_\cL \\
		&= \|\pi_1|_{[s_{i(1)}, r]} \oplus \pi_2|_{[r, t]}\|_\cL + \|\pi_2|_{[s_{i(2)}, r]} \oplus \pi_1|_{[r, t]}\|_\cL.
		\end{split}
		\end{equation}
		Moreover, we have
		\begin{equation}
		\label{E:cLcL'}
		\cL(p_{i(1)}; x_{i'(1)}, t) + \cL(p_{i(2)}; x_{i'(2)}, t) = \cL(p_{i(2)}; x_{i'(1)}, t) + \cL(p_{i'(1)}; x_{i(2)}, t).
		\end{equation}
		Indeed, this is a tautology if either $i(1) = i(2)$ or $i'(1) = i'(2)$. If both $i(1) \ne i(2)$ and $i'(1) \ne i'(2)$ then this holds since $p_1, p_2 \in X$. Combining \eqref{E:pi-i-1} with \eqref{E:cLcL'} gives that the paths in \eqref{E:pipi'} are geodesics.
	\end{claimproof}

	\begin{claim2}
		Suppose that $|X| \ge 136$. Then we can find three points 
		$$
		\{p_i = (y_i, s_i), i = 1, 2, 3 \}\sset X
		$$
		with either $s_1 \le s_2 \le s_3$ or $s_1 \ge s_2 \ge s_3$, $\tilde \bx \in \{(x_1, x_2), (x_2, x_3)\}$, and geodesic $2$-stars $\tau^1, \tau^2, \tau^3$ from $p_i$ to $(\tilde \bx, t)$ such that
		\begin{itemize}[nosep]
			\item $\tau^1 = \tau^2 = \tau^3$ on some interval $[t-\de, t]$.
			\item $\tau^i_2$ and $\tau^j_1$ are not disjoint for all $i < j \in \II{1, 3}$. 
		\end{itemize}
	\end{claim2} 
	
	\begin{claimproof}
		Suppose that we can find points $p_1 = (y_1, s_1), \dots, p_{136} = (y_{136}, s_{136}) \in X$ and let $\pi^i, i \in \II{1, 136}$ be geodesic $3$-stars from $p_i$ to $(\bx, t)$. We can further assume that there exists $\de > 0$ such that for every $i \in \II{1, 136}$, on $[t-\de, t]$, each path $\pi^i_j$ follows a leftmost geodesic, since by the landscape shape theorem (Theorem \ref{T:landscape-shape}), the leftmost geodesics from $(\pi^i_j(t-\de), t- \de)$ to $(x_j, t)$, $j = 1, 2, 3$ are necessarily disjoint for small enough $\de$.

		We define an equivalence relation on the set $\II{1, 136}$ by saying that $i \sim_1 j$ if $\pi^i_1|_{[t-\ep, t]} = \pi^j_1|_{[t-\ep, t]}$ for some $\ep > 0$. This splits $\II{1, 136}$ into at most $3$ equivalence classes. Indeed, if $\pi^i, \pi^j$ are in different equivalence classes then $\pi^i_1, \pi^j_1$ cannot intersect on the interval $[t-\de, t)$, since they are both leftmost geodesics on this interval. Therefore if there were at least four equivalence classes, we could construct a $4$-geodesic star to $(x_1, t)$, contradicting Theorem \ref{T:star-computation}. Similarly defining $\sim_2, \sim_3$ and letting $i \sim j$ if $i \sim_1 j, i \sim_2 j, i \sim_3 j$, the relation $\sim$ splits $\II{1, 136}$ into at most $27$ equivalence classes. One of these must have at least $6$ elements. After relabeling, we may assume $1, \dots, 6$ belong to this class, and assume that $p_1, \dots, p_6$ are listed in order of increasing time coordinate (breaking ties arbitrarily).
		
		For every pair $i < j \in \II{1, 6}$, by planarity there must exist a time $s_{i, j}$ with $\pi^i_2(s_{i, j}) \in \{\pi^j_1(s_{i, j}), \pi^j_3(s_{i, j})\}$. Therefore (i.e. since the Ramsey number $R(3, 3) = 6$) we can find $m \in \{1, 3\}$ and a subset $S \sset \II{1, 6}$ with $|S| = 3$ such that if $i < j \in S$ then $\pi^i_2(s_{i, j}) = \pi^j_m(s_{i, j})$. Without loss of generality, $S= \II{1, 3}$, Then setting $\tau^i=(\pi^i_1, \pi^i_2)$ and $\tilde \bx = (x_1, x_2)$ if $m = 1$ or $\tau^i=(\pi^i_2, \pi^i_3)$ and $\tilde \bx = (x_2, x_3)$ if $m = 3$ gives the desired construction.
	\end{claimproof}
	
	We can use the two claims along with a planarity argument to prove the existence of geodesic bubbles if $|X| \ge 136$, contradicting Lemma \ref{L:rightmost-geods}. The details are as follows.
	
	\textbf{Proof of the lemma given the claims.} \qquad Let all notation be as in Claim 2. We assume $s_1 \le s_2 \le s_3$; the case when $s_1 \ge s_2 \ge s_3$ follows a symmetric argument. We first show that the ordering of the time coordinates $s_1 \le s_2 \le s_3$ implies that
	\begin{equation}
	\label{E:tau-2-s-t}
	\tau^1|_{[s_2, t]} \le \tau^2, \qquad \tau^2|_{[s_3, t]} \le \tau^3.
	\end{equation}
	We just show that $\tau^1|_{[s_2, t]} \le \tau^2$ as the second assertion in \eqref{E:tau-2-s-t} follows by an identical argument. Suppose first that $s' \in [s_2, t)$ is such that $\tau^1_1(s') = \tau^2_1(s')$. Note that $\bar \tau^1_1(s')$ is an interior point on a geodesic since $p_1 = \bar \tau^1_1(s_1)$ cannot lie on $\tau^2_1$ by construction. Now, by the first bullet point in Claim 2, we can find $s'' \in (s', t)$ with $\tau^1_1(s'') = \tau^2_1(s'')$. The point $\bar \tau^1_1(s'')$ is also in the interior of a geodesic and so by Lemma \ref{L:rightmost-geods}.3 (no geodesic bubbles) we have $\tau^1_1|_{[s', s'']} = \tau^2_1|_{[s', s'']}$. Therefore the set where $\tau^1_1 = \tau^2_1$ is a closed interval containing $t$. By a similar argument, the set where $\tau^1_2 = \tau^2_2$ is also a closed interval containing $t$.
	
Therefore by continuity, either $\tau^1_1|_{[s_2, t]} \le \tau^2_1$ or else $\tau^1_1|_{[s_2, t]} \ge \tau^2_1$ and $\tau^1_1(s_2) > \tau^2_1(s_2)$. Similarly, either $\tau^1_2|_{[s_2, t]} \le \tau^2_2$ or else $\tau^1_2|_{[s_2, t]} \ge \tau^2_2$ and $\tau^1_2(s_2) > \tau^2_2(s_2)$.
If either $\tau^1_1|_{[s_2, t]} \ge \tau^2_1$ and $\tau^1_1(s_2) > \tau^2_1(s_2)$, or $\tau^1_2|_{[s_2, t]} \ge \tau^2_2$ and $\tau^1_2(s_2) > \tau^2_2(s_2)$, then necessarily $\tau^1_2(r) > \tau^2_1(r)$ for all $r \in [s_2, t]$ by the ordering and disjointness between $\tau^1_1, \tau^1_2$ in the former case and between $\tau^2_1, \tau^2_2$ in the latter. This contradicts the second bullet point in Claim 2, completing the proof of \eqref{E:tau-2-s-t}.

	Now, for $i < j$ define 
	\begin{align*}
	X^{i, j} = \sup O(\tau^i_2, \tau^j_1), \qquad C_1^{i, j} = \inf O(\tau^i_1, \tau^j_1), \qquad C_2^{i, j} = \inf O(\tau^i_2, \tau^j_2).
	\end{align*}
	The `X' coordinates are crossing times for paths to different endpoints are defined for $i < j$ by the second bullet in Claim 2, whereas the `$C$' coordinates are coalescence times for paths to the same endpoint. Note that points of the form
	\begin{equation}
	\label{E:Xij}
	\bar \tau^i_2(X^{i, j}), \qquad  \bar \tau^i_k(C_k^{i, j})
	\end{equation}
	are always interior points on geodesics. Indeed, neither of these points equal $p_i$ because $p_i \ne p_j$ and $s_i \le s_j$, we have $X^{i, j} < t$ by continuity, and $C_k^{i, j} \le t - \de$ by the first bullet point in Claim $2$.
\begin{figure}
	\centering
	\begin{subfigure}[t]{4cm}
		\centering
		\includegraphics[height=5cm]{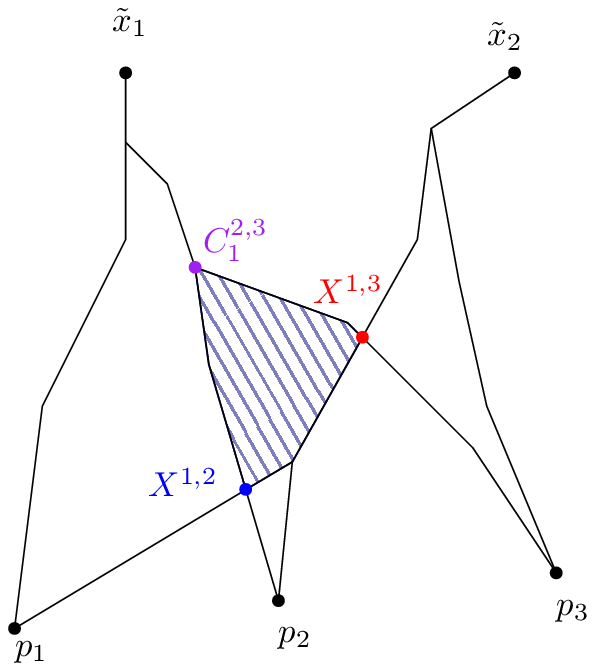}
		\caption{}
	\end{subfigure}
	\hspace{2cm}
	\begin{subfigure}[t]{4cm}
		\centering
		\includegraphics[height=5cm]{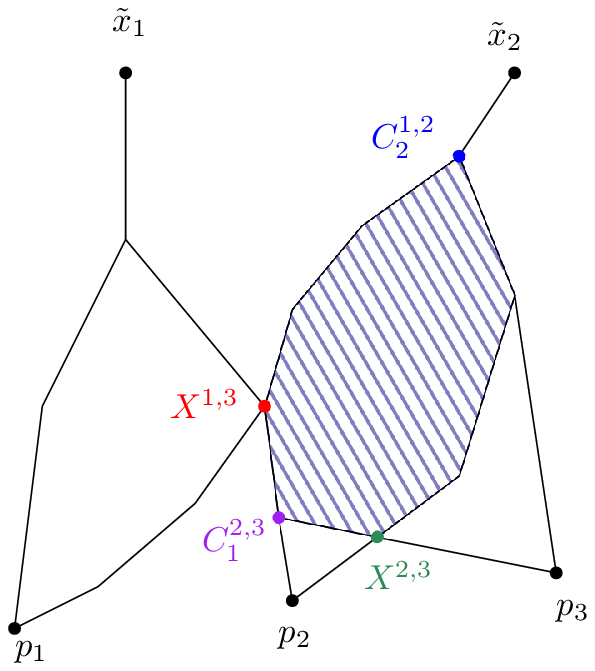}
		\caption{}
	\end{subfigure}
	\caption{Geodesic bubbles in the proof of Lemma \ref{L:upper-bd-3-stars}. Figure (A) covers the case when $X^{1, 3} < C^{2, 3}_1$ and figure (B) covers the case when $X^{1, 3} \ge C^{2, 3}_1$.}
	\label{fig:bubbles}
\end{figure}
	
	First suppose that $X^{1, 3} < C^{2, 3}_1$, see Figure \ref{fig:bubbles}(A). 
	Now, $\tau^2_1 \le \tau^3_1$ on $[s_3, t]$ by \eqref{E:tau-2-s-t} and $\tau^3_1(r) < \tau^1_2(r)$ for $r \in (X^{1, 3}, t]$. Therefore $\tau^2_1(r) < \tau^1_2(r)$ for $r \in (X^{1, 3}, t]$ and so $X^{1, 2} \le X^{1, 3}$. Moreover, we cannot have $X^{1, 2} = X^{1, 3}$ since this would imply $C^{2, 3}_1 \le X^{1, 3}$ and so $X^{1, 2} < X^{1, 3}$.
	Therefore by Claim 2,
	$$
	\tau^1_2|_{[X^{1, 2}, X^{1, 3}]} \oplus \tau^3_1|_{[X^{1, 3}, C^{2, 3}_1]} \qquad \mathand \qquad \tau^2_1|_{[X^{1, 2}, C^{2, 3}_1]}
	$$ 
	are both geodesics, which are distinct since $\tau^1_2(r) > \tau^2_1(r)$ for $r > X^{1, 2}$. These geodesic go between the same interior geodesic points of the form in \eqref{E:Xij}. This creates a geodesic bubble, contradicting Lemma \ref{L:rightmost-geods}.3.
	
	Next, suppose that $X^{1, 3} \ge C^{2, 3}_1$, see Figure \ref{fig:bubbles}(B). Note that $s_2 < C^{2, 3}_1$ since $p_2 \ne p_3$. In this case, at time $X^{1, 3} > s_2$ the paths $\tau^3_1, \tau^2_1$, and $\tau^1_2$ all overlap, and so $X^{1, 3} < C^{1, 2}_2$ since otherwise 
	$$
	\tau^2_2(X^{1, 3}) = \tau^2_1(X^{1, 3})
	$$
	which would contradict the disjointness of $\tau^2_1, \tau^2_2$ on $(s_2, t]$.
	Also, $X^{2, 3} < C^{2, 3}_1$ since otherwise $$
	\tau^2_2(X^{2, 3}) = \tau^3_1(X^{2, 3}) = \tau^2_1(X^{2, 3}),
	$$
	again contradicting disjointness of $\tau^2_1, \tau^2_2$ on $(s_2, t]$. Therefore
	$
	X^{2, 3} < C^{2, 3}_1 \le X^{1, 3} < C^{1, 2}_2
	$
	and so the paths
	$$
	\tau^2_2|_{[X^{2, 3}, C^{1, 2}_2]} \qquad \mathand \qquad \tau^3_1|_{[X^{2, 3}, X^{1, 3}]} \oplus \tau^1_2|_{[X^{1, 3}, C^{1, 2}_2]}
	$$
	form two geodesics between interior geodesic points. These geodesics are distinct by the definition of $C^{1, 2}_2$, contradicting Lemma \ref{L:rightmost-geods}.3.
\end{proof}

\section{There are at most $27$ networks}
\label{S:at-most-27}

In this section we prove the `only if' direction in Theorem \ref{T:geodesic-networks}, as the only remaining part of the proof is closely related to the proof of Proposition \ref{P:weighted-stars-upper-bound}. Indeed, in Section \ref{S:coalescent} we showed that the network graph $G(p; q)$ of any pair $(p; q) \in \Rd$ satisfies points 1-3 of Theorem \ref{T:geodesic-networks}. Theorem \ref{T:star-computation} shows that $G(p; q)$ satisfies point $5$, so it remains to show point $4$: namely, that all interior vertices of $G(p; q) = (V, E)$ have total degree equal to $3$. By construction, each of these vertices has total degree at least $3$, so to prove point $3$ we need to show that no such vertex can have degree $4$ or more. There are two possibilities we need to rule out:
\begin{enumerate}[label=\arabic*.]
	\item A vertex $v \in V \smin \{p, q\}$ has in-degree and out-degree at least $2$. This is ruled out by Lemma \ref{L:no-X-s}.
	\item A vertex $v \in V \smin \{p, q\}$ has out-degree at least $3$ (or in-degree at least $3$). This is ruled out by Lemma \ref{L:no-sceptres} and time-reversal symmetry of $\cL$.
\end{enumerate}

\begin{lemma}
	\label{L:no-X-s}
	Almost surely, the following holds. Consider any geodesics $\ga, \pi$. Then the overlap $O(\ga, \pi)$ is not a single point.
\end{lemma}

We prove Lemma \ref{L:no-X-s} by contradiction. The first step is to show that the negation of Lemma \ref{L:no-X-s} implies the existence of rational `figure eights'.
\begin{lemma}
	\label{L:maybe-figure-eights}
	Suppose that Lemma \ref{L:no-X-s} fails. Then with positive probability, we can find $u = (x, s; y, t) \in \Rd$ with $s, t$ rational such that the graph $\Ga(\ga^R(u); \ga^L(u))$ contains a vertex of degree $4$.
\end{lemma}

In the statement above recall that $\ga^R(p; q), \ga^L(p; q)$ denote the rightmost and leftmost geodesics from $p$ to $q$.
\begin{proof}
	Let $\ga:[s_1, t_1] \to \R, \pi:[s_2, t_2] \to \R$. If $O(\ga, \pi)$ consists of a single point $r$, then $r \in (s, t) \sset (s_1, t_1) \cap (s_2, t_2)$ for a rational interval $(s, t)$. Without loss of generality, $x^- := \ga(s) < x^+ := \pi(s)$, and we can assume that $\ga|_{(s, r)}$ is the leftmost geodesic from $(x^-, s)$ to $\bar \ga(r)$ and $\pi|_{(s, r)}$ is the rightmost geodesic from $(x^+, s)$ to $\bar \ga(r)$.
	Next, define
	$$
z^+ = \inf \{z \in [x^-, x^+]: O(\ga^R(z, s; \bar \ga(r)), \ga|_{(s, r)}) = \emptyset\}.
	$$ 
	For every fixed $r' \in [r, s]$, the function $z \mapsto \ga^R(z, s; \bar \ga(r))(r')$ is right-continuous by the monotonicity in Lemma \ref{L:regularity}.3 and since limits of geodesics are geodesics. Therefore the function $z \mapsto \ga^R(z, s; \bar \ga(r))$ is continuous in the overlap topology by Proposition \ref{P:overlap-metric}.2, so $O(\ga^R(z^+, s; \bar \ga(r)), \ga|_{(s, r)}) = \emptyset$. Next, let 
	$$
	z^- = \sup \{x \in [x^-, z^+]: O(\ga^L(x, s; \bar \ga(r)), \ga^R(z^+, s; \bar \ga(r))) = \emptyset\}.
	$$ 
	Again by Proposition \ref{P:overlap-metric}.2, we have that $O(\ga^L(z^-, s; \bar \ga(r)), \ga^R(z^+, s; \bar \ga(r))) = \emptyset$. We claim that $z^-=z^+$. Indeed, if not then we can find $x \in (z^-, z^+)$. By the definition of $z^+$, the geodesic $\ga^L(x, s; \bar \ga(r))$ must have non-empty overlap with $\ga|_{(s, r)}$ and hence also with $\ga^L(z^-, s; \bar \ga(r))$ by monotonicity (Lemma \ref{L:regularity}.3). By the definition of $z^-$ the geodesic $\ga^L(x, s; \bar \ga(r))$ must also have non-empty overlap with $\ga^R(z^+, s; \bar \ga(r))$. Let 
	\begin{align*}
	r^- &= \inf O(\ga^L(x, s; \bar \ga(r)), \ga^L(z^-, s; \bar \ga(r))), \\
	\qquad r^+ &= \inf O(\ga^L(x, s; \bar \ga(r)), \ga^R(z^+, s; \bar \ga(r)))
	\end{align*}

	If $r^+ < r^-$, then the geodesic
	$$
	\ga^L(x, s; \bar \ga(r))|_{[s, r^+]} \oplus \ga^R(z^+, s; \bar \ga(r))|_{[r^+, r]}
	$$
	is disjoint from $\ga^L(z^-, s; \bar \ga(r))$, and hence also from $\ga|_{(s, r)}$, which contradicts the minimality of $z^+$. If $r^+ > r^-$, then the geodesic
	$$
	\ga^L(x, s; \bar \ga(r))|_{[s, r^-]} \oplus \ga^L(z^-, s; \bar \ga(r))|_{[r^-, r]}
	$$
	is to the left of $\ga^L(x, s; \bar \ga(r))$ contradicting the fact that $\ga^L(x, s; \bar \ga(r))$ is a leftmost geodesic. Hence $z^- = z^+$. 
	
	Similarly, we can find a point $y \in (\ga(t), \pi(t))$ such that the leftmost and rightmost geodesics from $\bar \ga(r)$ to $(y, t)$ are disjoint. Let $u = (z^-, s; y, t)$. To finish the proof, we claim that $O(\ga^R(u), \ga^L(u)) = \{r\}$. Indeed, by geodesic ordering (Lemma \ref{L:regularity}.3),
	$$
	\ga|_{(s, t)} \le \ga^L(u) \le \ga^R(u) \le \pi|_{(s, t)},
	$$
	so $\{r\} \sset O(\ga^R(u), \ga^L(u))$ and $\ga(r) = \ga^R(u) = \ga^L(u)$. On the other hand, we have exhibited pairs of disjoint geodesics from $(z^-, s)$ to $\bar \ga(r)$ and from $\bar \ga(r)$ to $(y, t)$, so $O(\ga^R(u), \ga^L(u)) = \{r\}$.
\end{proof}

\begin{lemma}
	\label{L:no-figure-eights}
	Almost surely, there are no points $u = (x, s; y, t) \in \Rd$ with $s, t$ rational such that $O(\ga^R(u); \ga^L(u))$ is a single point.
\end{lemma}

Just as with the proof of Proposition \ref{P:weighted-stars-upper-bound}, the proof of Lemma \ref{L:no-figure-eights} is based on the analysis of a problem involving rare events for the Airy line ensemble. As in the proof of Proposition \ref{P:weighted-stars-upper-bound}, we require a lemma about the rarity of a certain Brownian event.

\begin{lemma}
	\label{L:no-X-Brownian}
	Fix $\ep > 0, m \ge 1$, and $a, b \in \R$ with $a + 2 \ep < b$.
	Let $\bar B = (B, B') = (B_1, B_2, B_1', B_2')$ be a $4$-tuple of Brownian bridges from $(a, \bx)$ to $(b, \by)$ for $\bx, \by \in [-m, m]^4$. Fix $z \in (a + \ep, b - \ep)$, and let $T_{a, b}(B, B'; z, \de)$ be the event where
	\begin{align*}
	B_1(z) - B_2(z) &< \de, \qquad \qquad B_1'(z) - B_2'(z) < \de, \text{ and } \\
	B_1(z) + B_1'(z) + \de &> B_1(w) + B_1'(w) \text{ for all } w \in [a, b]. 
	\end{align*}
	Then
	$$
	\P(T_{a, b}(B, B'; z, \de) , B \in \NI, B' \in \NI) \le e^{c_{\ep, a, b} m} \de^{14}.
	$$
\end{lemma}
\begin{proof}
	Consider the $4$-tuple of functions 
	$$
	W =(W_1, W_2, W_3, W_4) = (-B_2 + \de, - B_1 + \de, B_1' - [B_1(z) + B_1'(z)], B_2' - [B_1(z) + B_1'(z)]).
	$$
	Then $T_{a, b}(B, B'; z, \de) \cap \{B \in \NI\} \cap \{B' \in \NI\}$ is the same the event $T^*$ given by
	$$
	W \in \NI, \quad W_i(z) - W_{i+1}(z) < \de \quad \text{ for } i = 1, 3.
	$$
	To estimate $\P(T^*)$, observe that conditional on $\bar B(z)$, $W|_{[a, z]}$ and $W_{[z, b]}$ are two independent $4$-tuples of Brownian bridges from $(0, g(\bx))$ to $(z, g (\bar B(z)))$ and from $(z, g(\bar B(z)))$ to $(1, g(\by))$, where
	$$
	g(w_1, w_2, w_3, w_4) = (-w_2 + \de, - w_1 + \de, w_3 - [B_1(z) + B_1'(z)], w_4 - [B_1(z) + B_1'(z)]).
	$$
	Therefore
	\begin{align}
	\label{E:T*-bound}
	\P(T^*) = \int_\Delta \P_{0, z}( g(\bx), g(\ba); \NI) \P_{z, 1}(g(\ba), g(\by); \NI) \rho(\ba) d \ba,
	\end{align}
	where $\rho$ is the (Gaussian) density of $\bar B(z)$, and $\Delta$ is the set of all $\ba \in \R^4$ with $0 < a_1 - a_2 < \de$ and $0 < a_3 - a_4 < \de$. By Lemma \ref{L:BB-nonint},
	\begin{equation*}
	\P_{0, z}(g(\bx), g(\ba); \NI) \le e^{c_{a, b, \ep} (m + \|\ba\|_\infty)} \de^6, \qquad \P_{z, 1}(g(\ba), g(\by); \NI) \le e^{c_{a, b, \ep} (m + \|\ba\|_\infty)} \de^6,
	\end{equation*}
	for $\ba \in \Delta$, where we have used that $\bx, \by \in [-m, m]^4$ to remove the dependence on these parameters. Therefore by \eqref{E:T*-bound} we have the desired bound:
	\[
	\P(T^*) \le \int_\Delta e^{2c_{a, b, \ep} (m + \|\ba\|_\infty)} \de^{12} \rho(\ba) d \ba\le e^{c_{a, b, \ep} m} \de^{14}. \qedhere \]
\end{proof}

For the proof of Lemma \ref{L:no-figure-eights}, we use both the $\cL^{p, t}, \hat \cL^{q, t}$ notation used in the proof of Proposition \ref{P:weighted-stars-upper-bound} and the notation of Lemma \ref{L:no-X-Brownian}.
\begin{proof}[Proof of Lemma \ref{L:no-figure-eights}]
	By landscape rescaling and a countable additivity argument, it suffices to show that almost surely, for every fixed $m > 0$ there are no such points $u = (x, 0; y, 1)$ with $x, y \in [-m, m]$ such that the overlap $O(\ga^R(u); \ga^L(u))$ is a single point $p = (z, t)$ with $z \in [-m, m]$ and $t \in [1/m, 1-1/m]$. Suppose that there is a triple $(x, y, p)$ for which this is indeed the case. Then the following three equations hold:
	\begin{align*}
	\cL(x, 0; z, t) + \cL(z, t; y, 1) &= \max_{w \in [-m - 1, m + 1]} \cL(x, 0; w, t) + \cL(w, t; y, 1), \\
	\cL((x, x), 0; (z, z), t) &= 2 \cL(x, 0; z, t), \\
	\cL((z, z), t; (y, y), 1) &= 2 \cL(z, t; y, 1).
	\end{align*}
We can equivalently write these equations in terms of the line ensembles $\cL^{x, t} := \cL^{(x, 0), t}, \hat \cL^{y, t} := \hat \cL^{(y, 1), t}$:
	\begin{equation}
	\label{E:A-z}
	\begin{split}
	\cL^{x, t}_1(z) + \hat \cL^{y, t}_1(z) &= \max_{w \in [-m - 1, m + 1]} \cL^{x, t}_1(w) + \hat \cL^{y, t}_1(w), \\
	\quad \cL^{x, t}_1(z) &= \cL^{x, t}_2(z), \quad \hat \cL^{y, t}_1(z) = \hat \cL^{y, t}_2(z).
	\end{split}
	\end{equation}
	Now fix $\de > 0$. By the modulus of continuity for the extended landscape (Lemma \ref{L:extended-modulus}), if \eqref{E:A-z} holds for some $x, y, z \in [-m, m]$ and $t \in [1/m, 1-1/m]$ then
	 for all small enough $\ep > 0$ the event
	 \begin{equation}
	 \label{E:union-T}
	 \bigcup_{\substack{x, y, z \in \ep^{2 + \de} \Z \cap [-m - 1/2, m + 1/2],\\ t \in \ep^{3+\de} \Z \cap [1/(2m), 1 - 1/(2m)]}} T_{-m-1, m+1}(\cL^{x, t}, \hat \cL^{y, t}; z, \ep)
	 \end{equation}
	 holds. By Theorem \ref{T:resampling-candidate}, for all choices of $x, y, z, t$ in the union in \eqref{E:union-T} there are random vectors $\bw =\bw(x, y, t), \bw' = \bw'(x, y, t)$ and a constant $c_m > 0$ such that letting $B, B'$ be independent Brownian bridges from $(-m-1, \bw)$ to $(m + 1, \bw')$ we have
	 \begin{align*}
\P(T_{-m-1, m+1}(\cL^{x, t}, \hat \cL^{y, t}; z, \ep)) &= \P(T_{-m-1, m+1}(\cL^{x, t}, \hat \cL^{y, t}; z, \ep), \cL^{x, t} \in \NI, \hat \cL^{y, t} \in \NI) \\
&\le c_m \E [\P(T_{-m-1, m+1}(B, B'; z, \ep), B \in \NI, B' \in \NI \mid \bw, \bw')] \\
&\le c_m \ep^{14} \E e^{c_m \|(\bw, \bw')\|_\infty}.
	 \end{align*}
	Here the fact that the constant in the first inequality only depends on $m$ uses the fact that $x, y \in [-m - 1, m + 1], t \in [1/(2m), 1 - 1/(2m)]$ and the landscape symmetries in Lemma \ref{L:invariance}, and the second inequality is Lemma \ref{L:no-X-Brownian}. Finally, $\E e^{c_m \|(\bw, \bw')\|_\infty} \le c_m'$ by the tail bounds in Theorem \ref{T:resampling-candidate}.3, so the probability of \eqref{E:union-T} is $O(\ep^{5 - 4 \de})$ as $\ep \to 0$. Therefore the probability that there is some $x, y, z, t$ for which \eqref{E:A-z} holds is $0$.
\end{proof}

Lemma \ref{L:no-X-s} is an immediate consequence of Lemmas \ref{L:maybe-figure-eights} and \ref{L:no-figure-eights}.

\begin{lemma}
	\label{L:no-sceptres}
	Almost surely, we cannot find three geodesics $\pi_1, \pi_2, \pi_3$ that start and end at common points $p = (x, s)$ and $q = (y, t)$, coincide on an interval $[s, r]$ with $s < r < t$ and are mutually disjoint on the interval $(r, t)$.
\end{lemma}

The strategy for proving Lemma \ref{L:no-sceptres} is similar to that of Lemma \ref{L:no-X-s}, except we do not require an initial rational reduction step (a version of Lemma \ref{L:maybe-figure-eights}). We start with a Brownian lemma, the analogue of Lemma \ref{L:no-X-Brownian} in the present setting.

\begin{lemma}
	\label{L:no-sceptres-Brownian}
	Fix $\ep > 0, m \ge 1$, and $a, b \in \R$ with $a + 2 \ep < b$. Let $\bar B = (B, B') = (B_1, B_2, B_3, B_1')$ be a $4$-tuple of Brownian bridges from $(a, \bx)$ to $(b, \by)$ for $\bx, \by \in [-m, m]^4$. Fix $z \in (a + \ep, b - \ep)$, and let $S_{a, b}(B, B', z, \de)$ be the event where
	\begin{align*}
	B_1(z) - B_3(z) < \de, \qquad
	&B_1(z) + B_1'(z) + \de > B_1(w) + B_1'(w)\quad \forall w \in [a, b].
	\end{align*}
	Then
	$$
	\P(S_{a, b}(B, B', z, \de), B \in \NI) \le e^{c_{\ep, a, b} m} \de^{14}.
	$$
\end{lemma}

\begin{proof}
	The proof is similar to the proof of Lemma \ref{L:no-X-Brownian}.
	Consider the $4$-tuple of functions 
	$$
	W =(W_1, W_2, W_3, W_4) = (-B_3 + \de, - B_2 + \de, -B_1 + \de, B_1' - [B_1(z) + B_1'(z)]).
	$$
	Then $S_{a, b}(B, B', z, \de) \cap \{B \in \NI\}$ is contained in the event $S^*$ where
	$$
	W \in \NI, \quad W_1(z) - W_3(z) < \de.
	$$
	To estimate $\P(S^*)$, observe that conditional on $\bar B(z)$, $W|_{[0, z]}$ and $W|_{[z, 1]}$ are two independent $4$-tuples of Brownian bridges from $(0, g(\bx))$ to $(z, g(\bar B(z)))$ and from $(z, g(\bar B(z)))$ to $(1, g(\by))$, where
	$$
	g(w_1, w_2, w_3, w_4) =(-w_3 + \de, - w_2 + \de, -w_1 + \de, w_4 - [B_1(z) + B_1'(z)]).
	$$
	Therefore
	\begin{align}
	\label{E:S*-bound}
	\P(S^*) = \int_\Delta \P_{0, z}(g(\bx), g(\ba); \NI) \P_{z, 1}(g(\ba), g(\by); \NI) \rho(\ba) d \ba,
	\end{align}
	where $\rho$ is the (Gaussian) density of $\bar B(z)$, and $\Delta$ is the set of all $\ba \in \R^4$ with $0 < a_1 - a_3 < \de$ and $a_1 > a_2 > a_3$. From here the proof goes through as in the proof of Lemma \ref{L:no-X-Brownian}.
\end{proof}

\begin{proof}[Proof of Lemma \ref{L:no-sceptres}]
	First note that if such geodesics $\pi_1, \pi_2, \pi_3$ exist, then by possibly restricting $\pi_1, \pi_2, \pi_3$ to an interval $[s', t]$ with $s' \in \Q \cap (s, r)$ we may assume that $s$ is rational. Let $\cX_{m, s}$ denote the event where three geodesics $\pi_1, \pi_2, \pi_3$ as in the statement of the lemma exist, with
	\begin{equation}
	\label{E:pq-conditions}
		x, y, t \in [-m, m], \qquad s \le r - 1/m < r + 1/m \le t, \qquad z := \pi_1(r) \in [-m, m].
	\end{equation}
	By countable additivity, it suffices to show that $\P(\cX_{m, s}) = 0$ for all fixed $m \in \N, s \in \Q$. Suppose that such a triple $\pi_1, \pi_2, \pi_3$  does exist. Then the following two equations hold:
	\begin{align*}
	\cL(x, s; z, r) + \cL(z, r; y, t) &= \max_{w \in [-m - 1, m + 1]} \cL(x, s; w, r) + \cL(w, r; y, t) \\
	\cL(z^3, r; y^3, t)&= 3 \cL(z, r; y, t).
	\end{align*}
	In other words, with notation as in the proof of Lemma \ref{L:no-figure-eights} we have
	\begin{equation}
	\label{E:A-z-three}
	\begin{split}
	\cL^{p, r}_1(z) + \hat \cL^{\hat q, r}_1(z) = \max_{w \in [-m - 1, m + 1]} \cL^{p, r}_1(w) + \hat \cL^{\hat q, r}_1(w), \qquad \hat \cL^{\hat q, r}_1(z) = \hat \cL^{\hat q, r}_3(z).
	\end{split}
	\end{equation}
	Now, fix $\de > 0$. By the modulus of continuity for the extended landscape (Lemma \ref{L:extended-modulus}), if \eqref{E:A-z-three} holds for some $p = (x, t), q = (y, s), r, z$ satisfying \eqref{E:pq-conditions}, then for all small enough $\ep > 0$ the following event holds:
	\begin{equation}
	\label{E:union-S}
		\bigcup_{\substack{x, y, z \in \ep^{2 + \de} \Z \cap [-m - 1/2, m + 1/2], \\
			t \in \ep^{3+\de} \Z \cap [-m - 1/2, m + 1/2], \\
			r \in \ep^{3+\de} \Z \cap (s + 1/(2m), t - 1/(2m))}} S_{-m - 1, m + 1}(\hat \cL^{(y, t), r}, \cL^{(x, s), r}, z, \ep).
	\end{equation}
	By Theorem \ref{T:resampling-candidate}, for all choices of $x, y, z, r, t$ in the union in \eqref{E:union-S} there are random vectors $\bw =\bw(x, y, r, t), \bw' = \bw'(x, y, r, t)$ and a constant $c_m > 0$ such that letting $B, B'$ be independent Brownian bridges from $(-m-1, \bw)$ to $(m + 1, \bw')$ we have
	\begin{align*}
	\P(S_{-m - 1, m + 1}(\hat \cL^{(y, t), r}, \cL^{(x, s), r}, z, \ep)) &= \P(S_{-m - 1, m + 1}(\hat \cL^{(y, t), r}, \cL^{(x, s), r}, z, \ep), \hat \cL^{(y, t), r} \in \NI) \\
	&\le c_m \E [\P(S_{-m-1, m+1}(B, B'; z, \ep), B \in \NI \mid \bw, \bw')] \\
	&\le c_m \ep^{14} \E e^{c_m \|(\bw, \bw')\|_\infty}.
	\end{align*}
		Here the fact that the constant in the first inequality only depends on $m$ uses that $x, y, z, r, t \in [-m - 1/2, m + 1/2], |r-t|, |r - s| \ge 1/(2m)$. The second inequality is Lemma \ref{L:no-sceptres-Brownian}. Finally, $\E e^{c_m \|(\bw, \bw')\|_\infty} \le c_m'$ by the tail bounds in Theorem \ref{T:resampling-candidate}.3, so the probability of \eqref{E:union-S} is $O(\ep^{2 - 5 \de})$ as $\ep \to 0$ and hence $\P(\cX_{m, s}) = 0$.
\end{proof}

\section{Networks and their dimensions}
\label{S:networks-dimensions}
In this section, we prove Theorem \ref{T:geodesic-networks} and parts $1$ and $2$ of Theorem \ref{T:Hausdorff-dimension}. The major part of this section is devoted to carefully understanding how to cut weighted geodesic stars away from geodesic networks, and how to paste weighted geodesic stars together via an intermediate forest. See Section \ref{S:geometric-ideas} for a rough outline. This section contains three preliminary parts before we proceed to the main proofs in Section \ref{S:cut-and-paste}: a part on graph theory basics, a short section on constructing forests in $\cL$, and a section containing two technical lemmas about the Airy sheet.

\subsection{Graph theory}
\label{S:graph-theory}
\textbf{Directed embeddings of directed graphs.} \qquad Suppose that $G = (V, E)$ is a finite, loop-free planar directed graph with no directed cycles and no vertices of both in-degree $1$ and out-degree $1$. We call $G$ a \textbf{candidate graph}. 
A map $\Phi: V \cup E \to \{\text{subsets of } \R^2\}$ is a \textbf{directed embedding} of $G$ if it satisfies the following conditions:
\begin{itemize}[nosep]
	\item $\Phi(v)$ is a point for all $v \in V$.
	\item For every edge $e = (u, v) \in E$, there is a continuous function $\ga_e:[s, t] \to \R$ such that $\Phi(e) = \fg \ga|_{(s, t)}$, and $\Phi(u) = \bar \ga(s), \Phi(v) = \bar \ga(t)$.
	\item $\Phi(a) \cap \Phi(b) = \emptyset$ for all $a \ne b \in V \cup E$.
\end{itemize}
We call $\Phi(G) := \bigcup_{a \in E \cup V} \Phi(a)$ the image of $G$ under $\Phi$. Now suppose that $\Ga \sset \R^2$ is the image of some candidate graph $G$; we claim that $G$ is unique up isomorphism. Indeed, the vertex set $V$ is given by all points $v \in \Ga$ such that there is no open set $O \sset \R^2$ containing $v$ for which the intersection $\Ga \cap O$ is given by $\{(\pi(s), s) : s \in (a, b)\}$ for some $a < b$ and a continuous function $\pi$. Edges correspond to the connected components of $\Ga \smin V$ with the direction inherited from the time coordinate, yielding a graph $G = (V, E)$. For all $(p; q) \in \Rd$, the network graph $G(p; q)$ can be constructed from the geodesic network $\Ga(p; q)$ in this way. We will also use this map to associate other subsets of $\R^2$ to directed graphs.

\textbf{Ordered forests and interior forests.} \qquad We say that a candidate graph $G = (V, E)$ is a $(k, \ell)$-\textbf{ordered forest} if:
\begin{itemize}[nosep]
	\item $G$ is a forest (as an undirected graph).
	\item $V$ has $k$ source vertices $p_1, \dots, p_k$, all of which have degree $1$.
	\item $V$ has $\ell$ sink vertices $q_1, \dots, q_\ell$, all of which have degree $1$.
	\item $G$ has a directed embedding into $\R \X [0, 1]$ with $p_i \mapsto (i, 0), q_j \mapsto (j, 1)$ for all $i \in \II{1, k}, j \in \II{1, \ell}$.
\end{itemize}
Setting some notation, let $\del V^- = \{p_1, \dots, p_k\}, \del V^+ = \{q_1, \dots, q_\ell\}, \del V = \del V^- \cup \del V^+$, and $V^\circ = V \smin \del V$. We call a $(i, j) \in \II{1, k} \X \II{1, \ell}$ an \textbf{admissible pair} if there is a path from $p_i$ to $q_j$ in $G$, and write $\operatorname{Adm}(G) \sset \II{1, k} \X \II{1, \ell}$ for the set of admissible pairs. 

If $\Ga(p; q)$ is a geodesic network from $p$ to $q$ with network graph $G$, then by virtue of rule $3$ in Theorem \ref{T:geodesic-networks}, if we remove a small ball around $p$ and $q$ we are left with an ordered forest, which we call an interior forest for $G$.  More precisely, for a graph $G = (V, E) \in \sG$ with source and sink $p, q$, we say that a graph $G'$ is an \textbf{interior forest} for $G$ if there is a directed embedding $\Phi$ of $G$ into $\R^2$ such that:
\begin{itemize}[nosep]
	\item $\Phi(p) \in \R \X (-\infty, 0), \Phi(q) \in \R\X(1, \infty)$, and $\Phi(v) \in \R \X (0, 1)$ for all $v \in V \smin \{p, q\}$.
	\item $\Phi(G) \cap (\R\X[0, 1])$ is a directed embedding of $G'$. 
\end{itemize}
While the graph structure of $G'$ is unique up isomorphism, the ordering on its source and sink vertices may not be.

\textbf{Vertex restrictions for $G \in \sG$.} \qquad Next, we prove a simple lemma restricting the size of the vertex set for a graph $G \in \sG$.

\begin{lemma}
	\label{L:graph-theory}
	Let $G = (V, E) \in \sG$ with source degree $\deg(p) = k \in \II{1, 3}$ and sink degree $\deg(q) = \ell \in \II{1, 3}$. Then 
	$
	2 + |k - \ell| \le |V| \le k + \ell
	$
	and
	\begin{equation}
	\label{E:dim-form}
	d(G) := 12 - \frac{|V| + k^2 + \ell^2}{2} = 11 - |F| - {k \choose 2} - {\ell \choose 2}  \ge 0.
	\end{equation}
	Here $|F|$ denotes the number of faces in any planar embedding of $G$.
\end{lemma}
\begin{proof}
	First, since all vertices in $V \smin \{p, q\}$ have degree $3$ we have $2|E| = k + \ell + 3(|V| - 2)$. Therefore by Euler's formula $|F| + |V| = |E| + 2$ for planar graphs, we have
	\begin{equation}
	\label{E:Eulers}
	2|F| = k + \ell + |V| - 2.
	\end{equation}
	This implies the equality in \eqref{E:dim-form}. Next, if we remove $p$ and $q$ from $G$ and the $k + \ell$ edges incident to $p$ and $q$, then we are left with a forest. Therefore
	$
	|E| - k - \ell + 3 \le |V|,
	$
	which by the previous formula for $|E|$ implies that $|V| \le k + \ell$ and also gives the inequality in \eqref{E:dim-form}.
	
	Finally, in a directed embedding of $G$ in $\R^2$, if $e_1, \dots, e_k$ are the edges exiting from $p$, listed from left to right, then each of the pairs $(e_k, e_1), (e_1, e_2), \dots, (e_{k-1}, e_k)$ must bound a different face since any two directed paths exiting $p$ eventually meet. Therefore $|F| \ge k$, and so by \eqref{E:Eulers} we have $2 + k - \ell \le V$. By symmetry, $2 + \ell - k \le V$ as well.
\end{proof}

\textbf{Weighted forests and slide-invariant affine sets.} \qquad We finish this section with a lemma that will characterize the requirements on the weight of a pair of geodesic stars connected through a given interior forest.

\begin{lemma}
	\label{L:face-to-linear-relations}
	Let $k, \ell \in \N$, and let $G = (V, E)$ be a $(k, \ell)$-ordered forest. Let $Z:E \to \R$ and for $(i, j) \in \operatorname{Adm}(G)$, let $d_Z(p_i, q_j)$ be the $Z$-weight of the unique path in $G$ from $p_i$ to $q_j$. Now let $W(Z)$ be the set of all $(\bx, \by) \in \R^k\X \R^\ell$ such that
	\begin{equation}
	\label{E:xi-dZ}
	x_i + d_Z(p_i, q_j) + y_j = x_{i'} + d_Z(p_{i'}, q_{j'}) + y_{j'} 
	\end{equation}
	for all $(i, j), (i', j')\in \operatorname{Adm}(G)$. Then:
	\begin{itemize}[nosep]
		\item $W(Z)$ is a slide-invariant affine set.
		\item Let $e^i_p, e^j_q$ be the unique edges exiting $p_i$ and entering $q_j$ in $G$, and let $E^* = \{e^1_p, \dots, e^{k}_p, e^1_q, \dots, e^\ell_q\}$. Let $F:E(G) \to \R$ be a function which is $0$ outside of $E^*$. Then
		$$
		W(Z + F) = W(Z) - (F(e^1_p), \dots, F(e^{k}_p), F(e^1_q), \dots, F(e^{\ell}_q)).
		$$
		\item $W(Z)$ has linear dimension $c(G) + 1$, where $c(G)$ is the number of components in $G$.
	\end{itemize} 
\end{lemma}

\begin{proof}
	The first bullet is immediate from definition. For the second bullet, simply observe that 
	$$
	d_{Z+F}(p_i, q_j) = F(e^i_p) + d_Z(p_i, q_j) + F(e_q^j).
	$$
	We come to the final bullet point. For a component $J$ of $G$, let 
	$$
	I_p(J) = \{i \in \II{1, k}: p_i \in J\}, \qquad \qquad I_q(J) = \{j \in \II{1, \ell} : q_j \in J\}.
	$$ 
	\begin{claim}
For any component $J$ of $H(G)$, the set of all $(\bx, \by) \in \R^{I_p(J)} \X \R^{I_q(J)}$ such that \eqref{E:xi-dZ} holds for all admissible pairs $(i, j), (i', j') \in I_p(J) \X I_q(J)$ has linear dimension $2$.
	\end{claim}

\begin{claimproof} Without loss of generality, we may assume $I_p(J) = \II{1, m}, I_q(J) = \II{1, n}$ for some $m, n \in \N$; the other cases have the same proof after relabeling.
	First, let $A$ be the undirected graph whose vertex set is $\II{1, m} \cup \{1', \dots, n'\}$ and whose edges are pairs $(i, j')$ with $(i, j) \in \operatorname{Adm}(G)$. 
	Note that $A$ must be connected, and so $A$ has at least $m + n - 1$ edges $(i_1, j_1'), \dots, (i_{m + n - 1}, j_{m + n - 1}')$. Next, the $m + n - 2$ linear equations
	$$
	x_{i_u} + d_Z(p_{i_u}, q_{j_u'}) + y_{j_u'} = x_{i_{u+1}} + d_Z(p_{i_{u+1}}, q_{j_{u+1}'}) + y_{j_{u+1}'}, \qquad u \in \II{1, m + n -2},
	$$
	are linearly independent. Therefore the set of all $(\bx, \by) \in \R^{I_p(J)} \X \R^{I_q(J)}$ satisfying \eqref{E:xi-dZ} has dimension at most $(m + n) - (m + n - 2) = 2$.

	Next, create a graph $J'$ from $J$ by adding in a vertex $p_*$ and connecting it to all the vertices $p_i, i \in \II{1, m}$ and a vertex $q_*$ and connecting all the vertices $q_i$ to it. Put weights $x_i$ on the edges $(p_*, p_i)$ and weights $y_i$ on the edges $(q_i, q_*)$.
	The graph $J'$ is still planar. Indeed, from the definition of $(k, \ell)$-ordered forests, we can find a directed embedding of $J$ into $\R \X [0, 1]$ where $p_i \mapsto (i, 0), q_j \mapsto (j, 1)$. We can easily extend this to an embedding of $J'$ into $\R \X [-1, 2]$ which maps $p_* \mapsto (0, -1), q_* \mapsto (0, 2)$. In this embedding:
	\begin{itemize}[nosep]
		\item There is a unique face containing $p_*, q_*$, and this face is unbounded.
		\item For all $i \in \II{2, m}$, there is a unique bounded face $f_i$ such that the edges $(p_*, p_{i-1}), (p_*, p_i)$ are both on the boundary of $f_i$.
		\item For all $j \in \II{2, n}$ there is a unique bounded face $f_j'$ such that the edges $(q_{j-1}, q_*), (q_j, q_*)$ are both on the boundary of $f_j'$.
	\end{itemize}  
	Note that all the $f_i, f_j'$ are distinct since otherwise the original graph $J$ would not be connected. Moreover, $f_i, i \in \II{2, m}, f_j', j \in \II{2, n}$ constitute all bounded faces of $J'$ since $J$ is a tree.
	
	Now, equation \eqref{E:xi-dZ} holds for all admissible $(i, j), (i', j')$ if the sum of weights along any undirected cycle in $J'$ equals $0$, where we add an edge weight with a negative sign if that edge is traversed backwards. Since the cycles around the bounded faces  $f_i, i \in \II{2, m}, f_j', j \in \II{2, n}$  span the entire cycle space of $J'$, it suffices to check that the sum of weights around these cycles is $0$. Now, for any $i\in \II{2, m}$ the sum of edges around the face $f_i$ only involves the variables $x_{i-1}, x_i$ from the set $\{x_1, \dots, x_m, y_1, \dots, y_n\}$, and for any $j \in \II{2, n}$, the sum of edges around the face $f_j'$ only involves the variables $y_{j-1}, y_j$. Therefore given any choice of $(x_1, y_1) \in \R^2$ we can recursively define $x_i, y_j$ for $i \ge 2, j \ge 2$ to get a solution to the set of equations \eqref{E:xi-dZ} that holds for all admissible $(i, j), (i', j')$. Hence this solution set has dimension at least $2$.
\end{claimproof}

	Now let $J^1, \dots, J^{c(G)}$ be the components of $G$. By the above claim, the set $W^*(Z)$ of all $(\bx, \by)$ that satisfy \eqref{E:xi-dZ} whenever $(i, j), (i', j')$ are admissible and in the same block $I_p(J_a) \X I_q(J_a)$ has dimension 
	$
	2c(G).
	$
	Moreover, $W^*(Z)$ is invariant under the map
	\begin{equation}
	\label{E:xy}
	(\bx, \by) \mapsto (\bx + \bz, \by + \bw)
	\end{equation}
	for any vectors $\bz \in \R^k, \bw \in \R^\ell$ where the  indices of $\bz$ are constant on each of the blocks $I_p(J_a)$ and the indices of $\bw$ are constant on each of the blocks $I_q(J_a)$.
	In order for $(\bx, \by)$ to satisfy \eqref{E:xi-dZ} for all admissible $(i, j), (i', j')$, making an arbitrary choice of an admissible $(i_a, j_a)$ in $I_p(J_a) \X I_q(J_a)$ for all $a \in \II{1, c(G) - 1}$, it is enough to impose the following additional $c(G)-1$ constraints: 
	\begin{equation*}
	x_{i_{a-1}}+ d_Z(p_{i_{a-1}}, q_{j_{a-1}}) +y_{j_{a-1}}= x_{i_{a}}+ d_Z(p_{i_{a}}, q_{j_{a}}) +y_{i_{a}}, \quad a \in \II{2, c(G)}.
	\end{equation*}
	By using the invariance \eqref{E:xy}, we can see that this reduces the dimension of the solution set by exactly $c(G) -1$, yielding the desired result.
\end{proof}

\begin{corollary}
\label{C:dimension-interior-forests}
In the context of Lemma \ref{L:face-to-linear-relations}, suppose that $G = (V, E)$ is an interior forest for some $G' = (V', E') \in \sG$. Then $c(G) = (k + \ell - |V'|)/2+ 1$, so in the final bullet we have $\dim W(Z) = (k + \ell - |V'|)/2 + 2$.
\end{corollary}

\begin{proof}
We have $|V'| = |V| - k - \ell + 2$ by the construction of $G$ from $G'$. Also, $c(G) = |V| - |E|$ since $G$ is a forest. Finally, since all vertices in $V^\circ$ have degree $3$, and the $k$ source and $\ell$ sink vertices have degree $1$, we have $2|E| = 3 |V| - 2(k + \ell)$. Combining these three formulas gives the result. 
\end{proof}

\subsection{Planting forests}
\label{S:planting}

Consider $\bx \in \R^k_<, \by \in \R^\ell_<$ and a set $S \sset \II{1, k} \X \II{1, \ell}$. Let $G^{0, 1}(\bx, \by, S)$ be the unique candidate graph that can be direct-embedded in $\R^2$ as the set
$$
\Ga^{0, 1}(\bx, \by, S) := \bigcup_{(i, j) \in S} \Ga(x_i, 0; y_j, 1).
$$
(Here recall that $\Ga(x_i, 0; y_i, 1)$ is the geodesic network from $(x_i, 0)$ to $(y_i, 1)$). We also write $G^{s, t}(\bx, \by, S), \Ga^{s, t}(\bx, \by, S)$ if the start and end times are $s, t$.
The goal of this short section is to show that we can realize certain ordered forests as networks of the form $G^{0, 1}(\bx, \by, S)$. This is a necessary step for proving existence of geodesic networks and lower bounds in Theorem \ref{T:Hausdorff-dimension}.

\begin{lemma}
	\label{L:existence-of-forests}
	Let $G = (V, E)$ be a $(k, \ell)$-ordered forest such that all vertices in $V^\circ$ have degree $3$ and let $\al > 0$. Then there exist (random) vectors $\bx \in \Q^k_<, \by \in \Q^\ell_<$ and $\ba \in \R^k, \bb \in \R^\ell$ such that:
	\begin{enumerate}[nosep, label=\arabic*.]
		\item (Forest Condition) $
		G^{0, 1}(\bx, \by, \operatorname{Adm}(G)) = G
		$.
		
		\item (Weight Condition) For all $(i, j) \in \operatorname{Adm}(G)$ we have
		$$
		a_i + \cL(x_i, 0; y_j, 1) + b_j \in (0, 1),
		$$
		and for all $(i, j) \in \operatorname{Adm}(G)^c := (\II{1, k} \X \II{1, \ell}) \smin \operatorname{Adm}(G)$ we have
		$$
		a_i + \cL(x_i, 0; y_j, 1) + b_j < - \al.
		$$
	\end{enumerate}
\end{lemma}

\begin{figure}
	\centering
	\includegraphics[scale=1.0]{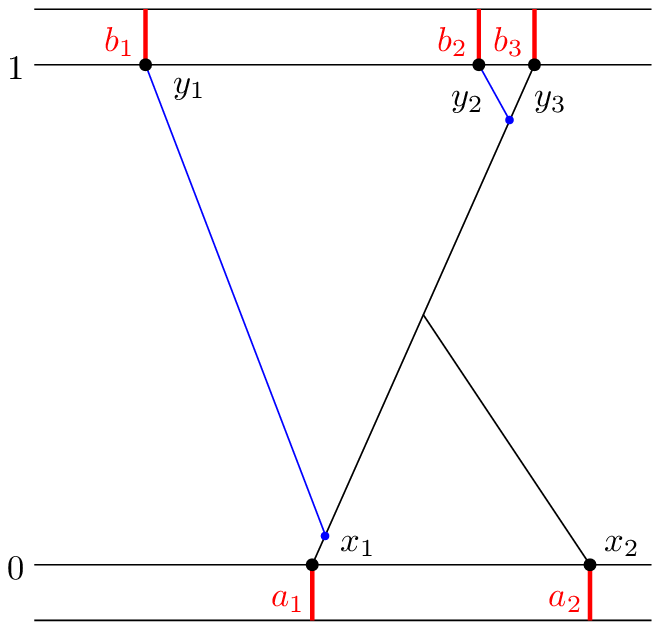}
	\caption{The construction in Lemma \ref{L:existence-of-forests}. The idea of the lemma is that we can find vectors $\bx, \by$ such that if we add in only geodesics from $(x_i, 0)$ to $(y_j, 1)$ for $(i, j) \in \operatorname{Adm}(G)$, then the resulting forest graph is given by $G$. The weight condition implies that if we pad the landscape on the strips $(-\ep, 0)$ and $(1, 1 + \ep)$ with paths of weights $\ba, \bb$ then the paths between $(x_i, -\ep)$ and $(y_j, 1 + \ep)$ will be of comparable weight when $(i, j) \in \operatorname{Adm}(G)$, but highly inferior $(i, j) \notin \operatorname{Adm}(G)$. The forthcoming Lemma \ref{L:sheet-shenanigans} will allows us to create this padding.}
	\label{fig:lemma64}
\end{figure}

See Figure \ref{fig:lemma64} for an illustration of the lemma.
\begin{proof}
	Throughout the proof, when $x, y \in \Q$ we write $\pi^{x, y}$ for the unique geodesic from $(x, 0)$ to $(y, 1)$. We also assume that none of the points $(x, 0), (y, 1)$ with $x, y \in \Q$ are $2$-star points; this is an almost sure event since $\operatorname{Star}_2$ is a lower-dimensional set (Theorem \ref{T:star-computation}). Finally, for $\bx \in \Q^k_<, \by \in \Q^\ell_<$ and $S \sset \II{1, k} \X \II{1, \ell}$ we write $S_{\bx, \by} = \{(x_i, y_j) : (i, j) \in S\}$. Condition $2$ in the lemma can be equivalently phrased in terms of $\operatorname{Adm}(G)_{\bx, \by}$ and $\operatorname{Adm}(G)^c_{\bx, \by}$.

	\textbf{Case 1: $G$ is a tree.} \qquad We will prove the result in this case by induction on $k + \ell$. The base case when $k=\ell = 1$ holds almost surely with $x_1 = 0, y_1 = 0, a_1 = - \cL(0,0;0,1)$ and $b_1 = 1/2$.
	Now fix $m \ge 3$, suppose the result holds when $k + \ell = m - 1$, and consider a $(k, \ell)$-ordered tree $G$ where $k + \ell = m$.
	
	Since $|V| = |E| + 1 = (\sum_{v \in V} \deg(v))/2 + 1$ for a tree, source and sink vertices in $G$ have degree $1$, and all interior vertices have degree $3$, we must have that $|V^\circ|= m - 2 = k + \ell - 2$. Note that $m-2 \ge 1$. Since all source and sink vertices connect to an interior vertex, the pigeonhole principle implies that at there is a vertex $v \in V^\circ$ that connects to two vertices $v_1, v_2$ in $\del V$. Without loss of generality, we assume $v$ has two outgoing edges and one incoming edge, and that one of the outgoing edges is to $v_1$, so necessarily $v_1 \in \del V^+$. 
	We let $v_3$ be the third vertex connected to $v$ by an edge. We now define a $(k, \ell -1)$-ordered tree $G^-$ as follows. There are two cases:
	\begin{enumerate}[nosep, label=\alph*.]
		\item First suppose $v_1, v_2 \in \del V^+$.  Define $G^-$ by removing the vertices $v, v_1$ and the edges $(v, v_1), (v, v_2), (v_3, v)$, and adding in a new edge $(v_3, v_2)$.
		\item Suppose $v_1 \in \del V^+, v_2 \in \del V^-$. In this case we may assume that $v_3 \notin \del V^+$ since otherwise we are in case $1$ after switching the labels on $v_2, v_3$. Then the three edges incident to $v$ are $(v_2, v), (v, v_1), (v, v_3)$.
		Define $G^-$ by removing the vertices $v, v_1$ and the edges $(v_2, v), (v, v_1), (v, v_3)$, and adding in a new edge $(v_2, v_3)$.
	\end{enumerate}
In the example in Figure \ref{fig:lemma64}, we can create $G'$ from $G$ by removing either of the blue edges: the short edge corresponds to Case 1a and the long edge to Case 1b.
	In both cases we have $\del V^+(G^-) = \del V^+ \smin \{v_1\}$ and $\del V^-(G^-) = \del V^-$. 
	Now, in both of the above cases, by the inductive hypothesis we can find vectors $\bx \in \Q^k_<, \by \in \Q^{\ell - 1}_<, \ba \in \R^k, \bb \in \R^{\ell-1}$ such that conditions $1$ and $2$ in the lemma hold for $\bx, \by, \ba, \bb, G^-$.
	Let $\ep > 0$ be small enough so that there are no interior vertices in $\Ga^{0, 1}(\bx, \by, \operatorname{Adm}(G^-)) \cap (\R \X (\ep, 1 - \ep))$.
	
	\textbf{Case 1a:  $v_1, v_2 \in \del V^+$.} \qquad Let $y_i$ be such that $(y_i, 1)$ is identified with $v_2$ in $\Ga^{0, 1}(\bx, \by, \operatorname{Adm}(G^-))$. Let 
	\begin{align*}
\hat \bb &= (b_1, \dots, b_i, b_i, b_{i+1}, \dots, b_{\ell - 1}) \in \R^\ell, \quad \mathand \\ 
\by[y] &= (y_1, \dots, y_i, y, y_{i+1}, \dots, y_{\ell - 1}) \in \Q^\ell_<
	\end{align*}
	for $y \in (y_i, y_{i+1})$. Then 
	$$
	\operatorname{Adm}(G)_{\bx, \by[y]} = \operatorname{Adm}(G^-)_{\bx, \by} \cup \{(x_i, y) : (x_i, y_i) \in \operatorname{Adm}(G^-)_{\bx, \by} \}.
	$$
	By continuity of $\cL$ and condition $2$ for $G^-, \bx, \by$, there exists $\ep_1 \in (0, 1)$ such that if $0 < y - y_i < \ep_1$, then condition $2$ in the lemma holds for $G, \bx, \by[y], \ba, \hat \bb$. Moreover, Proposition \ref{P:overlap-metric}.2 implies that there exists $\ep_2 > 0$ such that for all rational points $y$ with $0 < y - y_i < \ep_2$, we have $\pi^{x_j, y}|_{[0, 1- \ep]} = \pi^{x_j, y_i}|_{[0, 1- \ep]}$ for all $j \in \II{1, k}$. Therefore by condition $1$ for $G^-, \bx, \by$ we have
	$
	G^{0, 1}(\bx, \by[y], \operatorname{Adm}(G)) = G
	$
	whenever $0 < y - y_i < \ep_2$, so for $y \in (y_i, y_i + \ep_1 \wedge \ep_2) \cap \Q$, the quadruple $\bx, \by[y], \ba, \hat \bb$ satisfies the lemma.
	
	\textbf{Case 1b:  $v_1\in \del V^+, v_2\in \del V^-$.} \qquad Let $x_i$ be such that $(x_i, 0)$ is identified with $v_2$. We claim $i \in \{1, k\}$. Indeed, because of how $v_2$ was formed, we know that in a directed embedding of $G$ into $\R \X [0, 1]$, that there is a face touching both $(x_i, 0)$ and the line $\R \X \{1\}$. This can only happen if $i \in \{1, k\}$, since otherwise $G$ would be disconnected. 
	Without loss of generality, assume $i = k$, and for $y > y_{\ell - 1}$ let $\by[y] = (y_1, \dots, y_{\ell-1}, y)$ and $$
	\bb[y] = (b_1, \dots, b_{\ell-1}, -a_k -\cL(x_k, 0; 0, y) + 1/2).
	$$
	In this case
		$$
	\operatorname{Adm}(G)_{\bx, \by[y]} = \operatorname{Adm}(G^-)_{\bx, \by} \cup \{(x_k, y) \}.
	$$
	Then for all large enough $y$, the landscape shape theorem (Theorem \ref{T:landscape-shape}) and condition $2$ for $G^-, \bx, \by, \ba, \bb$ guarantees that condition $2$ holds for $G, \bx, \by[y], \ba, \bb[y]$.
	For condition $1$, Theorem \ref{T:landscape-shape} again guarantees that for large enough rational $y$ we have  $\pi^{x_k, y}(r) > \pi^{x_i, y_j}(r)$ for all $r \in [\ep, 1]$ and $(i, j) \in \II{1, k} \X \II{1, \ell - 1}$. On the other hand, since $x_k$ is not a $2$-star point, $\pi^{x_k, y}(r) = \pi^{x_k, y_{\ell -1}}(r)$ for some $r > 0$, and hence $O(\pi^{x_k, y}, \pi^{x_k, y_{\ell -1}}) = [0, \de]$ for some $\de \in (0, \ep)$. Therefore  for large enough rational $y$ we have
	$
	G^{0, 1}(\bx, \by[y], \operatorname{Adm}(G)) = G,
	$
	and so for large enough rational $y$, the quadruple $\bx, \by[y], \ba, \bb[y]$ satisfies the lemma.
	
	\textbf{Case 2: $G$ is a forest.} \qquad Let $G_1, \dots, G_m$ be the components of $G$, listed so that if $p_{j_1} \in G_i, p_{j_2} \in G_{i+1}$, then $j_1 < j_2$. Let $k_i, \ell_i$ be such that each $G_i$ is a $(k_i, \ell_i)$-tree. Then for each $G_i$ there exist $\bx^i \in \Q^{k_i}_<, \by_i \in \Q^{\ell_i}_<, \ba^i \in \R^{\ell_i}, \bb \in \R^{\ell_i}$ such that with positive probability both conditions in the lemma hold for $\bx^i, \by^i, \ba^i, \bb^i$. Then by the landscape shape theorem (Theorem \ref{T:landscape-shape}) and since $\cL$ is strong mixing with respect to the spatial shift (Remark \ref{R:ergodicity}), there exists a random vector $\bz \in \Z^m_<$ such that the conditions of the lemma hold for 
	\begin{align*}
		(\bx^1 + z_1 {\bf 1}^{k_1}, \dots, \bx^m + z_m {\bf 1}^{k_m}),& \quad (\by^1 + z_1 {\bf 1}^{\ell_1}, \dots, \by^m + z_m {\bf 1}^{\ell_m}), \\
		 \quad (\ba^1, \dots, \ba^m),& \quad (\bb^1, \dots, \bb^m).
	\end{align*}
	Here ${\bf 1}^n = (1, \dots, 1) \in \R^n$.
\end{proof}

\subsection{Tools for the Airy sheet}
\label{S:sheet-tools}
	
In order to cut and paste together geodesic stars and interior forests, we need two technical lemmas about the Airy sheet $\cS$. Both of these are proven using the Airy line ensemble representation. The first lemma plants good regions in the Airy sheet and will allow us to locate geodesics.
\begin{lemma}
	\label{L:sheet-shenanigans}
	Let $\cS_\ga$ be an Airy sheet of scale $\ga > 0$.
	For $\bx \in \R^k_<, \ep, \beta > 0, \bw \in \R^k$ let $T(\cS; \bx, \beta, \bw, \ep, \ga)$ be the event where:
	\begin{enumerate}[label=\arabic*.]
		\item For all $x, y \in \R$ we have
		$$
		\cS_\ga(x, y) + \frac{(x-y)^2}{\ga^3} \le \beta \log^2(3 + |x|/\ga^2 + |y|/\ga^2) + \sum_{i=1}^k w_i \indic(|y-x_i| < \ep).
		$$
		\item $\cS_\ga(x, x_i) \in [w_i, w_i + \beta]$ for all $i \in \II{1, k}$ and $|x - x_i| \le \ga^2$. 
	\end{enumerate}
	 Then there exists $c_k' > 0$ such that for all $\bx \in \R^k_<, \bw \in (c_k' \ga, \infty)^k,$ $\be > c_k \ga, \ep > 0,$ we have $\P(T(\cS; \bx, \beta, \bw, \ep, \ga)) > 0$. 
\end{lemma}

\begin{proof}
	First, by Airy sheet rescaling (Lemma \ref{L:invariance}) it suffices to prove the lemma when $\ga = 1$.
By translation invariance of $\cS$, we may assume $x_1 > 1$.
First, as long as $c_k > 0$ is large enough and $\de \in (0, 1)$ is small enough given $\bx$, the following events hold simultaneously with probability at least $1/2$:
	\begin{enumerate}[nosep, label=(\Roman*)]
		\item $|\cS(x, y) + (x-y)^2| \le (c_k/2) \log^2(3 + |x| + |y|)$ for any $(x, y) \in \R^2$. This uses the landscape shape theorem (Theorem \ref{T:landscape-shape}).
		\item $\cS(x, x_i) \in [-c_k/4, c_k/4]$ for all $i \in \II{1, k}, |x - x_i| \le 1$. This uses spatial stationarity of $\cS$ (Lemma \ref{L:invariance}.2).
		\item Let $x_0 = 0$. Then for all $i, j \in \II{0, k}$ and all $(x, y), (x', y') \in \R^2$ with $|x - x_i| \le 1, |y-x_j| \le 1$ and $\|(x, y) - (x', y')\|_2 \le 2\de$ we have
		$$
		|\cS(x, y) - \cS(x', y')| \le \|(x, y) - (x', y')\|_2^{1/3}.
		$$
		This uses the modulus of continuity, Proposition \ref{P:mod-land-i}.
	\end{enumerate}
	Let $A(\bx, \de)$ denote the event where these conditions hold, and let $\fA$ be coupled to $\cS$ as in Definition \ref{D:Airy-sheet}. Next, let $\de \in (0, \ep)$ be small enough so that the intervals $(x_i - \de, x_i + \de)$ are all disjoint and assume $w_i \ge c_k/2$ for all $i$. Define $f^\de:\R\to \R$ by the following rules:
	\begin{itemize}[nosep]
		\item $f^\de(x) = 0$ for $x \notin U(\bx, \de) := \bigcup_{i=1}^k (x_i - \de, x_i + \de)$.
		\item $f^\de(x_i) = w_i + c_k/2$ for all $i$.
		\item $f^\de$ is linear on each of the intervals $[x_i - \de, x_i], [x_i,-x_i + \de]$.
	\end{itemize}
	Define $\fA^\de = (\fA_1 + f^\de, \fA_2, \fA_3, \dots)$. 
	Since $f^\de$ is a non-negative Lipschitz function,and equals $0$ outside of the compact set $[x_1 - \de, x_k + \de]$, by the Brownian Gibbs property for $\fA$ (Lemma \ref{L:BG-property}) applied to the set $\{1\} \X [x_1 - \de, x_k + \de]$, the law of $\fA^\de$ is absolutely continuous with respect to the law of $\fA$ for every fixed $\de$. 
	Therefore letting $\cS^\de$ be defined from $\fA^\de$ as in Definition \ref{D:Airy-sheet}, the law of $\cS^\de$ is absolutely continuous with respect to the law of $\cS$. We claim that for $\de \in (0, \ep)$ small enough and $\beta > c_k$, on the event $A(\bx, \de)$, the event $T(\cS^\de; \bx, \ep, \beta, \bw, 1)$ holds.
	
\textbf{Verifying Condition $1$.} \qquad Let $x \ge 0, y \in \R$. For any parabolic path $\pi$ from $(x, \infty)$ to $(y, 1)$, we have 
	\begin{align}
	\nonumber
		\|\pi\|_{\fA^\de} &= \cS^\de(x,x_1 - \de) + \lim_{z \to -\infty} \lf( \|\pi|_{[z,y]}\|_{\fA^\de} - {\fA^\de}[(z, \pi(z)) \to (x_1 - \de, 1)]\rg) \\
		\nonumber
		&= \cS(x,x_1 - \de) + \lim_{z \to -\infty} \lf( \|\pi|_{[z,y]}\|_{\fA^\de} - \fA[(z, \pi(z)) \to (x_1 - \de, 1)]\rg) \\
		\label{E:pi-fA-difference}
		&= \|\pi\|_{\fA} + \|\pi|_{[x_1 - \de,y]}\|_{\fA^\de} - \|\pi|_{[x_1 - \de,y]}\|_{\fA}.
	\end{align}
	Here the first equality uses the comment after \eqref{E:piSW}, and the remaining equalities use that $\fA = \fA^\de$ outside of of $\{1\} \X U(\bx, \de)$. Now, if  $y \notin U(\bx, \de)$ then $\|\pi|_{[x_1 - \de,y]}\|_{\fA^\de} \le \|\pi|_{[x_1 - \de,y]}\|_{\fA}$, and if $y \in [x_i - \de, x_i + \de]$ then $\|\pi|_{[x_1 - \de,y]}\|_{\fA^\de} \le \|\pi|_{[x_1 - \de,y]}\|_{\fA} + w_i + c_k/2$. Therefore for $x \ge 0, y \in \R$, by Lemma \ref{L:basics-para-path}(i) we have
	\begin{equation}
	\label{E:cSal-de}
	\cS^\de(x, y) \le \cS(x, y) +\sum_{i=1}^k (w_i + c_k/2)\indic(|x-x_i| < \de).
	\end{equation}
	Similar reasoning for $x \le 0$ using the Airy line ensemble $\tilde \fA(y) := \fA(-y)$ yields the same bound for all $x, y \in \R$. Combined with the bound in condition (I) for $\cS$, this yields condition $1$ of the lemma.
	
	\textbf{Verifying Condition $2$.} \qquad By the quadrangle inequality \eqref{E:quadrangle} for the Airy sheet, using that $x_1 > 1$, for any $i \in \II{1, k}$ and $|x-x_i| \le 1$ we have
	\begin{align}
	\nonumber
\cS^\de(x, x_i) &\ge \cS^\de(x, x_i - \de) + \cS^\de(0, x_i) - \cS^\de(0, x_i-\de) \\
\nonumber
&= \cS(x, x_i - \de) + \cS(0, x_i) + w_i + c_k/2 - \cS(0, x_i-\de) \\
\label{E:S-de-bound}
&\ge \cS(x, x_i) + w_i + c_k/2 - \de^{1/3} \ge w_i.
	\end{align}
Here the equality uses that $\cS(0, \cdot) = \fA_1(\cdot)$, $\cS^\de(0, \cdot) = \fA^\de_1(\cdot)$ and that $\cS(x, w) = \cS^\de(x, w)$ for $x \ge 0, w \le x_1 - \de$ by Definition \ref{D:Airy-sheet}. The second inequality uses condition (III) for the event $A(\bx, \de)$, and the final inequality uses condition (II). This yields the lower bound in condition $2$ in the lemma. On the other hand, \eqref{E:cSal-de} gives that $\cS^\de(x_i, x_i) \le \cS(x_i, x_i) + w_i + c_k/2$, which is bounded above by $w_i + c_k$ by (II), giving the upper bound in condition $2$.
\end{proof}

The next lemma gives a weak anti-concentration estimate for certain functions of the Airy sheet.

\begin{lemma}
	\label{L:resampling-part-a}
	Let $I_1 = [a_1, b_1], \dots, I_k = [a_k, b_k]$ be compact intervals with $b_i < a_{i+1}$ for all $i \in \II{1, k-1}$, and let $f_i:I_i \to \R, i \in \II{1, k}$ be continuous functions. Similarly let $I_1' = [a_1', b_1'], \dots, I_k' = [a_k', b_k']$ be compact intervals with $b_i' < a_{i+1}'$ for all $i \in \II{1, k-1}$, and let $g_i:I_i' \to \R, i \in \II{1, k}$ be continuous functions. Then the vector $X \in \R^n$ with
	$$
	X_i = \max_{x \in I_i, y \in I_i'} f_i(x) + \cS(x, y) + g_i(y)
	$$
	has a Lebesgue density.
\end{lemma}

\begin{proof}
	By translation invariance of $\cS$, we may assume $a_1, a_1' > 0$. First, by induction it is enough to show that for all $j \in \II{1, k}$, conditional on the random variables $X_1, \dots, X_{j-1}$, the distribution of $X_j$ almost surely has a Lebesgue density.
	Let $\fA, \cS$ be coupled as in Definition \ref{D:Airy-sheet}. By Lemma \ref{L:basics-para-path}(iii) we have
	\begin{align*}
	X_j &= \max_{x \in I_j, y \in I_j', n \in \N} 
	f_j(x) + \fA[(x, \infty) \to (b_{j-1}', n)] + \fA[(b_{j-1}', n) \to (y, 1)] + g_j(y).
	\end{align*}
	It is important that this is a maximum, rather than a supremum. We can then write $X_j = \max_{n \in \N} W_n + Z_n$, where
	where
	$$
	W_n = \max_{x \in I_j} f_j(x) + \fA[(x, \infty) \to (b_{j-1}', n)], \qquad Z_n = \max_{y \in I_j'} \fA[(b_{j-1}', n) \to (y, 1)] + g_j(y).
	$$
	Therefore it is enough to show that conditional on $X_1, \dots, X_{j-1}$, $W_n + Z_n$ has a Lebesgue density for all $n$. From Definition \ref{D:Airy-sheet}, we see that $\sig(X_1, \dots, X_{j-1}, W_n) \sset \sig(\fA|_{(-\infty, b_{j-1}']})$, so it is enough to show that conditional on $\sig(\fA|_{(-\infty, b_{j-1}']})$, $Z_n$ has a Lebesgue density. By the Brownian Gibbs property applied to the set $K := \II{1, n} \X [b_{j-1}', b_j' + 1]$, 
	given $\fA|_{(-\infty, b_{j-1}']}$ we have 
	$$
	\fA|_K \ll \fA|_K + f^U
	$$
	where $U$ is a uniform random variable on $[0, 1]$ independent of $\fA$, and $f^U:K \to \R$ is the function whose components $f^U_1, \dots, f^U_n$ are all identical, piecewise linear functions with pieces defined on $[b_{j-1}', a_j'], [a_j', b_j'], [b_j', b_j' + 1]$, and satisfy 
	$$
	f^N_\ell(b_{j-1}') = f^N_\ell(b_j' + 1) = 0, \quad  f^N_\ell(a_j') = f^N_\ell(b_j') = U, \qquad \ell \in \II{1, n}.
	$$	
	Since last passage commutes with shifts by a common function, for $y \in I_j'$ we have
	\begin{align*}
	(\fA + f^U)[(b_{j-1}', n) \to (y, 1)] = \fA[(b_{j-1}', n) \to (z, 1)] + U
	\end{align*}
	and so $Z_n \ll Z_n + U$ conditional on $\sig(\fA|_{(-\infty, b_{j-1}']})$, which yields the result.
\end{proof}

\subsection{Cut and paste: the proofs of Theorem \ref{T:geodesic-networks} and parts $1$ and $2$ of Theorem \ref{T:Hausdorff-dimension}}
\label{S:cut-and-paste}

The proofs of Theorem \ref{T:geodesic-networks} and parts $1$ and $2$ of Theorem \ref{T:Hausdorff-dimension} are divided into two parts: an upper bound (Proposition \ref{P:upper-bd}) based on cutting networks into components stars connected by an interior forest, and a lower bound (Proposition \ref{P:lower-bd}) based on pasting these pieces back together.

\begin{prop}
	\label{P:upper-bd}
	For all $G \in \sG$, a.s. we have
	$
	\dim_{1:2:3}(N_\cL(G)) \le d(G),
	$
	and if $d(G) = 0$, then $N_\cL(G)$ is at most countable.
\end{prop}

To prove Proposition \ref{P:upper-bd}, the basic idea is to cut up a potential geodesic network $\Ga(p; q)$ into a pair of geodesic stars at $p$ and $q$ and an interior forest for $G$. We can then appeal to Lemma \ref{L:face-to-linear-relations} and Corollary \ref{C:dimension-interior-forests} to show that the weight vector of the geodesic stars at $p$ and $q$ must lie in some $(k+\ell - |V|)/2 + 2$ dimensional linear subspace $W$. The set of pairs $(p; q)$ that support such stars has dimension at most $d(G)$ by Proposition \ref{P:weighted-stars-upper-bound}.
 Making this argument rigorous is a fairly delicate process. Indeed, we must ensure that we only need a bound for countably many distinct linear subspaces $W$; this relies on coalescence properties. Moreover, each of these subspaces cannot have a strong dependence on the weight of the geodesic stars; for this we will use Lemma \ref{L:resampling-part-a}. These difficulties are the reason for the rather involved setup that comprises the first half of the proof.

\begin{proof}
	Fix $G = (V, E) \in \sG$, with source vertex $p$ of degree $k$ and a sink $q$ of degree $\ell$. Without loss of generality assume $k \le \ell$.
	We can write $N_\cL(G) = N_\cL^1(G) \cup N_\cL^2(G)$, where $(p; q) \in N_\cL^1(G)$ if $G(p; q) = G$ and $(p; q) \in N_\cL^1(G)$ if $G(p; q) = G^T$. By time-reversal symmetry of $\cL$, it suffices to prove the proposition with $N_\cL^1(G)$ in place of $N_\cL(G)$.
	
	\begin{figure}
		\centering
		\includegraphics[scale=0.9]{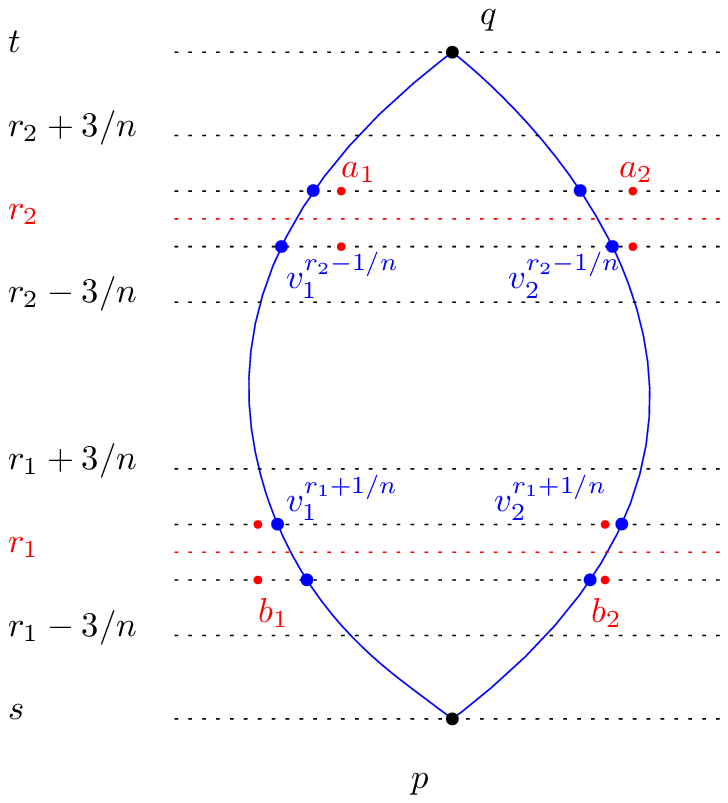}
		\caption{An example of a pair $(p; q) \in N_\cL(G; \br, n, \ba, \bb, S)$. Here $G$ is the two-vertex graph $\{p, q\}$ with a double-edge from $p$ to $q$, and $S = \{(1, 1), (2, 2)\}$.}
		\label{fig:prop681}
	\end{figure}

	\textbf{Step 1: Division into countably many parts.} \qquad We do this in two steps. First, for a rational vector $\ba \in \R^k_<$ we write $\de(\ba) = \min \{a_i - a_{i-1} : i  \in \II{2, k}\}$. We can write
	$$
	N_\cL^1(G) = \bigcup_{\substack{\br \in \Q^2_<, n \in \N \text{ s.t. } r_1 + 3/n < r_2 - 3/n
	\\
	\ba \in \Q^k_<, \bb \in \Q^\ell_<, S \sset \II{1, k} \X\II{1, \ell}}} N_\cL(G; \br, n, \ba, \bb, S),
	$$
	where $ N_\cL(G; \br, n, \ba, \bb, S)$ is the set of all $(p; q) = (x, s; y, t) \in N_\cL^1(G)$ satisfying the following conditions:
	\begin{enumerate}[nosep, label=\arabic*.]
		\item $p \in (-n, n) \X (-\infty, r_1 - 3/n], q \in (-n, n) \X [r_2 + 3/n, \infty)$.
		\item $\Ga(p; q) \sset (-n, n)^2$ and no interior vertices of $\Ga(p; q)$ are contained in the time intervals $(s, r_1 + 3/n)$ and $(r_2 - 3/n, t)$. 
		\item For $r \in (s, r_1 + 3/n)$, let $\bv^r \in \R^{k}_<$ be the vector whose coordinates are the spatial coordinates of elements of $\Ga(p; q) \cap (\R \X \{r\})$. For $r \in (r_2 - 3/n, t)$ similarly define $\bv^r \in \R^\ell_<$.
		 Then
		$$
		\max_{r \in \{r_1 - 1/n, r_1 + 1/n\}} \|\bv^r - \ba\|_\infty < \de(\ba)/4, \qquad \max_{r \in \{r_1 - 1/n, r_1 + 1/n\}} \|\bv^r - \bb\|_\infty < \de(\bb)/4.
		$$
		\item There is a geodesic from $p$ to $q$ through the points $(v^{r_1}_i, r_1)$ an $(v^{r_2}_j, r_2)$ if and only if $(i, j) \in S$. Note that by point $2$, we could equivalently state this with $r_1, r_2$ replaced by any $r_1' \in (s, r_1 + 3/n), r_2' \in (r_2 - 3/n, t)$.
	\end{enumerate}
See Figure \ref{fig:prop681} for an illustration of much of the definition.
By countable stability, it suffices to prove the Hausdorff dimension and countability bounds in the proposition for  $N_\cL(G; \sig_0)$ for each fixed $\sig_0 := (\br, n, \ba, \bb, S)$. We will break up $N_\cL(G; \sig_0)$ once more before proceeding to the bound.

By Proposition \ref{P:sparse-space-2}(i), for $i \in \{1, 2\}$, there are finitely many points
$z^-_{i, 1} <  \dots < z^-_{i, \ell^-(i)}$ such that $(z^-_{i, j}, r_i - 2/n)$ lies on a geodesic between points $p' \in (-n, n) \X \{r_i - 3/n\}$ and $q' \in (-n, n) \X \{r_i - 1/n\}$. Similarly, there are finitely many points
$z^+_{i, 1} <  \dots < z^+_{i, \ell^+(i)}$ such that $(z^+_{i, j}, r_i + 2/n)$ lies on a geodesic between points in $(-n, n) \X \{r_i + 1/n\}$ and $(-n, n) \X \{r_i + 3/n\}$. For $(p; q) \in N_\cL(G; \sig_0)$, for $i = 1, 2$, we can write
$$
\bv^{r_i \pm 2/n} = (z^\pm_{i, j} : j \in I_i^\pm[p, q])
$$ 
where $I_i^\pm[p, q]$ is a subset of $\N$ consisting of $k$ elements if $i=1$ and $\ell$ elements if $i = 2$. 
Next, for all $i \in \II{1, k}$, the point $(v^{r_1 - 2/n}_i, r_1 - 2/n)$ is an interior point on the unique geodesic from $p$ to $(v^{r_1 - 1/n}_i, r_1 - 1/n)$. Therefore by Proposition \ref{P:overlap-metric}, there is a \textit{rational} point $Q_1[p, q](i) \in (-n, n)$ such that $(v^{r_1 - 2/n}_i, r_1 - 2/n)$ is an interior point on a geodesic from $p$ to $(Q_1[p, q](i), r_1 - 1/n)$.
We can similarly define rational points $Q_2[p, q](j) \in (-n, n)$ such that $(v^{r_1 + 2/n}_j, r_2 + 2/n)$ is an interior point on a geodesic from $(Q_2[p, q](j), r_2 + 1/n)$ to $q$ for all $j \in \II{1, \ell}$. See Figure \ref{fig:prop682} for an illustration of these definitions.

Therefore we can divide $N_\cL(G; \sig_0)$ into countably many pieces $N_\cL(G; \sig_0, I, Q)$ where:
\begin{itemize}[nosep]
	\item $I = (I_1^-, I_1^+, I_2^-, I_2^+)$ is any vector consisting of four subsets of $\N$ with $|I_1^\pm| = k, |I_2^\pm| = \ell$. 
	\item $Q = (Q_1, Q_2) \in [\Q \cap (-n, n)]^k_< \X [\Q \cap (-n, n)]^\ell_<$.
	\item If $(p; q) \in N_\cL(G; \sig_0, I, Q)$ then $I_i^\pm = I_i^\pm[p, q]$ and $Q_i=Q_i[p, q]$ for $i = 1, 2$.
\end{itemize} 
It suffices to prove the proposition for $N_\cL(G; \sig)$ for each fixed $\sig = (\sig_0, I, Q)$. From this point forward, we assume that $I^\pm_1 = \II{1, k}, I^\pm_2 = \II{1, \ell}$; the proof in the remaining cases is the same up to relabeling. We may also assume that $N_\cL(G; \sig)$ is non-empty, since otherwise the result holds trivially.

\begin{figure}
	\centering
	\includegraphics[scale=0.9]{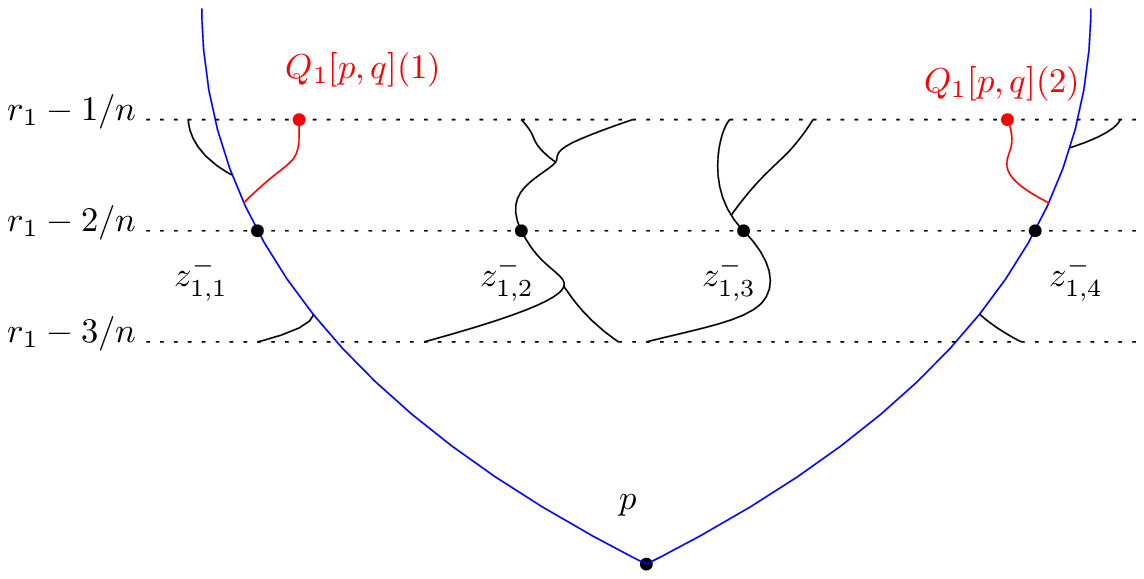}
	\caption{Some of the objects relevant to the definition of $N_\cL(G, \sig)$. In this example, $I^-_1[p, q] = \{1, 4\}$. The blue paths are the start of the geodesic network $\Gamma(p; q)$, the black paths are geodesics from time $r_1 - 3/n$ to $r_1 - 1/n$, and the red paths are parts of the geodesics from $p$ to $Q_1[p, q](1), Q_1[p, q](2)$.}
	\label{fig:prop682}
\end{figure}

	\textbf{Step 2: From $N_\cL(G; \sig)$ to weighted stars.} \qquad To relate $N_\cL(G; \sig)$ to the existence of weighted stars at $p, q$, we use Lemma \ref{L:face-to-linear-relations}. First, for $(i, j) \in S$, if $(p; q) \in N_\cL(G; \sig)$ then
\begin{equation}
\label{E:xi-dZ-applied}
\begin{split}
\cL(p; q) = \cL(p; z^-_{1, i}, r_1 - 2/n) + \cL(z^-_{1, i}, r_1 - 2/n ;z^+_{2, j}, r_2 + 2/n) + \cL(z^+_{2, j}, r_2 + 2/n; q).
\end{split}
\end{equation} 
Using the conditions in $N_\cL(G; \sig)$ on $\ba, \bb, Q$ we can rewrite the right-hand side above. First, for $i = 1, 2$ let 
$$
\cS^i(x, y) = \cL(x, r_i - 1/n; y, r_i + 1/n)
$$
Then, for $i \in \II{1, k}, j \in \II{1, \ell}$ define functions
\begin{align*}
F_{1, i}(x) &=  \cL(z^-_{1, i}, r_1 - 2/n; x, r_1 - 1/n) - \cL(z^-_{1, i}, r_1 - 2/n; Q_1(i), r_1 - 1/n), \\
G_{1, i}(y) &= \cL(y, r_1 + 1/n; z^+_{1, i}, r_1 + 2/n) \\
F_{2, j}(y) &=  \cL(y, r_2 + 1/n; z^+_{2, j}, r_2 + 2/n) - \cL(Q_2(j), r_2 + 1/n;z^+_{2, j}, r_2 + 2/n) \\
G_{2, j}(x) &= \cL(z^-_{2, j}, r_2 - 2/n; x, r_2 - 1/n)
\end{align*}
and let
\begin{align*}
X_i &= \max_{x, y \in [a_i - \de(\ba)/4, a_i + \de(\ba)/4]} F_{1, i}(x) + \cS^1(x, y) + G_{1, i}(y), \\
Y_j &= \max_{x, y\in [b_j - \de(\bb)/4, b_j + \de(\bb)/4]} G_{2, j}(x) + \cS^2(x, y) + F_{2, j}(y).
\end{align*}
Then by metric composition, if $(p, q) \in N_\cL(G; \sig)$, then \eqref{E:xi-dZ-applied} equals
\begin{align*}
\cL(p; Q_1(i), r_1 - 1/n) + X_i + \cL(z^+_{1, i}, r_1 + 2/n ;z^-_{2, j}, r_2 - 2/n) + Y_j + \cL(Q_2(j), r_2 + 1/n; q).
\end{align*}
Now, set $\bz^+_1 = (z^+_{1, 1}, \dots, z^+_{1, k}), \bz^-_2 = (z^-_{2, 1}, \dots, z^-_{2, \ell})$ and 
$
G' = (V', E') = G^{r_1 + 2/n, r_2 - 2/n}(\bz^+_1, \bz^-_2, S);
$
here notation is as in Section \ref{S:planting}. Note that $S = \operatorname{Adm}(G')$, since otherwise $N_\cL(G; \sig)$ would necessarily be empty.
Denote by $p_1, \dots, p_k, q_1, \dots, q_\ell$ the source and sink vertices of $G'$. There is a weight function 
$
Z:E' \to \R
$
such that 
$$
\cL(z^+_{1, i}, r_1 + 2/n ;z^-_{2, j}, r_2 - 2/n) = d_Z(p_i, q_j)
$$
for all $(i, j) \in S$. 
Moreover, let $e^i_p$ be the unique edge in $E'$ emanating from $p_i$, and let $e^j_q$ be the unique edge connecting to $q_j$, and let $f_{X, Y}:E' \to \R$ be the function with $f_{X, Y}(e^i_p) = X_i$ for $i \in \II{1, k}$,  $f_{X, Y}(e^j_q) = Y_j$ for $j \in \II{1, \ell}$, and $f_{X, Y}(e) = 0$ for all other $e \in E'$. Then
$$
d_{Z+ f_{X, Y}}(p_i, q_j) = X_i + \cL(z^+_{1, i}, r_1 + 2/n ;z^-_{2, j}, r_2 - 2/n) + Y_j.
$$
Hence by Lemma \ref{L:face-to-linear-relations}, the weight vector
$$
(\cL(p; Q_1(1), r_1 - 1/n), \dots, \cL(p; Q_1(k), r_1 - 1/n)), \cL(Q_2(1), r_2 + 1/n; q), \dots, \cL(Q_2(\ell), r_2 + 1/n; q))
$$
lies in some $(c(G') + 1)$-dimensional affine space $W = W(Z+ f_{X, Y})$. By the definition of $N_\cL(G; \sig)$, $G'$ is an interior forest for $G$ so $c(G') + 1 = (k + \ell - |V|)/2 + 2$ by Corollary \ref{C:dimension-interior-forests}.
Moreover, there must also exist geodesic stars from $p$ to $(Q_1, r_1 - 1/n)$ and from $(Q_2, r_2 + 1/n)$ to $q$. In summary, using the notation of introduced at the beginning of Section \ref{S:weighted-stars} we have shown that 
$$
N_\cL(G; \sig) \sset \operatorname{Star}(Q_1, r_1 - 1/n ; Q_2, r_2 + 1/n; W). 
$$
From here we will complete the proof using either Proposition \ref{P:weighted-stars-upper-bound} (for the dimension bound) or Lemma \ref{L:upper-bd-3-stars} (for the countability claim when $d(G) = 0$).

\textbf{Step 3: The dimension bound.} \qquad
First, observe that if $Z+ f_{X, Y}$ were independent of the function
\begin{equation}
\label{E:star-function}
\Theta(\cdot) := \operatorname{Star}(Q_1, r_1 - 1/n ; Q_2, r_2 + 1/n; \cdot),
\end{equation}
then by Proposition \ref{P:weighted-stars-upper-bound} and the fact that $\dim (W(Z+ f_{X, Y})) = 2 + (k + \ell - |V|)/2$ we would have that almost surely,
\begin{align*}
\dim_{1:2:3}(N_\cL(G)) &\le \dim_{1:2:3}(\Theta(W(Z+ f_{X, Y})) \\
&\le 12 + \frac{k + \ell - |V|}{2}- \frac{k(k+1)}{2} - \frac{\ell(\ell+ 1)}{2} = d(G).
\end{align*}
  However, there is a subtlety because $Z+ f_{X, Y}$ is not independent of $\Theta$. To deal with this, we use Lemma \ref{L:resampling-part-a} to show that the law of $(\Theta, Z+ f_{X, Y})$ is absolutely continuous with respect to the law of a pair of independent random variables.
  First, observe that $\Theta$ 
is measurable given the $\sig$-algebra $\cF_1$ generated by $\cL$ restricted to the following strips:
\begin{equation}
\label{E:R-na1}
(\R \X (-\infty, r_1 - 1/n])^2 \cap \Rd, \qquad (\R \X [r_2 + 1/n, \infty))^2 \cap \Rd.
\end{equation}
Next, the weight function $Z$ (which depends on the identities of the random points $z^+_{1, i}, z^-_{1, j}$) is measurable given the $\sig$-algebra $\cF_2$ generated by $\cL$ restricted to the set
$
(\R \X [r_1 + 1/n, r_2 - 1/n])^2,
$
and hence is independent of $\Theta$. Next, the functions $F_{1, i}, F_{2, j}, G_{1, i}, G_{2, j}$ are measurable given $\sig(\cF_1, \cF_2)$. Therefore since $\cS^1, \cS^2$ are independent of each other and of $\sig(\cF_1, \cF_2)$, by Lemma \ref{L:resampling-part-a}, given $\sig(\cF_1, \cF_2)$, the vector $(X_1, \dots, X_k, Y_1, \dots, Y_\ell)$ has a Lebesgue density. Therefore if $(N, N') = (N_1, \dots, N_k, N_1', \dots, N_\ell')$ is a standard normal vector, independent of all else, the law of $(\Theta, Z+ f_{X, Y})$ is absolutely continuous with respect to the law of $(\Theta, Z+ f_{N, N'})$ which is the desired result.
	
	\textbf{Step 4: Countability when $d(G) = 0$.} \qquad First, by noting that $k, \ell \le 3$ and using Lemma \ref{L:graph-theory}, we see that the only way to have $d(G) = 0$ is if $k = \ell = 3, |V| = 6$. Therefore $\dim W = 2$. Since $W$ is slide-invariant, recalling the map $L_3$ from Lemma \ref{L:upper-bd-3-stars}, we have that
	$
	W = L_3^{-1}(\bb) \X L_3^{-1}(\bc)
	$ 
	for some $\bb, \bc \in \R^2$. Hence 
	\begin{align*}
	|N_\cL(G; \sig)| &\le |\operatorname{Star}(Q_1, r_1 - 1/n ; Q_2, r_2 + 1/n; L_3^{-1}(\bb) \X L_3^{-1}(\bc))| \\
	&\le |\operatorname{Star}(Q_1, r_1 - 1/n; L_3^{-1}(\bb)) \X \operatorname{Star}_R(Q_2, r_2 + 1/n; L_3^{-1}(\bc))|.
	\end{align*}
Lemma \ref{L:upper-bd-3-stars} then gives that $|N_\cL(G; \sig)| \le 135^2$, as desired. 
\end{proof}

\begin{prop}
	\label{P:lower-bd}
	For all $G \in \sG$, a.s. we have
	$
	\dim_{1:2:3}(N_\cL(G)) \ge d(G),
	$
	and if $d(G) = 0$, then $N_\cL(G)$ is infinite.
\end{prop}

Together with Proposition \ref{P:upper-bd} and the results of Section \ref{S:at-most-27}, Proposition \ref{P:lower-bd} yields Theorem \ref{T:geodesic-networks} and parts $1$ and $2$ of Theorem \ref{T:Hausdorff-dimension}. See Figure \ref{fig:lower-bd} for a sketch of the proof.


	\begin{figure}
	\centering
	\includegraphics[scale=0.9]{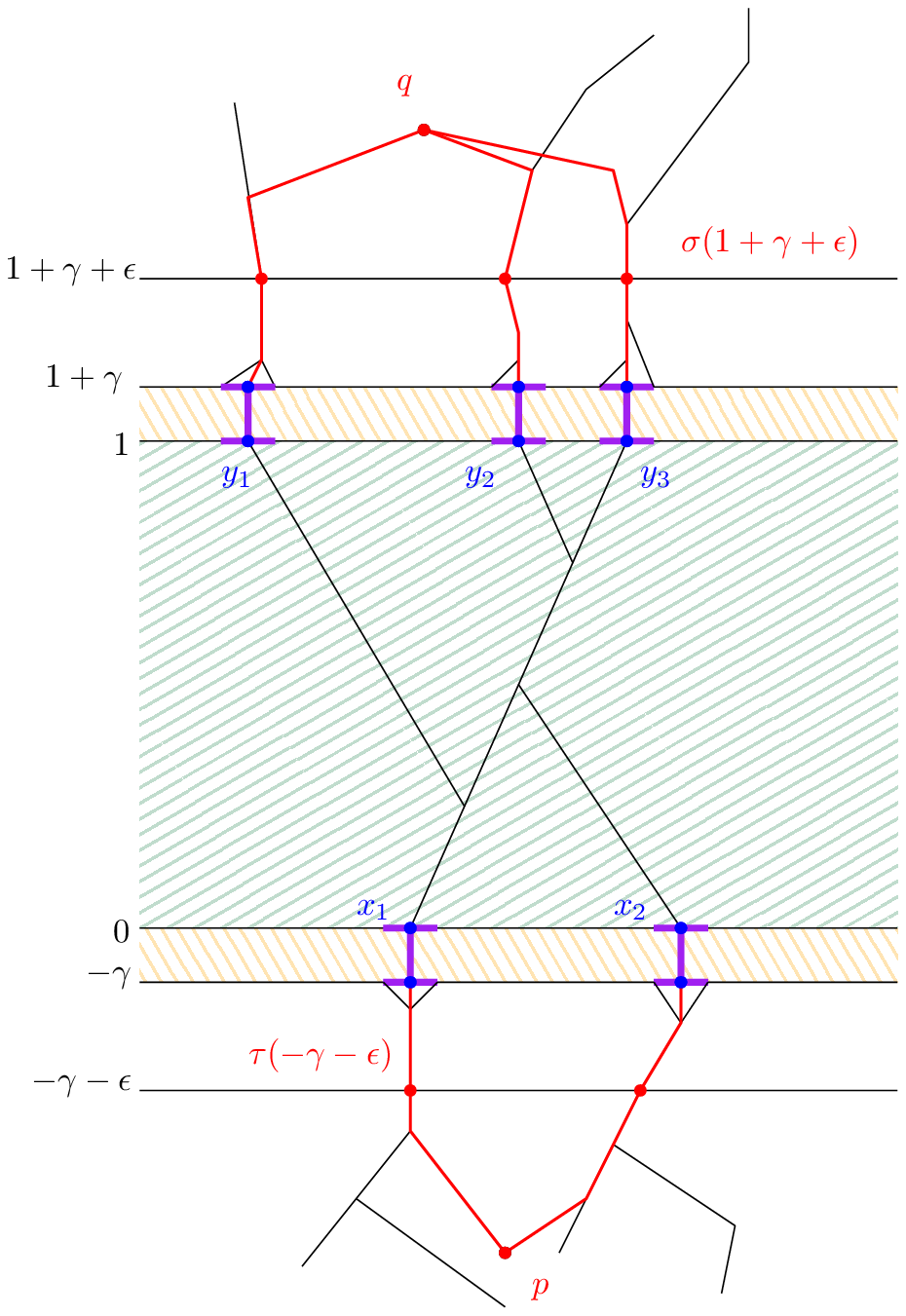}
	\caption{Elements of the construction in the proof of Proposition \ref{P:lower-bd}. The basic idea is that on the set $S$ (shaded in green) we plant an interior forest from $(\bx, 0)$ to $(\by, 1)$ by an application of Lemma \ref{L:existence-of-forests}. On the yellow sets $T_\ga^1, T_\ga^2$, we use Lemma \ref{L:sheet-shenanigans} to plant small favourable regions around the points $x_i, y_j$, indicated in purple. Geodesics moving through the strip $\R \X (-\ga, 1 + \ga)$ will use these regions, and by a coalescence argument will also use essentially follow paths in the planted interior forest. Given this construction, we will look for pairs $(p; q) \in (-\infty, -\ga) \X \R \X (1 + \ga, \infty) \X \R$ which support geodesic stars whose weights are chosen so that every path through the planted forest has equal weight. Proposition \ref{P:star-lower-bd} will yield a lower bound on the dimension of such pairs.}
	\label{fig:lower-bd}
\end{figure}

\begin{proof}
	Fix $G = (V, E) \in \sG$.
	By ergodic properties of $\cL$ (Remark \ref{R:ergodicity}), it is enough to show that with positive probability, \begin{equation}
	\label{E:lower-want}
	\dim_{1:2:3}(N_\cL(G)) \ge d(G) \qquad \text{ and } \qquad N_\cL(G) \ne \emptyset.
	\end{equation} 
	We will construct an event where \eqref{E:lower-want} holds. We first construct events for $\cL$ restricted to each of the five strips
	\begin{align*}
	S := (\R \X [0, 1])^2& \cap \Rd, \qquad T^\ga_1 := (\R \X [-\ga, 0])^2 \cap \Rd, \quad T^\ga_2 := (\R \X [1, 1 +\ga])^2 \cap \Rd, \\
	R^\ga_1 &:=(\R \X (-\infty, -\ga])^2 \cap \Rd, \quad R^\ga_2 := (\R \X [1 +\ga, \infty))^2 \cap \Rd,
	\end{align*}
	where $\ga > 0$ is a small parameter whose value we set later.
	Note that $\cL$ is independent on each of these five strips, so random elements defined on one strip do not affect the distribution on any other. To set these events, let $G'$ be any interior forest for $G$, and let  $\bx \in \Q^k_<, \by \in \Q^\ell_<$ and $\ba \in \R^k, \bb \in \R^\ell$ be defined as in Lemma \ref{L:existence-of-forests} for the forest $G'$ with $\al = 1000$. Note that these random variables only depend on $\cL|_S$. 
	
	Our events will be constructed in terms of small positive parameters $\de, \ep, \ep', \ga$ and large positive parameters $\be, \be', \be''$ that may depend on $\bx, \by$. The relationships between these parameters are given throughout the constructions in \eqref{E:conditions-1}, \eqref{E:conditions-2}, and \eqref{E:conditions-3}.
	
	\textbf{Event $1$: $\cL|_S$.} \qquad Let $A(\de, \ep')$ denote the event where the following conditions hold:
	\begin{enumerate}[nosep, label=(\Roman*)]
		\item We have $x_i - x_{i-1} > \de$ and $y_j - y_{j-1} > \de$ for $i \in \II{2, k}, j \in \II{2, \ell}$. Moreover, $|x_i|, |a_i|, |b_i|, |y_i|$ are bounded above by $\de^{-1}$ for any $i$.
		\item In the embedding of $G'$ as the set
		$$
		\Ga(\bx, \by, G') := \bigcup_{(i, j) \in \operatorname{Adm}(G')} \Ga(x_i, 0; y_j, 1),
		$$
		no vertices are sent to the strips $\R \X (0,\de]$ and $[1-\de, 1) \X \R$.
		\item For every $(i, j) \in \II{1, k} \X \II{1, \ell}$, let $\pi^{i, j}$ denote the unique geodesic from $(x_i, 0)$ to $(y_j, 1)$. Then if $|x-x_i| < \ep'$ and $|y - y_j| < \ep'$, then 
		any geodesic $\pi$ from $(x, 0)$  to $(y, 1)$ satisfies $[\de, 1- \de] \sset O(\pi, \pi^{i, j})$.
	\end{enumerate}
	Conditions (I) and (II) hold with positive probability for all $\de < \de_0$. Condition (III) holds for all $\ep' < \ep'_0(\de)$ by Proposition \ref{P:overlap-metric}(ii) and the continuity of $\cL$. 
	Therefore $\P A(\de, \ep') > 0$ for all
	\begin{equation}
	\label{E:conditions-1}
	\de < \de_0, \qquad  \ep' < \ep'_0(\de).
	\end{equation}
	
	\textbf{Event $2$: $\cL|_{R^\ga_1}, \cL|_{R^\ga_2}$, and a shape constraint.} \qquad
	Define the event $B(\ga, \de, \be, \be', \ep, \ep')$ as follows:
	\begin{enumerate}[nosep, label=\arabic*.]
		\item If $k = 3$, then
		$$
		\cL(0, -\ga - 1; x_i, -\ga)- \cL(0, -\ga - 1; x_2, -\ga) > 200,
		$$
		for $i = 1, 3$. Similarly, if $\ell = 3$, then 
		$$
		\cL(y_i, 1 + \ga; 0, 2 + \ga)- \cL(y_2, 1 + \ga; 0, 2 + \ga) > 200,
		$$
		for $i = 1, 3$. These conditions guarantee that $(-100, 100)^3 \sset S^-(\bx, -\ga)$ if $k = 3$ and $(-100, 100)^3 \sset S^+(\by, 1+ \ga)$ if $\ell = 3$, where $S^-(\bx, -\ga), S^+(\by, 1+ \ga)$ are as in Proposition \ref{P:star-lower-bd}.
		\item Define the compact sets
		$$
		K^1_{\be} = [-\be, \be] \X [-1-\ga, - \be^{-1} - \ga], \qquad K^2_{\be} =[-\be, \be] \X [1 + \be^{-1} + \ga, 2 + \ga]
		$$
		and let $K_{\be} = K^1_{\be} \X K^2_{\be}$.
		For every slide-invariant affine set $W \sset \R^{k + \ell}$ with dimension $(k + \ell - |V|)/2 + 2$ and satisfying
		$
		W \cap (-100, 100)^{k + \ell} \ne \emptyset,
		$
		we have 
		\begin{align*}
		&\dim_{1:2:3}[\operatorname{Star}(\bx, -\ga; \by, 1 + \ga; g(W)) \cap K_{\be}
		\ge d(G),
		\end{align*}
		where $g(W)$ is the smallest slide-invariant set containing $W \cap (-100, 100)^{k + \ell}$.
		\item 
		For all points $p \in K^1_{\be}$, $i \in \II{1, k}$, and $x$ with $|x-x_i| \le \ep'$, any geodesic from $p$ to $(x, -\ga)$ overlaps with the unique geodesic $\tau^i$ from $(0, -2)$ to $(x_i, -\ga)$ at time $-\ga - \ep$. Similarly, for all points $q \in K^2_{\be}$, $i \in \II{1, \ell}$ and $y$ with $|y-y_i| \le \ep'$, any geodesic from $(y, 1+\ga)$ to $q$ overlaps with the unique geodesic $\sig^i$ from $(y_i, 1+\ga)$ to $(0, 2)$ at time $1 + \ga + \ep$.  The vectors $\tau(-\ga-\ep), \sig(1 + \ga + \ep)$ are included in Figure \ref{fig:lower-bd}.
		
		\item For any point $p \in K^1_{\be}$ (or $q \in K^2_{\be}$) the function $\cL_p(y) := \cL(p; y, -\ga)$ (or $\cL_q(y) := \cL(y, 1 + \ga; q)$) is H\"older-$1/3$ with constant $\be'$ on the interval $[-\de^{-1}, \de^{-1}]$. Moreover, for $u = (x, t; y, t + s) \in R^\ga_1 \cup R^\ga_2 \cup S$ we have the general shape estimate
		\begin{equation*}
		\left|\cL(u) + \frac{(x - y)^2}{s} \right|\le \be' s^{1/3} \log^{4/3}\lf(\frac{2(\|u\|_\infty + 2)^{3/2}}{s}\rg)\log^{2/3}(\|u\|_\infty + 2).
		\end{equation*}
	\end{enumerate}
	Given arbitrary $\de, \ga$, we claim that $A(\de, \ep') \cap B(\ga, \de, \be, \be', \ep, \ep')$ holds with positive probability whenever $\P A(\de, \ep') > 0$ and 
	\begin{equation}
	\label{E:conditions-2}
	\be > \be_0, \qquad \be' > \be'_0(\be, \de), \qquad \ep < \ep_0(\be), \qquad \ep' < \ep_0'(\ep, \be).
	\end{equation}
	Condition on the values of $\bx, \by$. Under this condition, point $1$ holds with positive probability by the Brownian Gibbs property (Lemma \ref{L:BG-property}). Point $2$ holds with positive probability given $\bx, \by$ and point $1$ for large enough $\beta$ by Proposition \ref{P:star-lower-bd}. The condition \eqref{E:intersection-condition-star} is simplified because $(-100, 100)^3 \sset S^-(\bx, -\ga)$ if $k = 3$ and $(-100, 100)^3 \sset S^+(\by, 1+ \ga)$ if $\ell = 3$. Point $3$ holds almost surely for all small enough $\ep > 0$ and $\ep'$ small enough given $\ep$ since almost surely, there are no rational $2$-stars points (Theorem \ref{T:star-computation}) and the local compactness of the overlap topology (Proposition \ref{P:overlap-metric}). Given the event $A(\de, \ep')$ and Conditions $1, 2, 3,$
	condition $4$ holds almost surely for large enough $\be' >0$ by Theorem \ref{T:landscape-shape}. Note that by independence and stationarity of time increments that there is no dependence on $\ga$ in \eqref{E:conditions-2}.
	
	\textbf{Event $3$: $\cL|_{T^\ga_1}$ and $\cL|_{T^\ga_2}$.} \qquad Define
	$
	\cS^1_\ga(x, y) = \cL(x, -\ga; y, 0), \cS^2_\ga(x, y) = \cL(y, 1; x, 1 + \ga),
	$
	and with notation as in Lemma \ref{L:sheet-shenanigans} let 
	$$
	C(\ga, \ep', \be'') = T(\cS^1_\ga; \bx, 1, \ba + \be'' \mathbf{1}, \ep', \ga^{1/3}) \cap T(\cS^2_\ga; \by, 1, \bb + \be'' \mathbf{1}, \ep', \ga^{1/3}).
	$$
Lemma \ref{L:sheet-shenanigans} and landscape rescaling guarantees that given $\bx, \by$, the event $C(\ga, \ep, \be'')$ holds with positive probability for all 
	\begin{equation}
	\label{E:conditions-3}
	\ga < \ga_0, \qquad \be'' > 1.
	\end{equation}
	\textbf{Final Step: Constructing the networks.} \qquad We work on the event
	$$
	A(\de, \ep') \cap B(\ga, \de, \be, \be', \ep, \ep') \cap C(\ga, \ep', \be'')
	$$
	which holds with positive probability under conditions \eqref{E:conditions-1}, \eqref{E:conditions-2}, and \eqref{E:conditions-3}. 
	We may also assume that the parameters have the relative sizes
	\begin{equation}
	\label{E:conditions-4}
	(\be'')^{-1} \ll \ga \ll \ep' \ll \ep \ll \de \wedge \be^{-1} \wedge (\be')^{-1}.
	\end{equation}
	First, for every $(p; q) \in \Rd$, let $\cL^i_j(p; q)$ be the supremum of $\|\pi\|_\cL$ over paths $\pi$ from $p$ to $q$ satisfying
	\begin{equation}
	\label{E:basic-element}
	\max(|\pi(-\ga) - x_i| , |\pi(0) - x_i|) \le \ep', \qquad \max(|\pi(1) - y_j| , |\pi(1 + \ga) - y_j|)\le \ep'.
	\end{equation} 
	We write $\cL^i(p; q)$ if we only impose the first inequality and $\cL_j(p; q)$ if we only impose the second. We will only use these definitions when the relevant times $-\ga, 0, 1, 1 + \ga$ are in the domain of $\pi$. First, by the assumptions \eqref{E:conditions-4}, the shape conditions for $\cS^1_\ga, \cS^2_\ga$ imposed by $C(\ga, \ep', \be'')$, and the shape condition in $4$, for any $(p; q) \in K_\be$ we have
\begin{equation}
\label{E:max-over-ij}
\cL(p; q) = \max_{i \in \II{1, k}, j \in \II{1, \ell}} \cL^i_j(p; q).
\end{equation}
Next, by upper semi-continuity of path length, the supremum $\cL^i_j(p; q)$ is always attained by at least one path $\pi^{i, j}_{p, q}$. Moreover, $\pi^{i, j}_{p, q}$ is a geodesic on each of the five intervals 
	$$
	[s, -\ga], \quad [-\ga, 0],\quad  [0, 1],\quad  [1, 1 + \ga], \quad  [1 + \ga, t].
	$$
	In particular, if $(p; q) \in K_\be$, then by points (III), $3$ we have
	\begin{align}
	\label{E:piij-conditions}
	&\pi^{i, j}_{p, q}(- \ga - \ep) = \tau^i(-\ga - \ep), \qquad &&\pi^{i, j}_{p, q}(1 + \ga + \ep) = \sig^j(1 + \ga + \ep) \\
	&\pi^{i, j}_{p, q}(\de) = \pi^{i, j}(\de), \qquad &&\pi^{i, j}_{p, q}(1 - \de) =  \pi^{i, j}(1-\de),
	\end{align}
	and so $\cL^{i, j}(p; q)$ is equal to
	\begin{equation}
	\label{E:split-L-1}
	\begin{split}
	 \cL(p; \bar \tau^i(-\ga - \ep)) &+ \cL^i(\bar \tau^i(-\ga - \ep); \bar \pi^{i, j}(\de)) + \cL(\bar \pi^{i, j}(\de), \bar \pi^{i, j}(1- \de)) \\
	 &+ \cL_j(\bar \pi^{i, j}(1- \de); \bar \sig^j(1 + \ga + \ep)) + \cL(\bar \sig^j(1 + \ga + \ep); q).
	\end{split}
	\end{equation}
	Now, define
\begin{align*}
X_i &= \cL^i(\bar \tau^i(-\ga - \ep); \bar \pi^{i, j}(\de)) - \cL(\bar \tau^i(-\ga - \ep); x_i, -\ga) - \cL(x_i, 0; \bar \pi^{i, j}(\de)), \\
\qquad Y_j &= \cL_j(\bar \pi^{i, j}(1- \de); \bar \sig^j(1 + \ga + \ep)) - \cL(y_j, 1 + \ga; \bar \sig^j(1 + \ga + \ep)) -  \cL(\bar \pi^{i, j}(1-\de), y_j, 1).
\end{align*}
Then by points $3$ and (III) again we can rewrite \eqref{E:split-L-1} to get
\begin{equation}
\label{E:split-L-2}
\cL^{i, j}(p; q) = \cL(p; x_i, -\ga) + X_i + \cL(x_i, 0; y_j, 1) + Y_j + \cL(y_j, 1+ \ga; q).
\end{equation}
Now, observe that by the H\"older continuity condition in $4$ and the shape conditions for $\cS^1_\ga, \cS^2_\ga$ imposed by $C(\ga, \ep', \be'')$, as long as $\ep'$ is set small enough given  $\be'$ we have 
\begin{equation}
X_i \in [a_i + \be'' - 1, a_i + \be'' + 2], \qquad Y_j \in [b_j + \be'' - 1, b_j + \be'' + 2],
\end{equation}
and so from the definition of $\ba, \bb$ we have
\begin{equation}
\label{E:XiYj-deff}
X_i + \cL(x_i, 0; y_j, 1) + Y_j \in (-2 + 2 \be'', 4 + 2 \be'')
\end{equation}
whenever $(i, j) \in \operatorname{Adm}(G')$ and 
\begin{equation}
\label{E:XiYj-bad}
X_i + \cL(x_i, 0; y_j, 1) + Y_j < -900 + 2 \be''
\end{equation}
otherwise. Our next aim is to appeal to Lemma \ref{L:face-to-linear-relations}. Let $p_1, \dots, p_k, q_1, \dots, q_\ell$ denote the source and sink vertices in $G'$. We can define a weight map $Z:E(G') \to \R$  so that
$$
d_Z(p_i, q_j) = X_i + \cL(x_i, 0; y_j, 1) + Y_j 
$$
for all $(i, j) \in \operatorname{Adm}(G')$. Indeed, we can first weight all edges according their lengths in the embedding of $G'$ as the set $\Ga(\bx, \by, G')$, and then add $X_i$ to the weight of the edge emanating from $p_i$ and $Y_j$ to the weight of the edge coming into $q_j$ for all $i, j$. By Lemma \ref{L:face-to-linear-relations} and Corollary \ref{C:dimension-interior-forests}, the set $W(Z)$ defined in that lemma has dimension $(k + \ell - |V|)/2 + 2$. Moreover, by \eqref{E:XiYj-deff}, $W(Z) \cap (-100,100)^{k+\ell} \ne \emptyset$ so by point $2$ we have
\begin{align*}
&\dim_{1:2:3}[T(Z)] \ge d(G), \quad \text{ where } \quad T(Z) := \operatorname{Star}(\bx, -\ga; \by, 1 + \ga; g(W(Z))) \cap K_\be.
\end{align*}
For $(p; q) \in T(Z)$, by \eqref{E:split-L-2} all of the values $\cL^i_j(p; q)$ are equal to a common value $L(p; q)$ for $(p; q) \in \operatorname{Adm}(G')$. Moreover, by \eqref{E:XiYj-bad} we have $\cL^i_j(p; q) < L(p; q)$ for $(i, j) \notin \operatorname{Adm}(G')$. Hence by \eqref{E:max-over-ij}, for all $(p; q) \in T(Z)$ we have
$$
\cL(p; q) = L[p, q] = \cL^i_j(p; q)
$$
for all $(i, j) \in \operatorname{Adm}(G')$. Now, by \eqref{E:piij-conditions}, the geodesic network $\Ga(p; q)$ contains as a subset
$$
\bigcup_{(i, j) \in \operatorname{Adm}(G')} \Ga(\bar \pi^{i, j}(\de), \bar \pi^{i, j}(1- \de)).
$$
By point (II), this set is a directed embedding of $G'$. Next, by \eqref{E:piij-conditions} $\Ga(p; q)$ contains all of the points
$$
\bar\tau^i(-\ga -\ep), i \in \II{1, k}, \qquad \bar \sig^j(1 + \ga + \ep), j \in \II{1, \ell}.
$$ 
Since there are geodesic $k$- and $\ell$-star from $p$ to $(\bx, -\ga)$ and from $(\by, 1 + \ga)$ to $q$, respectively, point $3$ implies that there are geodesic $k$- and $\ell$-star from $p$ to $\bar\tau(-\ga -\ep)$ and from $\bar \sig(1 + \ga + \ep)$ to $q$. These geodesic stars are also contained in $\Ga(p; q)$. Hence $\Ga(p; q)$ contains a directed embedding of $G$, so $G$ is a subgraph of $G(p; q)$. In summary, we have shown that with positive probability:
\begin{equation}
\label{E:dim123-almost!}
\dim_{1:2:3} \bigcup_{H \in \sG: G \sset H} (N_\cL(H)) \ge d(G), \qquad \bigcup_{H \in \sG: G \sset H} (N_\cL(H)) \ne \emptyset.
\end{equation}
Now, if $G$ is such that $d(G) = 0$, then the only graph $H \in \sG$ containing $G$ as a subgraph is $G$ itself, so \eqref{E:dim123-almost!} implies \eqref{E:lower-want}. If $d(G) > 0$, then for any graph $H \ne G$ in $\sG$ with $G \sset H$ we have $d(G) < d(H)$. Therefore combining \eqref{E:dim123-almost!} with the upper bound in Proposition \ref{P:upper-bd} yields \eqref{E:lower-want}.
\end{proof}

\section{Dense geodesic networks}
\label{S:dense}

In this section, we prove Theorem \ref{T:Hausdorff-dimension}.3, which classifies the dense and nowhere dense networks in $\cL$. It is easier to show that if a graph $G \in \sG$ cannot be disconnected by removing a single edge, then $N_\cL(G)$ is nowhere dense. This follows from the next lemma.

\begin{lemma}
\label{L:neck-lemma}
Suppose that $G = (V, E) \in \sG$ and that there is no edge $e \in E$ such that the graph $G' = (V, E \smin \{e\})$ is disconnected (as an undirected graph). Then almost surely, the closure $\close{N_\cL(G)}$ of $N_\cL(G)$ satisfies $\dim_{1:2:3}(\close{N_\cL(G)}) \le 7$.
\end{lemma}

\begin{proof}
Suppose that $G$ cannot be disconnected along a single edge. Then we claim there are two edge-disjoint directed paths in $G$ connecting the source and sink vertices $p$ and $q$. Indeed, consider any directed embedding $\Phi(G)$ of $G$, and consider the leftmost and rightmost paths $\pi, \tau$ from $p$ to $q$ in this embedding. If these paths meet at a point away from the vertices $p, q$ then they must meet on an edge since all interior vertices in $G$ have degree $3$. Removing this edge would necessarily disconnect $G$, so the paths $\pi, \tau$ must be disjoint away from $p, q$, as desired.

Therefore if $(p_n, q_n) \in N_\cL(G)$ is a sequence converging to some $(p; q)$ then there are at least two disjoint geodesics $\pi^n, \tau^n$ from $p_n$ to $q_n$ for all $n$. Each of the sequences $\pi^n, \tau^n$ is precompact in the overlap topology on geodesics by Proposition \ref{P:overlap-metric}, so there is a subsequential overlap limit $(\pi, \tau)$ of $(\pi^n, \tau^n)$ and both $\pi, \tau$ are geodesics from $p$ to $q$. Since overlap convergence preserves disjointness, $\pi$ and $\tau$ are disjoint. Therefore $(p; q) \in N_\cL(G')$ for some $G' \in \sG$ with source and sink degree at least $2$. By Theorem \ref{T:Hausdorff-dimension}, $\dim(N_\cL(G')) \le 7$, yielding the lemma.
\end{proof}

Now, for every pair $k, \ell \in \II{1, 3}$, it is not hard to check that there is exactly one graph $G \in \sG$ that can be disconnected by a single edge. This follows from the full enumeration in Section \ref{S:botany}, though it is also an easy exercise for the reader. We call this graph $G_{k, \ell}$. All six choices of $G_{k, \ell}$ with $k \le \ell$ are enumerated in Figure \ref{fig:normal-networks}. The goal of the lemmas in the remainder of this section is to show that each of these networks is dense.  For these lemmas, we write $B_\ep(u)$ for the Euclidean ball of radius $\ep > 0$ centered at $u \in \R^n$.

\begin{lemma}
	\label{L:one-sided}
	Let $(p; q) \in \Rd$ be any fixed point, and let $\ep > 0$. Then almost surely, there are points $p_1, p_2, p_3 \in B_\ep(p)$ such that $G(p_k, q) = G^{k, 1}$ for $k = 1, 2, 3$.
\end{lemma}

To prove Lemma \ref{L:one-sided}, we need a simple fact about Brownian motions.

\begin{lemma}
\label{L:Brownian-lemma}
Let $W:[-1, 1] \to \R$ be any function whose law is absolutely continuous with respect to that of a $2$-sided Brownian motion on $[-1, 1]$. For every $\de > 0$, define $W^\de:[-\de^{-2}, \de^{-2}] \to \R$ by $W^\de(y) = \de^{-1}B(\de^2 y)$, and for $\al > 0$ let $T_\de(W^\de, \al)$ be the event where
\begin{equation}
\label{E:W-de}
|x|^{2/3} > W_\de(x) + \al, \qquad x \in [-\de^{-2},- \al^{-1}]\cup [\al^{-1}, \de^{-2}].
\end{equation}
Then for any fixed sequence $\de_n \to 0$ and any fixed $\al > 0$, almost surely we have
$$
\liminf_{n \to \infty} \frac{1}{n} \sum_{i=1}^n \indic(T_{\de_n}(W^{\de_n}, \al)) > 0.
$$
\end{lemma}

\begin{proof}
First, since the claim in the lemma is an almost sure statement, it suffices to prove it when $W$ is a Brownian motion. Extend $W$ to a $2$-sided Brownian motion on all of $\R$, and let $\tilde T(W^\de, \al)$ be the event where \eqref{E:W-de} holds for all $x \in \R, |x| \ge \al^{-1}$. Now, Brownian motion is strong mixing with respect to the  Brownian scaling $W \mapsto W^\de$, so
$$
\P(\tilde T(W, \al)) = \liminf_{n \to \infty} \frac{1}{n} \sum_{i=1}^n \indic(\tilde T(W^{\de_n}, \al)) \le \liminf_{n \to \infty} \frac{1}{n} \sum_{i=1}^n \indic(T_{\de_n}(W^{\de_n}, \al)).
$$
Finally, it is easy to check that $\P(\tilde T(W, \al)) > 0$.
\end{proof}

\begin{proof}[Proof of Lemma \ref{L:one-sided}]
	Fix $k \in \II{1, 3}$. By landscape rescaling, we may assume $(p; q) = (0,0; 0,1)$. Also, almost surely, no rational points are $2$-stars by Theorem \ref{T:star-computation}. We assume this during the proof. Define
$$
\cL^\de(x, s; y, t) = \de^{-1} \cL(\de^2 x, \de^3 s; \de^2 y, \de^3 t), \qquad \fA^\de(y) = \de^{-1} [\cL(\de^2 y, 0; 0, 1) - \cL(0, 0; 0, 1)].
$$
For $s < 0$ and $x \in \R$, suppose there is a unique argmax $M_{x, s}$ of the function
$$
f_{x, s}(y) := \cL^\de(x, s; y,0) + \fA^\de(y).
$$
Then all $\cL$-geodesics from $(\de^2 x, \de^3 s)$ to $(0, 1)$ go through the point $(\de^2 M_{x, s}, 0)$. Since $(\de^2 M_{x, s}, 0)$ is an interior geodesic point and $(0, 1)$ is not a $2$-star, there is a unique geodesic from $(\de^2 M_{x, s}, 0)$ to $(0, 1)$ since otherwise we could find a geodesic bubble, contradicting Lemma \ref{L:rightmost-geods}.3. Therefore if 
$$
G_{\cL^\de}(x, s; M_{x, s}, 0) = G_\cL(\de^2 x, \de^3 s; \de^2 M_{x, s}, 0) = G^{k, 1}
$$
then $G_\cL(\de^2 x, \de^3 s; 0, 1) = G^{k, 1}$ as well. Here $G_{\cL^\de}, G_\cL$ indicate network graphs in $\cL^\de, \cL$.
Now, fix $\de > 0$. Since each $\cL^\de$ is a directed landscape, with positive probability the following two events occur simultaneously. Here $c, c', \ga, \ga' > 0$ are absolute constants.
\begin{enumerate}[label=\arabic*.]
	\item There exists a (random) $p_\de = (x_\de, s_\de) \in [-c, c] \X [-c, -\ga]$ such that $(p_\de; 0,0) \in N_{\cL^\de}(G^{k, 1})$. Moreover, there exists $y_\de \in \R$ such that if $\pi$ is any geodesic from $p_\de$ to a point in $(-\ga', \ga') \X \{0\}$, then $\pi(-\ga) = y_\de$.
	
	\item For all $y \in \R$, we have
$$
	\left|\cL^\de(p_\de; y, 0) + \frac{(x - y)^2}{|s_\de|} \right|\le c' \log^2(2 + |y|).
	$$
\end{enumerate}
To see why point $1$ above holds with positive probability for large $c$ and small $\ga$, observe that $N_{\cL^\de}(G^{k, 1})$ is almost surely non-empty (Theorem \ref{T:geodesic-networks}) and that if $u \in N_{\cL^\de}(G^{k, 1})$ then necessarily $v \in N_{\cL^\de}(G^{k, 1})$ for all rational $v$ in a small ball around $u$ (this follows from Proposition \ref{P:overlap-metric} and the fact that no rational points are $2$-stars). Together with the symmetries of $\cL$, these points imply that $N_{\cL^\de}(G^{k, 1})$ contains a point of the form $(p; 0,0)$ with $p \in [-c, c] \X [-c, -\ga]$ with positive probability, and that all geodesics from $p$ to $(0,0)$ coincide at time $-\ga$ at some location $y_0$. 
A second application of Proposition \ref{P:overlap-metric}.2 implies that for small enough $\ga' > 0$, any geodesic $\pi$ from $p$ to a point in $(-\ga', \ga') \X \{0\}$ also satisfies $\pi(-\ga) = y_0$. The second point holds with positive probability given the first point by the landscape shape theorem (Theorem \ref{T:landscape-shape}) as long as $c'$ is sufficiently large.

Now, let $S_\de$ be the event where the two points above hold. Since $\cL$ is ergodic with respect to $1:2:3$ rescaling (Remark \ref{R:ergodicity}) we can find a random sequence $\de_n \to 0$ such that $S_{\de_n}$ holds for all $n \in \N$. Moreover, by Theorem \ref{T:landscape-shape}, we can find $C > 0$ such that 
\begin{equation*}
\label{E:A-basic}
|\fA^1(y) + y^2| \le C \log^2(2 + |y|)
\end{equation*}
for all $y \in \R$. Therefore by point $2$ above, for all large enough $n$, any argmax of the function $f_{p_{\de_n}}$ is attained on the set $[-\de^{-2}, \de^{-2}]$. Moreover, by Lemma \ref{L:BG-property} the law of the function $\fA^1|_{[-1, 1]} - \fA^1(0)$ is absolutely continuous with respect to a $2$-sided Brownian motion. By independence of landscape time increments, $\fA^1$ is also independent of the choice of the sequence $\de_n$, and all the functions $y \mapsto \cL^{\de_n}(p_{\de_n}; y, 0)$. Therefore by Lemma \ref{L:Brownian-lemma}, almost surely for all $\al > 0$, there are infinitely many values of $n \in \N$ for which $T_{\de_n}(\fA^{\de_n}, \al)$ holds. Combining this with point $2$ above, we find that there is a subsequence $Y \sset \N$ such that the argmax of the function $f_{p_{\de_n}}$ is attained on the set $(-\ga, \ga)$ for all $n \in Y$. 

Next, since $\fA^{\de_n}$ is absolutely continuous with respect to a $2$-sided Brownian motion on $[-\ga, \ga]$, and is independent of the function $y \mapsto \cL^{\de_n}(p_{\de_n}; y, 0)$, almost surely the function $f_{p_{\de_n}}$ has a unique argmax on $[-\ga, \ga]$ for all $n$. Therefore for all $n \in Y$, $f_{p_{\de_n}}$ has a unique global argmax $M_{p_{\de_n}}$, attained in $(-\ga, \ga)$. Now, by the second part of point $1$ for the event $S_{\de_n}$, a path is a geodesic from $p_{\de_n}$ to $(M_{p_{\de_n}}, 0)$ if and only if it can be written as a concatenation $\pi|_{[s_\de, -\ga]} \oplus \tau$, where $\pi$ is a geodesic from $p_{\de_n}$ to $(y_{\de_n}, -\ga)$ and $\tau$ is a geodesic from $(y_{\de_n}, -\ga)$ to $(M_{p_{\de_n}}, 0)$. The choice of $\tau$ is unique by Lemma \ref{L:rightmost-geods}.3, and so
$$
G_{\cL^{\de_n}}(p_{\de_n}; M_{p_{\de_n}}, 0) = G_{\cL^{\de_n}}(p_{\de_n}; y_{\de_n}, -\ga) = G_{\cL^{\de_n}}(p_{\de_n}; 0,0) = G^{k, 1}.
$$
Hence $G_\cL(\de^2 x_{\de_n}, \de^3 s_{\de_n}; 0, 1) = G^{k, 1}$ for all $n \in Y$, yielding the result.
\end{proof}

\begin{lemma}
	\label{L:basic-gluing}
Let $u = (p; q) \in \Rd$ be any fixed point, and let $\ep> 0$. Then almost surely, there are points $u_{k, \ell} \in B_\ep(u)$ for $k \le \ell \in \II{1, 3}$ such that $G(u_{k, \ell}) = G^{k, \ell}$ for $i = 1, 2, 3$.
\end{lemma}

Theorem \ref{T:Hausdorff-dimension}.3 follows from combining Lemma \ref{L:neck-lemma} and Lemma \ref{L:basic-gluing}.
\begin{proof}
By landscape rescaling, it suffices to prove the lemma with $p, q = (0,0; 0, 1)$. Almost surely, there is a unique geodesic $\pi$ from $p$ to $q$. Now, by Proposition \ref{P:overlap-metric}.2, for all small enough $\de > 0$, the following three statements hold:
\begin{enumerate}[nosep, label=\arabic*.]
	\item If $(p', q') \in B_\de(u)$ and $\tau$ is a geodesic from $p'$ to $q'$, then $[1/4, 3/4] \sset O(\pi, \tau)$.
	\item If $(p', q') \in B_\de((p, \bar \pi(1/2)))$ and $\tau$ is a geodesic from $p'$ to $q'$, then $1/4 \in O(\pi, \tau)$.
		\item If $(p', q') \in B_\de((\bar \pi(1/2), q))$ and $\tau$ is a geodesic from $p'$ to $q'$, then $3/4 \in O(\pi, \tau)$.
\end{enumerate} 
Now, by Lemma \ref{L:one-sided} almost surely for every rational point $w \in B_{\de/2}(\bar \pi(1/2))$ we can find points $(p_k, q_\ell) \in B_{\de/2}(u) \cap B_\ep(u)$ such that $G(p_k, w) = G^{k, 1}, G(w, q_\ell) = G^{1, \ell}$. By point $2$ above, all geodesics from $p_k$ to $w$ must merge in the time interval $(-\infty, 1/4]$ and must equal $\pi(1/4)$ at time $1/4$. Therefore $G(p_k, \bar \pi(1/4)) = G^{k, 1}$. Similarly, point $3$ yields that $G(\bar \pi(3/4), q_\ell) = G^{1, \ell}$. By point $1$, geodesics from $p_k$ to $q_\ell$ are exactly paths of the form $\tau_p \oplus \pi|_{[1/4, 3/4]} \oplus \tau_q$ where $\tau_p$ is a geodesic from $p_k$ to $\bar \pi(1/4)$ and $\tau_q$ is a geodesic $\bar \pi(3/4)$ to $q_\ell$. Therefore $G(p_k, q_\ell) = G^{k, \ell}$, as desired.
\end{proof}

\section{Geodesic Botany}
\label{S:botany}

In this final section, we enumerate the networks in $\sG$. The proof of the enumeration is straightforward but somewhat tedious. The reader may find it more useful and interesting to simply peruse the figures of each of the $27$ networks contained at the end of the section.

\begin{prop}
	\label{P:enumeration}
We have $|\sG| = 27$. Moreover, letting $\sG_{k, \ell} \sset \sG$ 
	be the set of networks with $\{\deg(p), \deg(q)\} = \{k, \ell\}$, we have
	\begin{equation*}
	|\sG_{1, 1}| = |\sG_{1, 2}| = |\sG_{1, 3}| = 1, \qquad |\sG_{2, 2}| = 3, \qquad |\sG_{2, 3}| = 6, \qquad |\sG_{3, 3}| = 15.
	\end{equation*}
\end{prop}

\begin{proof}
	Throughout the proof, we may assume $\deg(p) = k \le  \deg(q) = \ell$ since any $G \in \sG$ is transpose-isomorphic to a graph with these properties.
	We will also repeatedly appeal to the Lemma \ref{L:graph-theory} bound:
	\begin{equation}
	\label{E:pq-formula}
	2 + |k - \ell| \le |V| \le k + \ell.
	\end{equation} 
	\textbf{Case 1. $\sG_{1, \ell}$ for $\ell = 1, 2, 3$: one graph for each $\ell$.} \qquad By \eqref{E:pq-formula}, any $G \in \sG_{1, \ell}$ has $1 + \ell$ vertices. Therefore:
	\begin{itemize}[nosep]
		\item Any $G \in \sG_{1, 1}$ only has the two vertices $p$ and $q$, and so the only element of $\sG_{1,1}$ is the path from $p$ to $q$, see Figure \ref{fig:G1x}(a).
		\item Any $G \in \sG_{1, 2}$ has one interior vertex $v$ in addition to $p, q$. The only way to obtain a graph in $\sG$ on these vertices with $k = 1, \ell = 2$ is to connect $v$ once to $p$ and twice to $q$, see Figure \ref{fig:G1x}(b).
		\item Any $G \in \sG_{1, 3}$ has $4$ vertices $\{p, q, v_1, v_2\}$. Either $(p, v_1)$ or $(p, v_2)$ must lie in $E$ since otherwise either $v_1$ or $v_2$ is a source. Up to isomorphism we may assume that $(p, v_1) \in E$. Note that no other edges connect to $p$. Then $(v_1, v_2) \in E$ since otherwise $v_2$ would be a source. The edge $(v_1, v_2)$ appears only once and $(v_2, v_1)$ never appears, since otherwise $G \smin\{p, q\}$ would have a cycle. Therefore the third edge containing $v_1$ is $(v_1, q)$. We have examined all possible edges except $(v_2, q)$; this edge must appear twice resulting in a unique graph $G$, see Figure \ref{fig:G1x}(c).
	\end{itemize}  
	\textbf{Case 2. $\sG_{2, 2}$: $3$ total graphs.} \qquad
	By \eqref{E:pq-formula}, every graph in $\sG_{2, 2}$ has $2, 3,$ or $4$ vertices. There is only \textbf{one} graph when $|V| = 2$, where the edge $(p, q)$ appears twice, see Figure \ref{fig:G22}(a). There are no graphs on $3$ vertices since then the total degree of $G$ would be odd ($2 + 2 + 3 = 7$). Now suppose that $G = \{p, q, v_1, v_2\} \in \sG_{2, 2}$. Since the set $\{v_1, v_2\}$ has at most four edges connecting it to $\{p, q\}$ but $\deg(v_1) + \deg(v_2) = 6$, the vertices $v_1, v_2$ must be connected by an edge. The remaining four edges must go between the set $\{v_1, v_2\}$ and the set $\{p, q\}$, and these remaining four edges must hit each of these vertices twice. Up to isomorphism, there are \textbf{two} possibilities: each of the edges $(p, v_1), (v_1, q), (p, v_2), (v_2, q)$ are included once or the edges $(p, v_1)$ and $(v_2, q)$ are each included twice, see Figure \ref{fig:G22}(b, c). This gives a total of $1+2 = 3$ graphs in $\sG_{2, 2}$.
	
	\textbf{Case 3: $\sG_{2, 3}$.} \qquad By \eqref{E:pq-formula}, every graph in $\sG_{2, 3}$ has $3, 4,$ or $5$ vertices. There is only \textbf{one} graph on three vertices: where we take the unique graph $G \in \sG_{1, 2}$ and add an extra edge $(p, q)$, see Figure \ref{fig:G23}(a). There are no graphs on four vertices since in this case the sum of degrees would be odd. Now suppose that $G = \{p, q, v_1, v_2, v_3\} \in \sG_{2, 3}$. Since there are at most $2 + 3 = 5$ edges between the sets $\{v_1, v_2, v_3\}$ and $\{p, q\}$ but $\deg(v_1) + \deg(v_2) + \deg(v_3) = 9$, there must be at least two edges in the set $\{v_1, v_2, v_3\}^2$. Note also that there cannot be more than two edges, as this would create an undirected cycle. Therefore there are exactly two edges contained in $\{v_1, v_2, v_3\}^2$, exactly five edges between the sets $\{p, q\}$ and $\{v_1, v_2, v_3\}$, and $(p, q)$ is not an edge.
	Up to relabeling the vertices, there are three possibilities for the pair of edges in $\{v_1, v_2, v_3\}^2$: 
	$$
	\{(v_1, v_2), (v_1, v_3)\}, \{(v_1, v_2), (v_2, v_3)\}, \quad \mathand \quad \{(v_1, v_2), (v_3, v_2)\}.
	$$
	In the $\{(v_1, v_2), (v_1, v_3)\}$ case, we must include edges $(p, v_1), (v_2, q), (v_3, q)$ since otherwise either $v_1$ is a source or $v_2$ or $v_3$ is a sink. The remaining two edges must have one connection to $v_2$ and one connection to $v_3$. Since $\deg(p) = 2$, we can either include the edges $(p, v_2)$ and $(v_3, q)$ again, or include $(p, v_3)$ and $(v_2, q)$ again. These choices are isomorphic and hence yield \textbf{one} graph, see Figure \ref{fig:G23}(b).
	
	In the $\{(v_1, v_2), (v_2, v_3)\}$ case, we must include edges $(p, v_1)$ and $(v_3, q)$. There are three remaining edges to add to the graph, connecting between the sets $\{v_1, v_2, v_3\}$ and $\{p, q\}$. Since exactly one of these leaves from $p$, and the other two edges go to $q$ this leaves \textbf{three} possible choices, see Figure \ref{fig:G23}(c)-(e).
	
	In the $\{(v_1, v_2), (v_3, v_2)\}$ case, we must include edges $(p, v_1), (p, v_3)$ and $(v_2, q)$. Since $\deg(p) = 2$ and $\deg(q) = 3$, there is only \textbf{one} way to complete the graph: by including the edges $(v_3, q)$ and $(v_1, q)$, see Figure \ref{fig:G23}(f).
	Summarizing the whole analysis of $\sG_{2, 3}$ gives a total of $1 + 1 + 3 + 1 = 6$ graphs.
	
	\textbf{Case 4a: $\sG_{3, 3}$ with $|V| < 6$: $4$ total graphs.} \qquad By \eqref{E:pq-formula}, $|V| \in \II{2, 6}$ if $G \in \sG_{3, 3}$, and $|V|$  is even since otherwise $G$ would have odd total degree. There is \textbf{one} graph when $|V| = 2$ given by including $(p, q)$ $3$ times, see Figure \ref{fig:G33<6}(a). Now consider $|V| = 4$ and let $V = \{p, q, v_1, v_2\}$. Suppose first that $(p, q)$ is an edge. In this case, we can remove this edge to get a graph $G \in \sG_{2, 2}$ with $|V| = 4$: by our argument enumerating $\sG_{2, 2}$ there are \textbf{two} such graphs, see Figure \ref{fig:G33<6}(b, c).
	
	Next suppose that $(p,q)$ is not an edge. In this case, there are six edges between $\{p, q\}$ and $\{v_1, v_2\}$ so $(v_1, v_2)$ is not an edge. Therefore $(p, v_1), (v_1, q), (p, v_2), (q, v_2)$ must all be included at least once. The remaining two edges must hit all vertices exactly once, so the only option is to double up either $(p, v_1)$ and $(v_2, q)$, or  $(p, v_2)$ and $(v_1, q)$. These are isomorphic options, so we only get\textbf{ one} graph, see Figure \ref{fig:G33<6}(d).
	
	\textbf{Case 4b: $\sG_{3, 3}$ with $|V| = 6$: $11$ total graphs.}  \qquad Let $G \in \sG_{3, 3}$ have vertices $\{p, q\} \cup V^\circ$ where $V^\circ = \{v_1, v_2, v_3, v_4\}$. Since $\sum_{i=1}^4 \deg(v_i) = 12$ and there are at most $6$ edges connecting $\{p, q\}$ to $V^\circ$, there are at least $3$ edges in the set $(V^\circ)^2$. Since including any more than three edges in $(V^\circ)^2$ would create a cycle, there must be exactly three edges in this set, three edges in the set $\{p\} \X V^\circ$, and three edges in the set $V^\circ \X \{q\}$. The pair $(p, q)$ is not an edge. 
	
	We classify graphs in $\sG_{3, 3}$ with $|V| = 6$ by focusing first on the induced graph on $V^\circ$. Up to transpose-isomorphism, there are five possibilities:
	\begin{enumerate}[nosep, label=\arabic*.]
		\item Path 1. $(v_1, v_2), (v_2, v_3), (v_3, v_4)$.
		\item Path 2. $(v_1, v_2), (v_2, v_3), (v_4, v_3)$.
		\item Path 3. $(v_1, v_2), (v_3, v_2), (v_3, v_4)$.
		\item Star 1. $(v_1, v_2), (v_1, v_3), (v_4, v_1)$. 
		\item Star 2. $(v_1, v_2), (v_1, v_3), (v_1, v_4)$. 	
	\end{enumerate}
	There are no graphs $G \in \sG$ that extends Star 2, since in this case $v_1$ has three outgoing edges and so cannot accommodate an incoming edge. 
	
	\textbf{Path 1: $5$ Extensions.} In this case, $(p, v_1)$ and $(v_4, q)$ must be edges, since otherwise either $v_1$ would be a source or $v_4$ would be a sink. There are four remaining edges to place, each of which involves exactly one of $\{v_1, v_2, v_3, v_4\}$. Two of these edges go to $q$, and the other $2$ come from $p$, leaving ${4 \choose 2} = 6$ possibilities. Among this set of six possibilities, all are planar and two choices lead to transpose-isomorphic graphs: 
	\begin{itemize}[nosep]
		\item The graph obtained by adding $(p, v_1), (p, v_4), (v_2, q), (v_3, q)$ is isomorphic to the transpose of the graph obtained by adding $(p, v_2), (p, v_3), (v_1, q), (v_4, q)$.
	\end{itemize}
	There are no other transpose-isomorphisms, see Figure \ref{fig:G33-path1} (e.g. each of the $5$ graphs in that figure has a different number of paths from $p$ to $q$).
	
	\textbf{Path 2: $3$ Extensions.} In this case, $(p, v_1), (p, v_4)$ and $(v_3, q)$ must be edges. We have three edges left to add between $\{p, q\}$ and $\{v_1, v_2, v_3, v_4\}$, which must use $v_1, v_2, v_4$ once each, $p$ once and $q$ twice. We have three choices, depending on where we connect $p$, see Figure \ref{fig:G33-path2}. These options are not transpose-isomorphic. To see this, note that the graphs where we connect $p$ to $v_1$ or $v_4$ have a double edge, whereas the graph where we connect $p$ to $v_2$ does not. Also, connecting $p$ to $v_1$ results $6$ paths going from $p$ to $q$, whereas connecting $p$ to $v_4$ results in $5$ paths.
	
	\textbf{Path 3: $2$ Extensions.} In this case, $(p, v_1), (p, v_3), (v_2, q)$ and $(v_4, q)$ are necessary edges. We then need two edges between $\{p, q\}$ and $\{v_1, v_4\}$ that use all vertices exactly once. The two possibilities are $(p, v_1), (v_4, q)$ and $(p, v_4), (v_1, q)$. These are non-isomorphic options: in the second graph there is a path from $p$ to $q$ of length $2$, whereas no such path exists in the first graph.  
	
	\textbf{Star 2: $1$ Extension.} \qquad In this case, $(v_2, q), (v_3, q),$ and $(p, v_4)$ must be edges so that $v_2, v_3$ are not sinks and $v_4$ is not a source. We need to complete the graph with two more edges going out of $p$, one edge entering $q$, and one edge incident to each of $v_2, v_3, v_4$. Without loss of generality, we can therefore include $(p, v_2)$ as an edge since up to this point, the roles of $v_2$ and $v_3$ are interchangeable.
	
	There are then two ways to finish the graph: either with $(p, v_4)$ and $(v_3, q)$, or with $(p, v_3)$ and $(v_4, q)$. The first of these options results in a graph in $\sG_{3, 3}$. The second of these options breaks the planarity constraint: in this case, the underlying undirected graph is the complete bipartite graph $K_{3, 3}$, with $\{p, q, v_1\}$ and $\{v_2, v_3, v_4\}$ being the two parts.
	
	In summary, there are $5 + 3 + 2 + 1 = 11$ total graphs in $\sG_{3, 3}$ with $|V| = 6$.
\end{proof}

Corollary \ref{C:Tpq} follows from the main theorem and examining the figures below.

\begin{figure}[htb]
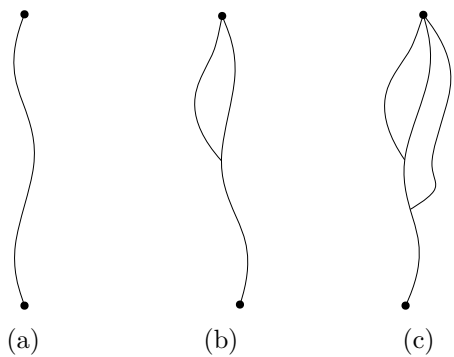

	\centering
	\begin{subfigure}[t]{2.5cm}
		\centering
		\includegraphics[height=4cm]{G11}
		\caption{}
	\end{subfigure}
	\begin{subfigure}[t]{2.5cm}
		\centering
		\includegraphics[height=4cm]{G21}
		\caption{}
	\end{subfigure}
	\begin{subfigure}[t]{2.5cm}
		\centering
		\includegraphics[height=4cm]{G31}
		\caption{}
	\end{subfigure}
	\caption{The three geodesic networks in $\sG_{1, 1}, \sG_{1, 2}$ and $\sG_{1, 3}$. All three of these networks have $N_\cL(G)$ dense. Respectively, networks (a), (b), and (c) have $1,2$ and $3$ geodesics, have 1:2:3-Hausdorff dimensions $10, 8$ and $5$ on $\Rd$.}
	\label{fig:G1x}
\end{figure}

\begin{figure}[htb]
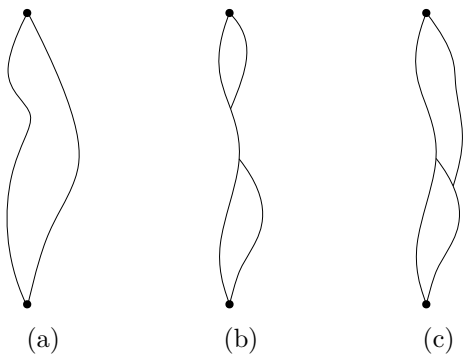

	\centering
	\begin{subfigure}[t]{2.5cm}
		\centering
		\includegraphics[height=4cm]{G22.2}
		\caption{}
	\end{subfigure}
	\begin{subfigure}[t]{2.5cm}
		\centering
		\includegraphics[height=4cm]{G22.4.1}
		\caption{}
	\end{subfigure}
	\begin{subfigure}[t]{2.5cm}
		\centering
		\includegraphics[height=4cm]{G22.4.2}
		\caption{}
	\end{subfigure}
	\caption{The three geodesic networks in $\sG_{2, 2}$. Network (b) is dense and networks (a) and (c) are nowhere dense. Respectively, networks (a), (b), and (c) have $2, 4,$ and $3$ geodesics, have 1:2:3-Hausdorff dimensions $7, 6,$ and $6$ on $\Rd$.}
	\label{fig:G22}
\end{figure}

\begin{figure}[htb]
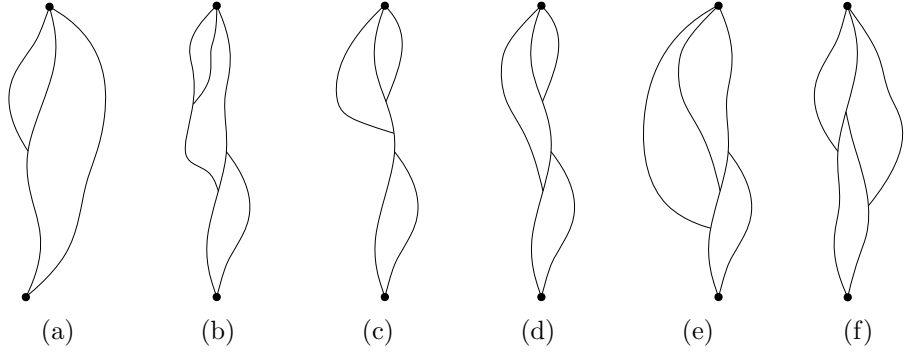

	\centering
	\begin{subfigure}[t]{2cm}
		\centering
		\includegraphics[height=4cm]{G32.3}
		\caption{}
	\end{subfigure}
	\begin{subfigure}[t]{2cm}
		\centering
		\includegraphics[height=4cm]{G32.5.a2}
		\caption{}
	\end{subfigure}
	\begin{subfigure}[t]{2cm}
		\centering
		\includegraphics[height=4cm]{G32.5.b1}
		\caption{}
	\end{subfigure}
	\begin{subfigure}[t]{2cm}
		\centering
		\includegraphics[height=4cm]{G32.5.b2}
		\caption{}
	\end{subfigure}
	\begin{subfigure}[t]{2cm}
		\centering
		\includegraphics[height=4cm]{G32.5.b3}
		\caption{}
	\end{subfigure}
	\begin{subfigure}[t]{2cm}
		\centering
		\includegraphics[height=4cm]{G32.5.a1}
		\caption{}
	\end{subfigure}
	\caption{The six geodesic networks in $\sG_{2, 3}$. All networks are nowhere dense except for (c), which is dense. Networks (a), (b), (c), (d), (e), and (f) appear with 1:2:3-Hausdorff dimensions $4, 3, 3, 3, 3,$ and $3$, and have $3, 4, 6, 5, 4,$ and $4$ geodesics.}
	\label{fig:G23}
\end{figure}

\begin{figure}
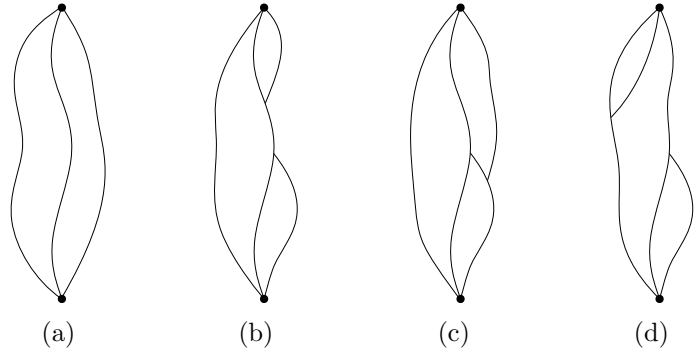

	\centering
	\begin{subfigure}[t]{2.5cm}
		\centering
		\includegraphics[height=4cm]{G33.2}
		\caption{}
	\end{subfigure}
	\begin{subfigure}[t]{2.5cm}
		\centering
		\includegraphics[height=4cm]{G33.4.1}
		\caption{}
	\end{subfigure}
	\begin{subfigure}[t]{2.5cm}
		\centering
		\includegraphics[height=4cm]{G33.4.2}
		\caption{}
	\end{subfigure}
	\begin{subfigure}[t]{2.5cm}
		\centering
		\includegraphics[height=4cm]{G33.4.3}
		\caption{}
	\end{subfigure}
	\caption{The four geodesic networks in $\sG_{3, 3}$ with $|V| < 6$. All networks are nowhere. Networks (a), (b), (c), and (d) appear with 1:2:3-Hausdorff dimensions $2, 1, 1,$ and $1$ and have $3, 5, 4,$ and $4$ geodesics.}
	\label{fig:G33<6}
\end{figure}

\begin{figure}[htb]
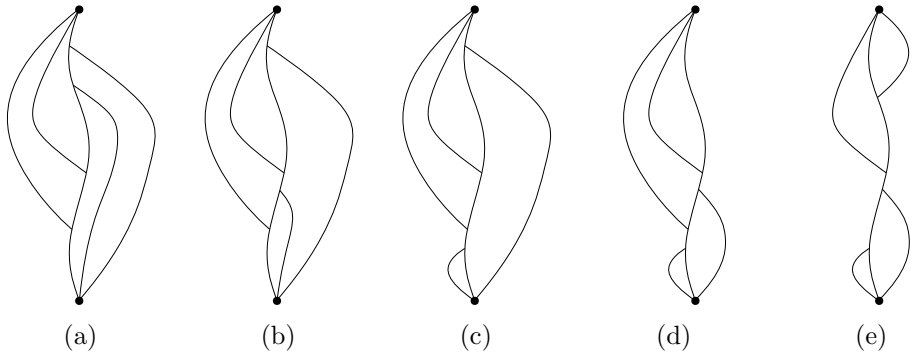

	\centering
	\begin{subfigure}[t]{2.5cm}
		\centering
		\includegraphics[height=4cm]{G33.6.p1.1}
		\caption{}
	\end{subfigure}
	\begin{subfigure}[t]{2.5cm}
		\centering
		\includegraphics[height=4cm]{G33.6.p1.2}
		\caption{}
	\end{subfigure}
	\begin{subfigure}[t]{2.5cm}
		\centering
		\includegraphics[height=4cm]{G33.6.p1.3}
		\caption{}
	\end{subfigure}
	\begin{subfigure}[t]{2.5cm}
		\centering
		\includegraphics[height=4cm]{G33.6.p1.4}
		\caption{}
	\end{subfigure}
	\begin{subfigure}[t]{2.5cm}
		\centering
		\includegraphics[height=4cm]{G33.6.p1.5}
		\caption{}
	\end{subfigure}
	\caption{The five geodesic networks $G \in \sG_{3, 3}$ that extend Path $1$. For (a), (b), (c), and (d), the networks are countable and nowhere dense. Network (e) is countably dense. The networks (a), (b), (c), (d), and (e) have $5, 6, 7, 8,$ and $9$ geodesics.}
	\label{fig:G33-path1}
\end{figure}

\FloatBarrier

\begin{figure}[htb]
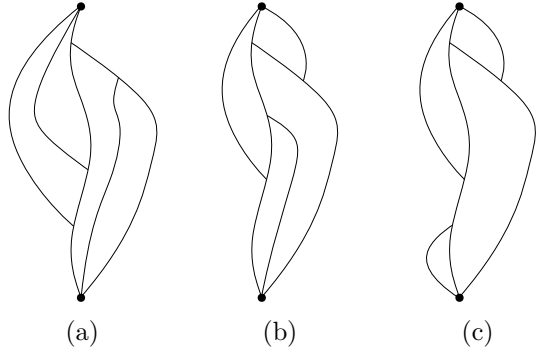

	\centering
	\begin{subfigure}[t]{2.5cm}
		\centering
		\includegraphics[height=4cm]{G33.6.p2.3}
		\caption{}
	\end{subfigure}
	\begin{subfigure}[t]{2.5cm}
		\centering
		\includegraphics[height=4cm]{G33.6.p2.2}
		\caption{}
	\end{subfigure}
	\begin{subfigure}[t]{2.5cm}
		\centering
		\includegraphics[height=4cm]{G33.6.p2.1}
		\caption{}
	\end{subfigure}
	\caption{The three geodesic networks $G \in \sG_{3, 3}$ that extend Path $2$. All networks are countable and nowhere dense. The networks (a), (b), and (c) have $5, 5,$ and $6$ geodesics respectively.}
	\label{fig:G33-path2}
\end{figure}

\begin{figure}[htb]
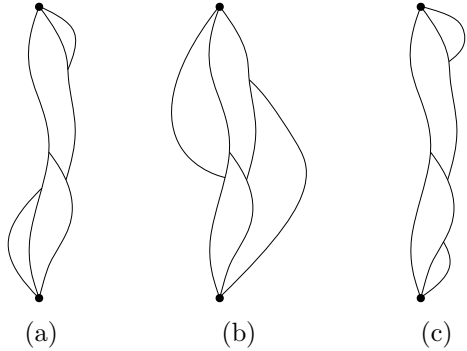

	\centering
	\begin{subfigure}[t]{2.5cm}
		\centering
		\includegraphics[height=4cm]{G33.6.p3.1}
		\caption{}
	\end{subfigure}
	\begin{subfigure}[t]{2.5cm}
		\centering
		\includegraphics[height=4cm]{G33.6.p3.2}
		\caption{}
	\end{subfigure}
	\begin{subfigure}[t]{2.5cm}
		\centering
		\includegraphics[height=4cm]{G33.6.s1}
		\caption{}
	\end{subfigure}
	\caption{The two networks $G \in \sG_{3, 3}$ that extend Path $3$ (networks (a), (b)) and the network that extends Star $2$ (network (c)). All networks are countable and nowhere dense. The networks (a), (b), and (c) have $5, 5,$ and $7$ geodesics respectively.}
	\label{fig:G33-star-path}
\end{figure}

\bibliographystyle{alpha}
\bibliography{MasterBib}

\end{document}